\documentclass[12pt]{amsart}
\usepackage{etex}

\usepackage{amscd,latexsym,amsthm,amsfonts,amssymb,amsmath,amsxtra,mathdots,tikz}
\usepackage[mathscr]{eucal}
\usepackage{xcolor}
\usepackage[normalem]{ulem}
\usepackage{hyperref}
\usepackage[all,cmtip]{xy}
\usepackage{enumitem}
\usepackage{tabularx}
\usepackage{textcomp}
\usepackage{caption}

\renewcommand\theequation{\thesection.\arabic{equation}}

\newcommand{\BA}{{\mathbb {A}}}

\newcommand{\BC}{{\mathbb {C}}}

\newcommand{\BF}{{\mathbb {F}}}

\newcommand{\BH}{{\mathbb {H}}}

\newcommand{\BR}{{\mathbb {R}}}

\newcommand{\BZ}{{\mathbb {Z}}}

\newcommand{\CA}{{\mathcal {A}}}

\newcommand{\CC}{{\mathcal {C}}}

\newcommand{\CI}{{\mathcal {I}}}

\newcommand{\CN}{{\mathcal {N}}}
\newcommand{\CO}{{\mathcal {O}}}
\newcommand{\CP}{{\mathcal {P}}}

\newcommand{\CT}{{\mathcal {T}}}
\newcommand{\CU}{{\mathcal {U}}}

\newcommand{\CY}{{\mathcal {Y}}}

\newcommand{\Fg}{{\mathfrak {g}}}

\newcommand{\Fm}{{\mathfrak {m}}}

\newcommand{\Fp}{{\mathfrak {p}}}

\newcommand{\Fu}{{\mathfrak {u}}}

\newcommand{\Ad}{{\mathrm{Ad}}}
\newcommand{\Aut}{{\mathrm{Aut}}}

\newcommand{\GL}{{\mathrm{GL}}}

\newcommand{\GSp}{{\mathrm{GSp}}}

\newcommand{\GSO}{{\mathrm{GSO}}}
\newcommand{\GSpin}{{\mathrm{GSpin}}}
\newcommand{\Spin}{{\mathrm{Spin}}}

\newcommand{\Hom}{{\mathrm{Hom}}}

\newcommand{\I}{{\mathrm{I}}}

\newcommand{\PGL}{{\mathrm{PGL}}}

\newcommand{\GU}{{\mathrm{GU}}}
\newcommand{\HSpin}{{\mathrm{HSpin}}}

\newcommand{\SO}{{\mathrm{SO}}}

\newcommand{\sgn}{{\mathrm{sgn}}}
\newcommand{\Sp}{{\mathrm{Sp}}}

\newcommand{\tr}{{\mathrm{tr}}}

\newcommand{\ud}{\,\mathrm{d}}

\newcommand{\vol}{{\mathrm{vol}}}

\newcommand{\back}{\backslash}

\def\BA{{\mathbb A}}
\def\BR{{\mathbb R}}
\def\BC{{\mathbb C}}

\def\Fu{{\mathfrak u}}
\def\back{{\backslash}}

\def\bks{{\backslash}}

\def\diag{{\rm diag}}

\def\eps{{\epsilon}}

\def\lam{{\lambda}}

\def\sig{{\sigma}}

\newtheorem{thm}{Theorem}[section]
\newtheorem{cor}[thm]{Corollary}
\newtheorem{lem}[thm]{Lemma}

\newtheorem{prop}[thm]{Proposition}
\newtheorem {conj}[thm]{Conjecture}

\newtheorem {ques/conj}[thm]{Question/Conjecture}

\newtheorem{defn}[thm]{Definition}
\newtheorem{rmk}[thm]{Remark}

\makeatletter

\newcommand{\Rmnum}[1]{\expandafter\@slowromancap\romannumeral #1@}
\makeatother

\begin{document}
\renewcommand{\theequation}{\arabic{equation}}
\numberwithin{equation}{section}

\title[Strongly Tempered Spherical Varieties]{Periods of Automorphic Forms associated to   Strongly Tempered Spherical Varieties}

\author{Chen Wan}
\address{Department of Mathematics \& Computer Science\\
Rutgers University – Newark\\
Newark, NJ 07102, USA}
\email{chen.wan@rutgers.edu}

\author{Lei Zhang}
\address{Department of Mathematics\\
National University of Singapore,
Singapore 119076}
\email{matzhlei@nus.edu.sg}

\date{}



\begin{abstract}
In this paper, we compute the local relative characters for 10 strongly tempered spherical varieties in the unramified case. We also study the local multiplicity for these models. By proving a geometric multiplicity formula, we show that the summation of the multiplicities is always equal to 1  over each local  tempered  Vogan $L$-packet defined on the pure inner forms of the strongly tempered spherical varieties. Finally, we formulate  the Ichino--Ikeda type conjecture  on a relation between the period integrals and the central values of certain automorphic $L$-functions for those strongly tempered spherical varieties.
\end{abstract}
\maketitle

\tableofcontents

\section{Introduction and main results}
Let $k$ be a number field and $\BA$ its ring of adeles. Let $G$ be a reductive group defined over $k$, and $H$ a closed connected subgroup of $G$. We say $(G,H)$ is a spherical pair if $X=H\back G$ is a spherical $G$-variety (i.e., a Borel subgroup of $G$ has a dense orbit in $X$). We assume that $(G,H)$ is a spherical pair for the rest of this paper. We say the spherical pair $(G,H)$ is reductive if $H$ is reductive.
Let $Z_G$ be the center of $G$ and let $Z_{G,H}=Z_G\cap H$. If $(G,H)$ is reductive, for a cuspidal automorphic form $\phi$ on $G(\BA)$ whose central character is trivial on $Z_{G,H}(\BA)$, 
we define the period integral $\CP_{H}(\phi)$ to be \footnote{In general if we allow $\phi$ to have nontrivial central characters, then we can also put some characters on $H$}
$$\CP_{H}(\phi):=\int_{H(k)Z_{G,H}(\BA)\back H(\BA)} \phi(h) \ud h.$$

Besides the reductive cases, one can also study the case when the spherical pair $(G,H)$ is the Whittaker induction of a reductive spherical pair $(G_0,H_0)$ (we refer the reader to Definition \ref{defn:Whittaker-induction} for  the definition of Whittaker induction). 
In this case, we have $H=H_0\ltimes U$ where $U$ is the unipotent radical of $H$ and is also the unipotent radical of a parabolic subgroup of $G$, and the period integral is defined to be  
$$\CP_{H}(\phi):=\int_{H(k)Z_{G,H}(\BA)\back H(\BA)} \phi(h) \xi(h)^{-1} \ud h$$
where $\xi=\Pi_v \xi_v$ is a generic character on $U(k)\backslash U(\BA)$, extended to $H(\BA)$ trivially on the reductive part $H_0(\BA)$. We refer the reader to Definition \ref{defn:generic} for the definition of generic characters.

Let $\pi$ be a cuspidal automorphic representation of $G(\BA)$ whose central character is trivial on $Z_{G,H}(\BA)$. 
One of the most fundamental problems in the relative Langlands program is to establish the relation between  $\CP_{H}|_{\pi}$-the period integral restricted to the space of $\pi$,
and special values of some automorphic $L$-functions  $L(s_0,\pi,\rho_X)$ of $\pi$. 
For example, if $G=\SO_{n+1}\times \SO_n$ and $H=\SO_n$, then $(G,H)$ is the famous Gross--Prasad model defined in \cite{GP1}, \cite{GP2} and its period integrals are related to the central value of the tensor $L$-function $L(1/2,\pi_1\times \pi_2)$ (here $\pi=\pi_1\otimes \pi_2$ is  a cuspidal automorphic representation of $\SO_{n+1}(\BA)\times \SO_n(\BA)$, and the representation $\rho_X$ is the standard tensor product representation of ${}^LG$).
This point of view was most systematically put forward by Sakellaridis \cite{Sa12}, and Sakellaridis-Venkatesh \cite{SV}. As in \cite{SV}, the spherical varieties under the consideration in this paper have no Type $N$ spherical root and are wavefront.
We refer the reader to Sections 2.1 and 3.1 of \cite{SV} for the definitions of wavefront and spherical roots.

In general, in order to find the $L$-functions related to the period integral $\CP_{H}(\phi)$ for $\phi=\otimes_v\phi_v\in \otimes_v\pi_v$, 
one needs to compute the local relative character $I_{H_v}(\phi_v)$ for the spherical pair $(G_v,H_v):=(G(k_v),H(k_v))$ over unramified places  $v\in |k|$. 
If the model $(G,H)$ is strongly tempered (see Section \ref{sec:notation}  for the definition of strongly tempered) or is the Whittaker induction of a strongly tempered pair $(G_0,H_0)$, 
the local relative character $I_{H_v}(\phi_v)$ is defined to be the integration of the matrix coefficients over $H(k_v)$, i.e.
\begin{equation}\label{eq:intro-I}
I_{H_v}(\phi_v)=\int_{Z_{G,H}(k_v)\back H(k_v)} \langle\pi_v(h)\phi_v,\phi_v\rangle\xi_v(h)^{-1}\ud h.	
\end{equation}
Note that if $(G,H)$ is the Whittaker induction of a strongly tempered pair, the integral above needs to be regularized (see Section \ref{sec non-red strategy 1} for details). 
In general, if the model $(G,H)$ is not strongly tempered, the local relative character $I_{H_v}(\phi_v)$ is defined via the Plancherel formula. 
For details, see Section 17.3 of \cite{SV}. 

For each spherical pair $(G,H)$, one expects that the local relative character $I_{H_v}(\phi_v)$ equals the quotient of  some special values of some local $L$-functions $\frac{L(s_0,\pi_v,\rho_X)}{L(1,\pi_v,\Ad)}$ times a product  of certain special values of local zeta functions (denoted by $\Delta_{X_v}$) over all the unramified places. 
For instance, 
 for the orthogonal Gross--Prasad model (which is strongly tempered), the local
relative character was computed by Ichino--Ikeda \cite{II}, which is equal to
$$\frac{L(\frac{1}{2},\pi_{1,v}\times \pi_{2,v})}{L(1,\pi_v,\Ad)}\cdot \Delta_{\SO_{n+1},v}(1).$$
Here for any reductive group $G$ defined over $k$ that is split over an unramified extension, we use $\Delta_{G}(s)=\Pi_{v\in |k|}\Delta_{G,v}(s)$ to denote the $L$-function of the dual $M^{\vee}$ to the motive $M$ associated to $G$ introduced by Gross in \cite{G}.

In \cite{Sa}, Sakellaridis developed a general method to compute the local relative characters  at unramified places under certain conditions. 
He showed that the $L$-function $L(s,\pi,\rho_X)$ is determined by the so-called ``virtual colors" of the spherical variety $X$ and the extra factor $\Delta_{X_v}$ is related to the volume of $X(\CO_v)$ ($\CO_v$ is the ring of integers of $k_v$). 
He also explicitly computed the virtual colors of many spherical varieties and hence the $L$-functions $L(s,\pi,\rho_X)$ (see Page 1379 of \cite{Sa}). 

In this paper, following the method of Sakellaridis, we explicitly compute the local relative characters for all the strongly tempered reductive spherical varieties without Type $N$ spherical root. 
We also compute the local relative characters for 7 non-reductive spherical varieties that are the Whittaker inductions of the trilinear $\GL_2$ model $(\GL_{2}^3,\GL_2)$.  
Our computation shows that the period integrals for these strongly tempered spherical varieties are always related to the central value of some $L$-functions of symplectic type, i.e. $s_0=\frac{1}{2}$ and $\rho_X$ is a self-dual representation of ${}^L(G/Z_{G,H})$ of symplectic type. Moreover, we show that the extra  factors $\Delta_{X_v}$ is equal to $\Delta_{G,v}(1)/\Delta_{H_0/Z_{G,H},v}(1)$ for all the models under  consideration (we would like point out that this is only true in the strongly tempered case). 
Note if $H$ is reductive we just let $H=H_0$ and $U=1$.

In addition, we study the local multiplicities for all the models considered in this paper (except for the $E_7$ case). By proving a geometric multiplicity formula, we show that the summation of the multiplicities is always equal to 1 over each local  tempered  Vogan $L$-packet defined on the pure inner forms of these spherical varieties. 
In other words, our results indicate that all these strongly tempered spherical varieties  enjoy the same local and global properties with the Gan--Gross--Prasad models.

Finally, combining our formulas of the local relative characters and our results for the local multiplicities, we are able to formulate the Ichino--Ikeda type conjectures for these models.

\subsection{The local relative character}
By the classification of split reductive spherical pairs in \cite{BP} (here we say the spherical pair $(G,H)$ is split if both $G$ and $H$ are split), it is easy to show that a split strongly tempered reductive spherical pair is either one of the following 4 cases
\begin{equation}\label{eq:4-models}
(\GL_{n+1}\times \GL_n,\GL_n),\;(\SO_{n+1}\times \SO_n,\SO_n),	
\end{equation}
$$(\GL_4\times \GL_2,\GL_2\times \GL_2),\;(\GSp_6\times \GSp_4,(\GSp_4\times \GSp_2)^0),$$
or it is a split symmetric pair (recall that we say a symmetric pair is split if the real form associated to it is split, e.g. $(\GL_n,\SO_n),\;(\Sp_{2n},\GL_n)$). 
Here $(\GSp_4\times \GSp_2)^0=\{(g,h)\in \GSp_4\times \GSp_2 \mid l(g)=l(h)\}$ where $l$ is the similitude character of $\GSp$. 
We refer the reader to Section \ref{sec:GSp-model} for the explicit description of the embeddings. 
By the classification of spherical root system in \cite{BP}, all the split symmetric pairs have Type $N$ spherical root unless $G$ only has one simple root (i.e. $G$ is of Type $A_1$). 
If $G$ only has one simple root, then split symmetric pair $(G,H)$ is essentially the model $(\PGL_2, \GL_1)$. So we only need to consider the 4 models in \eqref{eq:4-models}.

\begin{rmk}\label{rmk one open borel orbit}
For each model in \eqref{eq:4-models}, we can always modify the groups up to some central elements and some finite isogeny, which will give us some other models with the same root  systems (this will also preserve the strongly tempered property). 
For example, 
the model $(\GSp_6\times \GSp_4,(\GSp_4\times \GSp_2)^0)$ and the model $(\Sp_6\times \Sp_4, \Sp_4\times \Sp_2)$ have the same root systems.
In this paper,   we will always choose the spherical pairs $(G,H)$ so that  over the local field $k_v$, there is only one open Borel orbit in $G(k_v)/H(k_v)$. For example, the model $(\GSp_6\times \GSp_4,(\GSp_4\times \GSp_2)^0)$ we choose indeed has only one open Borel orbit (see Section 3.1) while the model $(\Sp_6\times \Sp_4, \Sp_4\times \Sp_2)$  has $|k_{v}^{\times}/(k_{v}^{\times})^2|$-many open Borel orbits.

We also want to point out that for a fixed root system, we may have more than one models with this root system and such that there is a  unique open Borel orbit over every local field. An easy example would be the models $(\SO_4\times \SO_3,\SO_3)$ and $((\PGL_2)^3,\PGL_2)$. Another example is $(\GL_{n+1}\times \GL_n,\GL_n)$ and $(U_{n+1}\times U_n,U_n)$.
\end{rmk}

The first one $(\GL_{n+1}\times \GL_n,\GL_n)$ is the model for the Rankin-Selberg integral of $\GL_{n+1}\times \GL_n$. There is also an analogue of this model for unitary groups, which is call the unitary Gan--Gross--Prasad model. 
The local relative characters have been computed by R. Neal Harris \cite{H} for both the general linear case and the unitary case. In the general linear case (resp. unitary case), $\rho_X$ is the standard tensor product representation of ${}^LG$ (resp. the standard product representation of base change). The second one $(\SO_{n+1}\times \SO_n,\SO_n)$ is the Gross--Prasad model for special orthogonal groups and the local relative characters have been computed by Ichino and Ikeda \cite{II}. In this case, $\rho_X$ is the standard tensor product representation of ${}^LG$. For these three models, the period integrals are related to the central values of the tensor $L$-functions.

In this paper, we give an explicit formula for the local relative characters over unramified places for the remaining two cases $(\GL_4\times \GL_2,\GL_2\times \GL_2)$ and $(\GSp_6\times \GSp_4,(\GSp_4\times \GSp_2)^0)$, as well as the analogue of the model $(\GL_4\times \GL_2,\GL_2\times \GL_2)$ for unitary groups. We also computed 7 non-reductive cases that are the Whittaker inductions of the trilinear $\GL_2$-model $(\GL_{2}^{3},\GL_2)$ (which is strongly tempered). 

To be specific, we consider the following table where $(G,H)$ is the spherical pair and $\rho_X$ is a representation of the L-group of $G/Z_{G,H}$.

\begin{figure}[h!]\leftskip-3cm
\begin{tabular}{| c | c | c |c| c|}
\hline
\textnumero & $G$ & $H$ &  $\rho_X$ & $\Delta_{X,v}=\Delta_{G,v}(1)/\Delta_{H_0/Z_{G,H},v}(1)$\\
\hline
1 & $\GL_4\times \GL_2$ & $\GL_2\times \GL_2$ &  $(\wedge^2\otimes {\rm std}_{2})\oplus {\rm std}_{4} \oplus {\rm std}_{4}^\vee$&  $\zeta_v(1)\zeta_v(3)\zeta_v(4)$\\
\hline
2 &  $\GU_4\times \GU_2$ & $(\GU_2\times \GU_2)^0$ &  $(\wedge^2\otimes {\rm std}_{2})\oplus {\rm std}_{4} \oplus {\rm std}_{4}^\vee$ & $\ast$ \\
\hline
3 &  $\GSp_6\times \GSp_4$ & $(\GSp_4\times \GSp_2)^0$ &  $\Spin_7 \otimes \Spin_5$ & $\zeta_v(1)^2\zeta_v(4)\zeta_v(6)$\\
\hline
4 &  $\GL_6$ & $\GL_2\ltimes U$ &  $\wedge^3$ & $\zeta_v(1)\zeta_v(3)\zeta_v(4) \zeta_v(5)\zeta_v(6)$ \\
\hline
5 & $\GU_6$ & $\GU_2\ltimes U$ & $\wedge^3$ & $\ast\ast$ \\
\hline
6 & $\GSp_{10}$ & $\GL_2\ltimes U$ &  $\Spin_{11}$ & $\zeta_v(1)\zeta_v(4)\zeta_v(6)\zeta_v(8)\zeta_v(10)$\\
\hline
7 & $\GSp_{6}\times \GL_2$ & $\GL_2\ltimes U$ &  $\Spin_{7}\otimes {\rm std}_2$ & $\zeta_v(1)\zeta_v(2)\zeta_v(4)\zeta_v(6)$\\
\hline
8 & $\GSO_8\times \GL_2$ & $\GL_2\ltimes U$ &  $\HSpin_8\otimes {\rm std}_{2}$ & $\zeta_v(1)^2\zeta_v(2)\zeta_v(4)^2 \zeta_v(6)$\\
\hline
9 & $\GSO_{12}$ & $\GL_2\ltimes U$ & $\HSpin_{12}$ & $\zeta_v(1)\zeta_v(4)\zeta_v(6)^2 \zeta_v(8)\zeta_v(10)$\\
\hline
10 & $E_7$ & $\PGL_2\ltimes U$ & $\omega_7$ & $\zeta_v(6)\zeta_v(8)\zeta_v(10) \zeta_v(12)\zeta_v(14)\zeta_v(18)$ \\
\hline
\end{tabular}
\captionof{table}{}
\label{fig:1}
\end{figure}

Here ${\rm std}_n$ is the standard representation of $\GL_n(\BC)$ and ${\rm std}_n^\vee$ is its dual representation, $\Spin_{2n+1}$ is the Spin representation of $\Spin_{2n+1}(\BC)$, $\HSpin_{2n}$ is a half-Spin representation of $\Spin_{2n}(\BC)$, $\omega_7$ is the 56 dimensional representation of $E_7$, and
$$\ast=\zeta_v(1)^2\zeta_v(4)L(1,\eta_{k_v'/k_v})L(3,\eta_{k_v'/k_v}),$$
$$\ast\ast=\zeta_v(1)\zeta_v(4)\zeta_v(6)L(1,\eta_{k_v'/k_v})L(3,\eta_{k_v'/k_v})L(5,\eta_{k_v'/k_v})$$
where $\eta_{k_v'/k_v}$ is the quadratic character for the quadratic extension $k_v'/k_v$. We refer the reader to Section \ref{sec GU} for more details about the representation $\rho_X$ for Models 2 and 5.
 
\begin{thm}
For all the spherical pairs in Table \ref{fig:1}, assume that all the data are unramified over $v$.
Then
\begin{equation}\label{eq:I-Delta-L}
I_{H_v}(\phi_v)=\frac{\Delta_{G,v}(1)}{\Delta_{H_0/Z_{G,H},v}(1)}  \cdot \frac{L(\frac{1}{2},\pi_v,\rho_X)}{L(1,\pi_v,\Ad)}
\end{equation}
where $\rho_X$ is a self-dual symplectic representation of ${}^L(G/Z_{G,H})$ given in Table \ref{fig:1}.
\end{thm}
\begin{rmk}
In \eqref{eq:intro-I}, we choose the local Haar measure ${\rm d} h$ such that $\vol(H(\CO_v),{\rm d}h)=1$. 
If we replace it by Weil’s canonical measure ${\rm d}_{\rm can}h=\Delta_{H_0/Z_{G,H},v}(1){\rm d} h$ (\cite[Chapter 2]{Weil}), then the constant  $\Delta_{H_0/Z_{G,H},v}(1)$ in the above theorem will disappear. 
\end{rmk}

In Section \ref{sec:strategy}, we will explain our strategies of the proof of this theorem. We will also give the formulas of the Whittaker--Shintani functions of these 10 spherical pairs in Propositions \ref{prop:WS-reductive}
and \ref{prop:WS-non-reductive}.

For the rest of this subsection, we explain how we derive the non-reductive models in Table \ref{fig:1}. Model 4 was introduced by Ginzburg--Rallis in \cite{GR} and Model 5 is an analogue of Model 4 for similitude unitary groups. Model 9 and 10 are inspired by one row of the Magic Triangle introduced by Deligne and Gross in \cite{DG} (which is a generalization of the the Freudenthal’s Magic Square). We recall the following row in the Magic Triangle in \cite[Table 1]{DG},  a series of algebraic groups of type: 
\[
A_1\subset A_{1}^3:=A_{1}\times A_1\times A_1\subset C_3\subset A_5\subset D_6\subset E_7.
\]
In this sequence, we observe the spherical pair of type $(A_{1}^{3}, A_{1})$ corresponding to the trilinear $\GL_2$-model.
And the algebraic groups $G$ of types $A_5$, $D_6$ and $E_7$ have a parabolic subgroup $P=LU$ such that the Levi subgroup $L$ is of type $A_1^3$ and the stabilizer $H_0$ of the generic characters $\xi$ of $U$ is of type $A_1$. This gives us the Whittaker inductions of the trilinear $\GL_2$-model for these 3 groups, which are the Models  4, 9, and 10 respectively. Meanwhile, the group of type $C_3$ does not have a Levi subgroup of type $A_1^3$, but it can be fixed by considering the product $C_3\times A_1$. This explains Model 7.

In addition, these non-reductive models are also related to the degenerated Whittaker models of smooth admissible representations (we refer the reader to \cite{GZ} for more details.)
For instance, consider the degenerated Whittaker model $Wh_\xi(\pi)$ of an irreducible representation $\pi$ of $\GSO_{12}$  with respect to $(U,\xi)$ in  Model 9.
Here $(U,\xi)$ is arisen from a nilpotent orbit of partition $[6,6]$ in the Lie algebra of $\GSO_{12}$ 
and $Wh_\xi(\pi)$ is considered as an $H_0$-module in sense of \cite{GZ}.
(Note that the partition $[6,6]$ is used to label two distinct stable nilpotent orbits. However, the corresponding models have no essential differences as explained in Section \ref{sec:GSO12}.)
The distinguished problem in Model 9 is equivalent to determine when the trivial representation of $H_0$ is a quotient representation in $Wh_\xi(\pi)$.
By using the theta correspondence, Gomez and Zhu in \cite{GZ} showed that the $H_0$-module $Wh_\xi(\pi)$ is isomorphic to the degenerated Whittaker model of certain representations of $\GSp_{10}$ as an $H_0$-module,
arisen from the nilpotent orbit of the partition $[5,5]$ in the Lie algebra of $\GSp_{10}$.
Hence, following \cite{GZ}, Model 6 and Model 9 are directly bridged by the theta correspondence. Similarly, Model 7 and Model 8 are also bridged by the theta correspondence.

Finally, Model 8 can be viewed as a reduced model of Model 9. To be specific, we can choose a parabolic subgroup of $\GSO_{12}$ in Model 9 whose Levi subgroup is isomorphic to $\GSO_8\times \GL_2$ such that the intersection of the Levi subgroup with the subgroup $H$ of $\GSO_{12}$ in Model 9 is exactly the subgroup $H$ of  $\GSO_8\times \GL_2$ in Model 8. Under this point of view, we can also view Model 7 as a reduced model of Model 6, view Model 4 as a reduced model of Model 9 and view Model 9 as a reduced model of Model 10. This explains all the non-reductive models in Table \ref{fig:1}.
We summarize the relations among these models in the following diagram:

\[
\xymatrix{
&   (\GL_{6},\GL_2\rtimes U)\ar[r]_{\text{outer form}} \ar@{<-}[d]^{\text{reduced}}&(\GU_{6},\GU_2\rtimes U)\\
(E_7,\PGL_2\rtimes U)\ar[r]_{\text{reduced}}&(\GSO_{12},\GL_2\rtimes U)\ar[r]_{\text{reduced}}\ar@{<->}[d]^{\text{$\theta$-correspondence}}&(\GSO_{8}\times \GL_2,\GL_2\rtimes U)\ar@{<->}[d]^{\text{$\theta$-correspondence}}&\\
&(\GSp_{10},\GL_2\rtimes U)\ar[r]_{\text{reduced}}&(\GSp_{6}\times \GL_2,\GL_2\rtimes U)
}
\]

\begin{rmk}\label{rmk non-red all cases}
Besides the 7 non-reductive cases in the table above, there are another three more non-reductive spherical pairs that are the Whittaker induction of strongly tempered reductive spherical pairs without Type N spherical root:
\begin{enumerate}
\item The Whittaker models for quasi-split reductive groups.
\item The non-reductive Gan--Gross--Prasad models for the general linear groups, the unitary groups, or the orthogonal groups. 
They are the Whittaker inductions of the reductive  Gan--Gross--Prasad models.
\item The model $(\GSO_{10},(\GL_2\times \GL_1)\ltimes U)$ introduced by Ginzburg \cite{Gi} in his study of the Spin L-function of $\GSO_{10}$. This is the Whittaker induction of the model $(\GL_3\times \GL_2,\GL_2)$.
\end{enumerate}
The local relative characters of the Whittaker models have been computed by Lapid-Mao in \cite{LM} and the local relative characters of the non-reductive Gan--Gross--Prasad models have been computed by Liu in \cite{L}. The period integral of the model $(\GSO_{10},(\GL_2\times \GL_1)\ltimes U)$ has been studied by Ginzburg in \cite{Gi} and its local relative character can be computed by the same method as in this paper. 
The local relative characters over unramified places for these models are also of the form \eqref{eq:I-Delta-L} as our models in Table \ref{fig:1}.
The representation $\rho_X$ is the tensor representation for the non-reductive Gan--Gross--Prasad models, and the Spin representation of $\GSpin_{10}(\BC)$ for the model $(\GSO_{10},(\GL_2\times \GL_1)\ltimes U)$.
For the Whittaker model, the numerator $L$-function $L(\frac{1}{2},\pi,\rho_X)$ is just 1.

In general, by a tedious case by case argument (i.e. we checked all the parabolic subgroups of all the reductive groups) which we will not include in this paper, we believe that any spherical pair that are the Whittaker induction of a strongly tempered spherical pair without Type $N$ spherical root must be one of the 10 cases above (7 in Table \ref{fig:1} and 3 in this remark). 
Hence the local relative character of a spherical pair that is either strongly tempered or the Whittaker induction of a strongly tempered spherical pair should always be the form \eqref{eq:I-Delta-L} over unramified places.
\end{rmk}

\subsection{The local multiplicity}
Let $(G,H)$ be one of the models in Table \ref{fig:1}. 
If $H$ is reductive, take $\chi$  to be the trivial character of $H(k_v)$; 
if $H=H_0\ltimes U$ is non-reductive, take $\chi$ to be the character $1\otimes \xi_v$ of $H(k_v)=H_0(k_v)\ltimes U(k_v)$ where $\xi_v$ is the generic character of $U(k_v)$. 
Let $\pi_v$ be an irreducible admissible representation of $G(k_v)$ whose central character is trivial on $Z_{G,H}(k_v)$. \tabularnewline
Define the multiplicity
$$m(\pi_v):=\dim \Hom_{H(k_v)}(\pi_v,\chi_v) .$$

In Section \ref{sec multiplicity}, for all the models in Table \ref{fig:1} except the $E_7$ case, we will prove a multiplicity formula $m(\pi_v)=m_{geom}(\pi_v)$ for all the tempered representations over non-archimedean fields or complex field. In the real case, we can prove the multiplicity formula for Models 1--4.
Then by using the multiplicity formula, together with the character identity in the local Langlands conjecture, we can show that the summation of the multiplicities is always equal to 1 over every local tempered Vogan $L$-packet (i.e. strong multiplicity one over the L-packet). Moreover, we will also show that the unique distinguished element in the L-packet corresponds to a character of the component group (note the the component group for some cases in Table \ref{fig:1} is not necessarily abelian). We refer the reader to Section \ref{sec multiplicity} for more details.

\begin{rmk}
The local multiplicity of some models in Table \ref{fig:1} has already been studied in our previous works. More specifically, Model 4  has been studied by the first author (\cite{Wan15}, \cite{Wan16}, \cite{Wan17}), Model 5 has been studied in our previous paper \cite{WZ}, and Model 1 has been studied in \cite{PWZ19}.
\end{rmk}

\begin{rmk}
Like in the Gan--Gross--Prasad model case (Section 17 of \cite{GGP}), one can also formulate an explicit conjecture about the unique distinguished element in the L-packet using the local epsilon factor $\epsilon(s,\pi_v,\rho_X)$ (i.e. the epsilon dichotomy conjecture). We will discuss this in our next paper \cite{WZ1}.
\end{rmk}

\subsection{The Ichino--Ikeda type conjecture}
Combining the results in the previous two subsections, we can now formulate the Ichino--Ikeda type conjectures for all the models in Table \ref{fig:1}. Let $(G,H)$ be one of these models. Since we assume that the central character is trivial on $Z_{G,H}$, we are actually working with the model $(G/Z_{G,H},H/Z_{G,H})$.
Following the definition in Section 16.5 of \cite{SV}, 
the pure inner forms of the spherical varieties are parameterized by the set $H^1(k,H/Z_{G,H})$. For all the models in Table \ref{fig:1} except Model 2, there is a natural bijection between the set $H^1(k,H/Z_{G,H})$ and the set of quaternion algebras $D$ over $k$. For each quaternion algebra $D/k$ (or for each $D\in H^1(k,H/Z_{G,H})$ in the case of Model 2), we can define an analogue of the model $(G,H)$ associated to $D$, which will be denoted by $(G_D,H_D)$. We can also define the period integral $\CP_{H_D}(\phi_D)$ and the local relative character $I_{H_{D,v}}(\phi_{D,v})$ where $\phi_D$ is a cuspidal automorphic form on $G_D(\BA)$. 
We refer the reader to later sections for the detailed descriptions  of $(G_D,H_D)$ for each spherical variety in Table \ref{fig:1}.
Remark that in our cases $G_D$ and $H_D$ are not the pure inner forms of $G$ and $H$ in general.
But after module the central part $Z_{G,H}$, they become pure inner forms of $G/Z_{G,H}$ and $H/Z_{G,H}$, respectively. 
% For the spherical varieties in Table \ref{fig:1}, the dual group $\check{G}_X$ is equal to $\hat{G}$. \red{The quasi-split case $\GU_6$ is still well-defined?}
% The ${\bf X}$-distinguished  Arthur parameters in sense of Definition in \cite[Section 16.2]{SV} are the tempered Arthur-parameters of $G$.

We fix a global tempered cuspidal $L$-packet $\Pi_{\phi}=\cup_D \Pi_\phi(G_D)$ of $G(\BA)$ whose central character is trivial on $Z_{G,H}(\BA)$. 
For each $\pi_D\in \Pi_\phi(G_D)$ in the $L$-packet, as in Section 17.4 of \cite{SV}, let $\nu:\pi_D\rightarrow \CA_{cusp}(G_D(\BA))$ be an embedding such that the period integral is identically zero on the orthogonal complement of $\nu(\pi_D)$ in the $\pi_D$-isotypic component $\CA_{cusp}(G_D(\BA))_{\pi_D}$. %\red{Revise. Automorphic realization.}
This embedding is not unique if the multiplicity of $\pi_D$ in $\CA_{cusp}(G_D(\BA))$ is greater than 1, but it does not affect the global conjecture.

We first consider all the models in Table \ref{fig:1} except the first one. 
For those models,  the center of $H/Z_{G,H}$ is anisotropic.  

% \red{Follow \cite{II} to Choose the Haar measure on $G(\BA).$}

\begin{conj}\label{strong global conjecture}
Let $D/k$ be a quaternion algebra that may be split (or $D\in H^1(k,H/Z_{G,H})$ if we are in the case of Model 2), $\pi_D\in \Pi_{\phi}(G_D)$ and $\phi_D\in \nu(\pi_D)$. We have
\begin{eqnarray*}
|\CP_{H_D}(\phi_D)|^2&=&\frac{1}{|S_{\phi}|}\cdot \frac{C_{H/Z_{G,H}}}{\Delta_{H_0/Z_{G,H}}(1)^S}\cdot \lim_{s\rightarrow 1} \frac{\Delta_G(s)^S}{L(1,\Pi_{\phi},Ad)^S}\\
&&\cdot L(1/2,\Pi_{\phi},\rho_X)^S \cdot\Pi_{v\in S} I_{H_{D,v}}(\phi_{D,v})
\end{eqnarray*}
where
\begin{itemize}
\item $S$ is a finite subset of $|k|$ such that $\phi$ is unramified outside $S$, and $\Delta_{H/Z_{G,H}}(1)^S,\Delta_G(s)^S,$ $L(1/2,\Pi_{\phi},\rho_X)^S,L(1,\Pi_{\phi},Ad)^S$ are the partial L-functions.
\item $C_{H/Z_{G,H}}$ is the Haar measure constant of $H/Z_{G,H}$ defined in Section 1 of \cite{II} (see also Section 1 of \cite{L}), and 
 the period integral $\CP_{H_D}$ is defined by the Tamagawa measure on $Z_{G_D,H_D}(\BA)\bks H_D(\BA)$.
\item $S_{\phi}$ is the conjectural global component group associated to the L-packet $\Pi_{\phi}$. We refer the reader to Section 3.2 of \cite{LM} for details.
\end{itemize}
\end{conj}

Then we consider the first model $(\GL_4\times\GL_2,\GL_2\times\GL_2)$ in Table \ref{fig:1}. In this case, we have $Z_H/Z_{G,H}\cong \GL_1$ and $$(G/Z_{G,H},H/Z_{G,H})=(\GL_4\times \GL_2/\GL_{1}^{diag},\; \GL_2\times \GL_2/\GL_{1}^{diag}).$$

\begin{conj}\label{strong global conjecture quaternion}
% Let $D/k$ be a quaternion algebra, $\pi_D\in \Pi_{\phi}(G_D)$ and $\phi_D\in \nu(\pi_D)$.
Under the above notation,
we have
\begin{eqnarray*}
|\CP_{H_D}(\phi_D)|^2 & = & \frac{1}{|S_{\phi}|}\cdot \frac{C_{H/Z_{H}}}{\Delta_{H_0/Z_{H}}(1)^S}\cdot \lim_{s\rightarrow 1}  \frac{\Delta_G(s)^S}{L(1,\Pi_{\phi},Ad)^S} \\
&& \cdot L(1/2,\Pi_{\phi},\rho_X)^S \cdot\Pi_{v\in S} \zeta_v(1)I_{H_{D,v}}(\phi_{D,v}).
\end{eqnarray*}
\end{conj}

Note that 
we have the extra factor $\zeta_v(1)$ due to $Z_H/Z_{G,H}=\GL_1$. 
This point of view has been discussed in Section 17.5 of \cite{SV}.

In particular, we have the following weak global conjecture, which is a direct consequence of the conjectures above and the multiplicity-one theorem on the local Vogan packets.

\begin{conj}\label{weak global conjecture}
The following are equivalent:
\begin{enumerate}
\item $L(\frac{1}{2},\Pi_{\phi},\rho_X)\neq 0$;
\item There exists a quaternion algebra $D/k$  (or $D\in H^1(k,H/Z_{G,H})$ if we are in the case of Model 2) such that the period integral $\CP_{H_D}(\phi_D)$ is nonzero for some $\phi_D\in \nu(\pi_D)$ and $\pi_D\in \Pi_{\phi}(G_D)$. 
\end{enumerate}
Moreover, if the above conditions hold, there exist  a unique $D$ and a unique $\pi_D\in \Pi_{\phi}(G_D)$ that satisfy Condition (2).
\end{conj}

When $D/k$ is split, one direction of Conjecture \ref{weak global conjecture} has been proved for Models  1 and 4  in joint works of the first author with Pollack and Zydor (\cite{PWZ18}, \cite{PWZ19}). 

Finally, similar to Gan--Gross--Prasad models as discussed in Section 27 of \cite{GGP}, one expects that the central value of $L$-functions in Models 2 and 3 of Table \ref{fig:1} are related to the arithmetic geometry of the cycles of the certain Shimura varieties. 
In Model 2, $\GU_4\times\GU_2$ and $(\GU_2\times\GU_2)^0$ can be associated with Shimura varieties of dimensions 5 and 2  (resp. 3 and 1).
In Model 3, $\GSp_6\times \GSp_4$ and $(\GSp_4\times \GSp_2)^0$ can be associated with Shimura varieties of dimensions 9 and 4. 
Then predicted by Beilinson--Bloch Conjecture, the order of $L(s,\pi,\rho_X)$ at $s=1/2$ should be related to the rank of the Chow groups of the corresponding cycles, which are all in the middle degree.
Like in Gross--Prasad models,  Conjecture \ref{strong global conjecture} would help one to relate the height pairing against the cycles to the first derivatives $L'(1/2,\pi,\rho_X)$.   

\subsection{Organization of the paper}
In Section \ref{sec:strategy}, we explain the strategy of our computation of the local relative characters. 
In Sections \ref{sec:GSp6-model} and \ref{sec:GL4}, we compute the local relative characters for the two split reductive cases in Table \ref{fig:1}.
In Sections \ref{sec:GL6} and \ref{sec:E7}, we study the non-reductive cases for $\GL_6$ and $E_7$, respectively.
In Section \ref{sec GU}, we deal with the non-split models $(\GU_4\times \GU_2,(\GU_2\times \GU_2)^0)$ and $(\GU_6,\GL_2\ltimes U)$ in Table \ref{fig:1}.
In Section \ref{sec:remaining}, we compute the formulas for the remaining  4 models. 
Finally, in Section \ref{sec multiplicity},  we will study the local multiplicity for all these models.

\subsection{Acknowledgments} 
We thank Yiannis Sakellaridis for the helpful comments on the first draft of this paper and for the helpful discussions about the virtual colors. We thank Aaron Pollack and Michal Zydor for the many helpful discussions about magic triangle which lead to the discovery of Model 9 and 10 in Table \ref{fig:1}. We thank Wei Zhang for the helpful discussions about Model 3 in Table \ref{fig:1}. We thank Dihua Jiang and Yifeng Liu for the helpful comments on the first draft of this paper. We also thank an anonymous referee for all the helpful comments and corrections. The work of the first author is partially supported by the NSF grant DMS-2000192 and DMS-2103720. 
The work of the second author is partially supported by AcRF Tier 1 grants A-0004274-00-00 and A-0004279-00-00 of National University of Singapore.

\section{The strategy}\label{sec:strategy}
In this section, we will explain the strategy of our computation. In the reductive cases, we closely follow the method developed by Sakellaridis in \cite{Sa}. 
For the Whittaker induction cases, due to the non-trivial unipotent radical of $H$, the local characters $I_{H_v}(\phi_v)$ in \eqref{eq:intro-I}  in these cases are not absolutely convergent.
To overcome this convergent issue, we  modify the method by regularizing the unipotent integrals. 
Then for all cases, we can reduce the computation of local relative characters to 
evaluate the local integrals associated to each simple root of $G$ and verify certain combinatorial identities. 
We refer the reader to  the detailed strategies in
Section \ref{red summary} for the reductive case and in Section \ref{sec:6-steps} for the non-reductive case.

More precisely,
in Section \ref{sec:notation} we discuss some notation and conventions of spherical varieties. Then we discuss the strategies for the reductive cases in Sections \ref{sec:reductive-case} and \ref{sec:S-theta}, and for the non-reductive cases in Sections 
\ref{sec non-red strategy 1} and \ref{sec non-red strategy 2}, respectively.

In Sections \ref{sec:strategy}--\ref{sec:remaining},
we only consider the non-archimedean places $v$ such that all data are unramified. Denote by $F=k_v$ a $p$-adic field.
Let $\CO_F$ be its ring of integers. Fix a uniformizer $\varpi$, and  denote by $\BF_q$ the residue field of $F$ with cardinality $q$ and of characteristic $p$ with $p\neq 2$. 
Fix a nontrivial unramified additive character $\psi:F\rightarrow \BC^{\times}$ of $F$.

\subsection{Notation}\label{sec:notation}
Let $G$ be a connected reductive group defined over $F$,  and $Z_G$ be the center of $G$. We fix a maximal open compact subgroup $K$ of $G(F)$ and let ${\rm d} g$ be the Haar measure on $G(F)$ such that the volume of $K$ is equal to 1. 
Denote by $W_G$ the Weyl group of $G(F)$.

\begin{defn}\label{defn:generic}
Let $P=LU$ be a proper parabolic subgroup of $G$ defined over $F$. For a character $\xi:U(F)\rightarrow \BC^{\times}$ of $U(F)$,
denote by $L_{\xi}$ the neutral component of the stabilizer of $\xi$ in $L$ (under the adjoint action).  

A character $\xi$ is called a generic character of $U(F)$ if $\dim(L_{\xi})$ is minimal, i.e. $\dim(L_{\xi})\leq \dim(L_{\xi'})$ for any character $\xi'$ of $U(F)$. It is easy to see that if $\xi$ is a generic character, so is ${}^l\xi$ for all $l\in L(F)$,
where ${}^l\xi$ is the character of $U(F)$ defined by ${}^l\xi(n)=\xi(l^{-1}nl)$.
\end{defn}
Moreover, there are finitely many generic characters of $U(F)$ up to $L(F)$-conjugation, which are in bijection with the open $L(F)$-orbits in $\Fu(F)/[\Fu(F),\Fu(F)]$ induced by the adjoint action on the Lie algebra $\Fu(F)$ of $U(F)$.

Let $H\subset G$ be a connected closed subgroup also defined over $F$.  We say that $H$ is a spherical subgroup if there exists a Borel subgroup $B$ of $G$ (not necessarily defined over $F$ since $G(F)$ may not be quasi-split) such that $BH$ is Zariski open in $G$. Such a Borel subgroup is unique up to $H(\bar{F})$-conjugation. 
Then,  $(G,H)$ is called a spherical pair and $X=G/H$ is the corresponding spherical variety of $G$.

From now on, we assume that $H$ is a spherical subgroup. We say the spherical pair $(G,H)$ is reductive if $H$ is reductive.
\begin{defn}\label{defn:Whittaker-induction}
A spherical pair $(G,H)$ is called  a Whittaker induction of a reductive spherical pair $(G_0,H_0)$ if there exists  a parabolic subgroup $P=LU$ of $G$, and a generic character $\xi:U(F)\rightarrow \BC^{\times}$ such that $H=H_0\ltimes U$ where $G_0\cong L$ and $H_0\cong L_{\xi}\subset L$ is the neutral component of the stabilizer of $\xi$ in $L$.
\end{defn}
Alternatively, we say that $(G,H)$ is the Whittaker induction of the triple $(G_0,H_0,\xi)$. 
For convenience, we also consider a reductive spherical pair $(G,H)$ as the Whittaker induction of $(G,H,1)$. 

\begin{rmk}
In general the stabilizer of a generic character is not necessarily a reductive  or   spherical subgroup of $L$.
For instance, if we take $G=\GL_3$ and a parabolic subgroup with Levi subgroup $L\cong \GL_2\times \GL_1$, then $L_{\xi}$ is isomorphic to  the Borel subgroup of $\GL_2$, which is not reductive;
if we take $G=\GL_9$ and a parabolic subgroup with Levi subgroup $L\cong \GL_3\times \GL_3\times \GL_3$, then $L_{\xi}\cong \GL_3$ is not a spherical subgroup of $L$.
\end{rmk}

Finally, for a reductive spherical pair $(G,H)$, we say it is {\it strongly tempered} if all the tempered matrix coefficients of $G(F)$ are absolutely convergent on $H(F)/Z_{G,H}(F)$. 
If the spherical pair $(G,H)$ is the Whittaker induction of a reductive spherical pair $(G_0,H_0)$, we say $(G,H)$ is strongly tempered if $(G_0,H_0)$ is strongly tempered.

In the rest of this section, we  assume that $G$ is split (this is true for all the models in Table \ref{fig:1} except the $\GU_4\times \GU_2$ and $\GU_6$ cases). The computation for the quasi-split case is slightly different from the split case. 
We refer the reader to Section \ref{sec GU} for details.

\subsection{The reductive case: some reduction}\label{sec:reductive-case}
Let $(G,H)$ be a reductive strongly tempered spherical pair  with $G(F)$ split. Assume that it does not have Type $N$ spherical root. 
Let $B=TN$ be a Borel subgroup of $G$ defined over $F$, $T$ the maximal split torus in $B$ and $N$ the unipotent radical of $B$,  and $\bar{B}=T\bar{N}$ be its opposite. 
There exists a unique open Borel orbit $B(F)\eta H(F)$  (note that for each root system, we already choose suitable representatives $(G,H)$ so that it has unique open Borel orbit, see Remark \ref{rmk one open borel orbit}). 
For all the four models in \eqref{eq:4-models}, it is easy to verify  $H(F)\cap \eta^{-1}B(F)\eta=Z_{G,H}(F)$, i.e. the stabilizer of the open Borel orbit belongs to the center of $G$.

\begin{rmk}
This is not true if the spherical pair has a Type $N$ spherical root. For example, for the model $(\GL_3,\SO_3)$, the stabilizer of the open orbit is isomorphic to $(\BZ/2\BZ)^2$ and does not belong to the center of $G$.
\end{rmk}

Our goal is to compute the local relative character
$$I(\phi_\theta)=\int_{H(F)/Z_{G,H}(F)}\phi_\theta(h)\ud h$$	
where $\phi_\theta$ is the unramified matrix coefficient of $I_{B}^{G}(\theta)$ normalized by $\phi_\theta(1)=1$, $\theta$ is a unitary unramified character of $T(F)$, and $I_{B}^{G}(\cdot)$ is the normalized induced representation from the Borel subgroup $B$. The integral is absolutely convergent since $(G,H)$ is strongly tempered. We follow the method in Sections 6-7 of \cite{Sa}. 

Let $f_\theta$ be the unramified vector in $I_{B}^{G}(\theta)$ with $f_\theta(1)=1$. 
Then the normalized unramified matrix coefficient $\phi_\theta$ is given by
$\phi_\theta(g)=\int_K f_\theta(kg)\ud k.$
This implies that 
$$I(\phi_\theta)=\int_{H(F)/Z_{G,H}(F)}\phi_\theta(h)\ud h=\int_{H(F)/Z_{G,F}(F)}\int_{K} f_\theta(kh)\ud k\ud h$$
$$=\int_{K}\int_{H(F)/Z_{G,F}(F)} f_\theta(kh)\ud h \ud k.$$
Note that since the integral is convergent if we replace $\theta$ by its absolute value (which changes $f_\theta$ to $f_{|\theta|}=|f_\theta|$), 
the above double integral is absolutely convergent. 
In particular, the integral
\begin{equation}\label{3.1}
\int_{H(F)/Z_{G,F}(F)} f_\theta(kh)\ud h
\end{equation}
is absolutely convergent for almost all $k\in K$. As a function on $k\in G$, this integral is right $H(F)$-invariant and left $(B(F),\delta^{1/2}_B\theta)$-invariant, where $\delta_B$ is the modular character of $B$. 
Since $B(F)\eta H(F)$ is open in $G(F)$, we have the integral \eqref{3.1} is absolutely convergent for all $k\in B(F)\eta H(F)$. 
% Thus we can define the function $T_\theta$  on $G(F)$ by
% $$T_{\theta}(g)=
% \begin{cases}
% 	\int_{H(F)/Z_{G,H}(F)} f_{\theta}(gh)\ud h &\text{if }g\in B(F)\eta H(F),\\
% 	0 &\text{otherwise.}
% \end{cases}
%  $$

On the other hand, consider the function $\CY_\theta$ on $G(F)$ satisfying the following conditions:
\begin{enumerate}
\item $\CY_\theta$ is supported on the open orbit $B(F)\eta H(F)$ with $\CY_\theta(\eta) = 1$;
\item $\CY_\theta$ is right $H(F)$-invariant and left $(B(F),\theta^{-1}\delta_{B}^{1/2})$-invariant.
\end{enumerate}
For $g\in B(F)\eta H(F)$, $\CY_{\theta^{-1}}(g)$   is proportional to \eqref{3.1}  and then 
$$\int_{H(F)/Z_{G,F}(F)} f_\theta(gh)\ud h=\int_{H(F) /Z_{G,H}(F)} f_{\theta}(\eta h)\ud h \cdot \CY_{\theta^{-1}}(g).$$
In consequence, since the complementary set of $B(F)\eta H(F)$ has measure zero,
 we have
\begin{eqnarray*}
I(\phi_\theta)&=&\int_{K}\int_{H(F)/Z_{G,F}(F)} f_\theta(kh)\ud h \ud k \\
& =&\int_{K}\CY_{\theta^{-1}}(k) \ud k\times \int_{H(F) /Z_{G,H}(F)} f_{\theta}(\eta h)\ud h.
\end{eqnarray*}
To obtain a formula of $I(\phi_\theta)$,  it suffices to compute
$$\int_{K}\CY_{\theta^{-1}}(k) \ud k \;\;\text{ and }\;\;\int_{H(F) /Z_{G,H}(F)} f_{\theta}(\eta h)\ud h.$$

To evaluate the integral $\int_{H(F) /Z_{G,H}(F)} f_{\theta}(\eta h)\ud h$, we need the following lemma. 
\begin{lem}\label{lem Haar measure}
Under the above notation,
for $f\in C^{\infty}_{c}(G(F))$, we have
\[
\int_{G}f(g)\ud g=\frac{\Delta_{G}(1)}{\Delta_{H/Z_{G,H}}(1)}\zeta(1)^{-rk(G)}\int_{H(F)/Z_{G,H}(F)}\int_{B(F)}f(b\eta h)\ud b \ud h,
\]
where $rk(G)$ is the $F$-rank of $G$.
\end{lem}
% \red{Recall that we always choose the Haar measures so that the volume of the maximal compact subgroup is equal to 1.}

\begin{proof}
% \red{Revise the proof for $B\eta H$.}
Without loss of generality, it is sufficient to consider the case $\eta=1$, that is,
$H(F)B(F)$ is an open dense subset of $G(F)$.
Denote by ${\rm d}_{\rm can}g$, ${\rm d}_{\rm can}b$, and ${\rm d}_{\rm can}h$ the Weil's canonical measures on the smooth varieties $G$, $B$ and $H/Z_{G,H}$, respectively. 
Since $B\cap H=Z_{G,H}$ and $BH$ is open dense in $G$, by  \cite[Chapter 2]{Weil} we have
\[
\int_{G(F)}f(g)\ud_{\rm can} g=\int_{H(F)/Z_{G,H}(F)} \int_{B(F)}f(bh)\ud_{\rm can} b\ud_{\rm can} h.
\]

By \cite[Chapter 2]{Weil}, since the smooth varieties $X$ under consideration are smooth over $\CO_F$ and have good reduction over $\BF_q$, 
we have
\[
vol(X(\CO_F),{\rm d}_{\rm can}x)=\frac{|X(\BF_q)|}{q^{\dim X}}.
\]
This implies that
\[
\ud_{\rm can} g=\frac{|G(\BF_q)|}{q^{\dim G}}\ud g, ~~
\ud_{\rm can} b=\frac{|B(\BF_q)|}{q^{\dim B}}\ud b,
\ud_{\rm can} h=\frac{|H/Z_{G,H}(\BF_q)|}{q^{\dim (H)-\dim(Z_{G,H})}}\ud h.
\]
Since $B\cap H=Z_{G,H}$, we have
\[
\int_{G}f(g)\ud g=\frac{|B(\BF_q)|\cdot |H/Z_{G,H}(\BF_q)| }{|G(\BF_q)|}\int_{H(F)/Z_{G,H}(F)}\int_{B(F)}f(bh)dbdh.
\]
Now the lemma follows from the following equation which is a consequence of (3.1) and (5.1) of \cite{G}
$$\frac{|B(\BF_q)|\cdot |H/Z_{G,H}(\BF_q)| }{|G(\BF_q)|}=\frac{\Delta_{G}(1)}{\Delta_{H/Z_{G,H}}(1)}\zeta(1)^{-rk(G)}.$$
\end{proof}

By Lemma \ref{lem Haar measure}, 
$\int_{K}\CY_{\theta}(k)\ud k=\int_{G(F)}1_K(g)\CY_{\theta}(g)\ud g$ is equal to
\begin{eqnarray*}
&&\frac{\Delta_{G}(1)}{\Delta_{H/Z_{G,H}}(1)}\zeta(1)^{-rk(G)}\int_{H(F)/Z_{G,H}(F)}\int_{B(F)}1_K(b\eta h)\theta^{-1}\delta^{\frac{1}{2}}(b)\ud b\ud h \\
&=&\frac{\Delta_{G}(1)}{\Delta_{H/Z_{G,H}}(1)}\zeta(1)^{-rk(G)}\int_{H(F)/Z_{G,H}(F)}f_{\theta}(\eta h) \ud h,
\end{eqnarray*}
where $1_K$ is the characteristic function on $K$.
As a result, we have proved the following proposition, which reduces  to evaluate the integral $\int_{K}\CY_{\theta}(k)\ud k$.

\begin{prop}\label{pro:I-Y}
The local relative character $I(\phi_\theta)$ is equal to
\begin{eqnarray*}
&&\int_{K}\CY_{\theta^{-1}}(k) dk\times \int_{H(F) /Z_{G,H}(F)} f_{\theta}(\eta h)\ud h\\
&=&\frac{\Delta_{H/Z_{G,H}}(1)}{\Delta_{G}(1)}\zeta(1)^{rk(G)}\int_{K}\CY_{\theta^{-1}}(k) \ud k\times \int_{K}\CY_{\theta}(k) \ud k.
\end{eqnarray*}
\end{prop}

In the next subsection \ref{sec:S-theta}, we will explain how to compute the integral $\int_{K}\CY_{\theta}(k)\ud k$.

% \red{Recall the normalized spherical root $\gamma$ in \cite{Sa}.
% Define $\Theta^+$ to be the positive cone given by $\gamma$ in weights of $\rho_X$ and then we obtain a decomposition of all weights of $\rho_X$.}

\begin{prop}\label{main prop}
Let $\Phi^+$ be the set of positive roots of $G$. There is a decomposition of the weights of a representation $\rho_X$ of $\hat{G}$, denoted by $\Theta=\Theta^+\cup \Theta^-$, such that
\begin{equation}\label{main prop equation 1}
\int_{K}\CY_{\theta}(k) \ud k=\frac{\Delta_{G}(1)}{\Delta_{H/Z_{G,H}}(1)}\zeta(1)^{-rk(G)}\cdot \beta(\theta),
\end{equation}
where 
$$\beta(\theta)=\frac{\prod_{\alpha\in \Phi^+}1-q^{-1}e^{\alpha^\vee}}{
	\prod_{\gamma^\vee\in \Theta^+}1-q^{-\frac{1}{2}}e^{\gamma^\vee}}(\theta).$$
Moreover, we have
\begin{equation}\label{theta plus}
\prod_{\gamma^\vee\in \Theta^+}1-q^{-\frac{1}{2}}e^{\gamma^\vee}(\theta^{-1})=\prod_{\gamma^\vee\in \Theta^-}1-q^{-\frac{1}{2}}e^{\gamma^\vee}(\theta).
\end{equation}
\end{prop}

Here for $\alpha\in \Phi^+$, we use $e^{\alpha^{\vee}}(\theta)$ to denote $\theta(e^{\alpha^\vee}(\varpi))$. For $\gamma^{\vee}\in \Theta^+$, we can identify it with a co-weight of $G$ and we let $e^{\gamma^{\vee}}$ be the associated homomorphism  from $\GL_1$ to $T$. We define $e^{\gamma^{\vee}}(\theta)=\theta(e^{\gamma^\vee}(\varpi))$.

\begin{rmk}
For all the models in Table \ref{fig:1}, the representation $\rho_X$ in the proposition above (or Proposition \ref{main prop nonred} for the non-reductive cases) is just the representation $\rho_X$ listed in Table \ref{fig:1}. In Theorem 7.2.1 of \cite{Sa}, for general (not necessarily strongly tempered) spherical varieties, Sakellaridis proved the identities \eqref{main prop equation 1} and \eqref{theta plus} for a $W_X$-invariant set $\Theta$ of weights of $\hat{G}$. Here $W_X\subset W$ is the little Weyl group of $X$ and we have $W=W_X$ if the model is strongly tempered. Later in Corollary 7.3.3 of \cite{SW}, Sakellaridis-Wang proved that in the case when $(G,H)$ over functional fields is strongly tempered and $H$ is reductive, $\Theta$ must be the set of weights of a representation $\rho_X$ of $\hat{G}$. Our computation in later sections shows that for all the non-reductive cases in Table \ref{fig:1}, $\Theta$ is also the set of weights of a representation $\rho_X$ of $\hat{G}$.
\end{rmk}

Combining Propositions \ref{pro:I-Y} and \ref{main prop}, we have
$$I(\phi_\theta)=\frac{\Delta_{H/Z_{G,H}}(1)}{\Delta_{G}(1)}\zeta(1)^{rk(G)}\int_{K}\CY_{\theta^{-1}}(k) \ud k\times \int_{K}\CY_{\theta}(k) \ud k$$
$$=\frac{\Delta_{G}(1)}{\Delta_{H/Z_{G,H}}(1)}\zeta(1)^{-rk(G)}\cdot \beta(\theta)\cdot \beta(\theta^{-1})=\frac{\Delta_{G}(1)}{\Delta_{H/Z_{G,H}}(1)}\cdot \frac{L(1/2,\pi,\rho_X)}{L(1,\pi,\Ad)}.$$
This finishes the computation. The $L$-functions $\frac{L(1/2,\pi,\rho_X)}{L(1,\pi,\Ad)}$ is just the $L_X$, $L$-function of the spherical variety $X=G/H$, defined in Definition 7.2.3 of \cite{Sa}.

\subsection{The computation of $S_\theta$} \label{sec:S-theta}
Set
\[
S_\theta(g)=\int_K \CY_\theta(kg^{-1})\ud g \text{ for } g\in G(F),	
\]
which is the Whittaker-Shintani function.
(See \cite{KMS03} for instance.)
In this section, our goal is to prove Proposition \ref{main prop}, i.e. compute $S_\theta(1)$.
Here, we follow the arguments and the notation in \cite{KMS03}. The unexplained notations and more details are referred to \cite{KMS03,Mac,C80}.

Let $\CI=B(\CO_F)\bar{N}(\varpi \CO_F)$ be the Iwahori subgroup of $G(F)$. For all the strongly tempered models in the introduction, we can choose a representative $\eta$ in the open double coset of $B(F)\bks G(F)/H(F)$ so that it satisfies the following lemma (we will check this lemma for each model in the later sections).

\begin{lem}\label{lem eta}
Then there exists a representative $\eta$ for the open double coset of $B(F)\bks G(F)/H(F)$  such that $\eta\in K$  and  
$\bar{N}(\varpi \CO_F)\eta\subset T(\CO_F)N(\varpi \CO_F)\eta H(\CO_F).$
\end{lem}

For $w\in W$ ($W$ is the Weyl group of $G$), let $\Phi_w=1_{\CI w \CI}$ be the characteristic function of $\CI w \CI$. Then $1_K$ is equal to $\sum_{w\in W} \Phi_w$. (See \cite{Iwa66} for instance.) 
Let $\alpha$ be a simple root and $w_\alpha$ be the corresponding reflection in $W$.
 We would need to compute
$$I_\alpha(\theta)=vol(\CI)^{-1}\int_{G(F)} \CY_\theta(x\eta)(\Phi_1(x)+\Phi_{w_\alpha}(x))\ud x
$$
for all simple roots $\alpha$.

First, by Lemma \ref{lem eta}, we have $\CI \eta\subset B(\CO_F)\eta H(\CO_F)$. Hence $\CY_\theta(x\eta)=1$ for all $x\in \CI$. This implies that 
\begin{equation}\label{eq:I-1}
vol(\CI)^{-1}\int_{G(F)} \CY_\theta(x\eta)\Phi_1(x)\ud x=1.	
\end{equation} 

For each root $\alpha\in \Phi_G$ of $G$,
let  (note that all the root spaces are one dimensional since we have assumed that $G$ is split)
\begin{equation}
u_{\alpha}\colon a\in F\mapsto u_{\alpha}(a)\in N(F)	
\end{equation}
be the one-parameter unipotent subgroup of $G(F)$ associated to the root $\alpha$.

\begin{lem}
We have
\begin{equation}\label{eq:I-alpha-1}
I_\alpha(\theta)=1+q\int_{\CO_F} (\theta^{-1}\delta^{\frac{1}{2}}_B)(e^{\alpha^\vee}(a^{-1}))\CY_\theta(u_{-\alpha}(a^{-1})\eta)\ud a,
\end{equation}
where $\delta_B$ is the modular character of $B$.
\end{lem}

\begin{proof}
This proof is similar to the one of Lemma 8.4 in \cite{KMS03}.
It is sufficient to compute the integral
\[
vol(\CI)^{-1}\int_{\CI w_\alpha \CI} \CY_\theta(x\eta)\ud x.
\]
First, let us evaluate $\CY_\theta(\cdot)$ on the set $\CI w_\alpha \CI \eta$. 
Referring to  \cite[Chapter 2]{Mac},
one has 
\[
\CI w_\alpha\CI=B(\CO_F)w_\alpha U_{\alpha}(\CO_F) \bar{N}(\varpi \CO_F)
\text{ and }
U_\alpha(\CO_F)=\{u_{\alpha}(a) \mid a\in \CO_F\},
\]
By Lemma \ref{lem eta},  
$$\CI w_\alpha \CI \eta\subset B(\CO_F) w_\alpha U_{\alpha}(\CO_F) \eta H(\CO_F).$$
As $vol(\CI w_\alpha \CI)=q\cdot vol(\CI)$, it follows that 
\begin{eqnarray*}
\int_{\CI w_\alpha \CI} \CY_\theta(x\eta)\ud x &=& vol(\CI w_\alpha \CI)\int_{\CO_F} \CY_\theta(w_\alpha u_{\alpha}(a)\eta)\ud a\\
&=&q\cdot vol(\CI)\int_{\CO_F} \CY_\theta(w_\alpha u_{\alpha}(a)\eta)\ud a.
\end{eqnarray*}
Then, since $w_\alpha u_{\alpha}(a)=u_{\alpha}(a^{-1})t_0\cdot e^{\alpha^\vee}(a^{-1})u_{-\alpha}(a^{-1})$
for some $t_0\in T(\CO_F)$, 
we have
\[
\CY_\theta(w_\alpha u_{\alpha}(a)\eta)=(\theta^{-1}\delta_B^{\frac{1}{2}})(e^{\alpha^\vee}(a^{-1}))\CY_\theta(u_{-\alpha}(a^{-1})\eta).
\]
This proves the lemma.
\end{proof}

Then for each model, by an explicit matrix computation, we will show that there exists $\beta_\alpha^{\vee}\in \Theta$ such that $-\beta_\alpha^{\vee}+\alpha^{\vee}\in \Theta$ (here we view $\alpha^{\vee}$ as a weight on the dual group) and  
\begin{equation}\label{main identity}
\CY_\theta(u_{-\alpha}(a^{-1})\eta)=\theta(e^{\beta_{\alpha}^{\vee}}(1+a^{-1}))\cdot |1+a^{-1}|^{-1/2}.
\end{equation}

This implies that
\begin{align}
I_\alpha(\theta)=&1+q\int_{\CO_F} (\theta^{-1}\delta^{\frac{1}{2}})(e^{\alpha^\vee}(a^{-1}))\CY_\theta(u_{-\alpha}(a^{-1})\eta)\ud a \label{I alpha red}\\
=&1+q\int_{\CO_F} (\theta^{-1}\delta^{\frac{1}{2}})(e^{\alpha^\vee}(a^{-1}))\theta(e^{\beta_{\alpha}^{\vee}}(1+a^{-1}))\cdot |1+a^{-1}|^{-1/2} \ud a\nonumber\\
=&1+q\int_{\CO_F} \theta(e^{\alpha^{\vee}}(a))\cdot |a|^{-1} \cdot \theta(e^{\beta_{\alpha}^{\vee}}(1+a^{-1}))\cdot |1+a^{-1}|^{-1/2} \ud a\nonumber\\
=&1+q\int_{\CO_F} (\theta(e^{\beta_{\alpha}^{\vee}})\cdot |\;|^{-1/2})(1+a)\cdot (\theta(e^{\alpha^{\vee}-\beta_{\alpha}^{\vee}})\cdot |\;|^{-1/2})(a)\ud a \nonumber\\
=&(q-1)\cdot \frac{1-q^{-1}e^{\alpha^{\vee}}(\theta)}{(1-q^{-1/2} e^{\beta_{\alpha}^{\vee}}(\theta))(1-q^{-1/2}e^{-\beta_{\alpha}^{\vee}+\alpha^{\vee}}(\theta))}.\nonumber
\end{align}
% $$=1+q\int_{\CO_F} (\theta^{-1}\delta^{\frac{1}{2}})(e^{\alpha^\vee}(a^{-1}))\theta(e^{\beta_{\alpha}^{\vee}}(1+a^{-1}))\cdot |1+a^{-1}|^{-1/2} \ud a 	$$
% $$=1+q\int_{\CO_F} \theta(e^{\alpha^{\vee}}(a))\cdot |a|^{-1} \cdot \theta(e^{\beta_{\alpha}^{\vee}}(1+a^{-1}))\cdot |1+a^{-1}|^{-1/2} \ud a$$
% $$=1+q\int_{\CO_F} (\theta(e^{\beta_{\alpha}^{\vee}})\cdot |\;|^{-1/2})(1+a)\cdot (\theta(e^{\alpha^{\vee}-\beta_{\alpha}^{\vee}})\cdot |\;|^{-1/2})(a)da$$
% $$=(q-1)\cdot \frac{1-q^{-1}e^{\alpha^{\vee}}(\theta)}{(1-q^{-1/2} e^{\beta_{\alpha}^{\vee}}(\theta))(1-q^{-1/2}e^{-\beta_{\alpha}^{\vee}+\alpha^{\vee}}(\theta))}.$$
Here we use the fact that for unitary unramified characters $\chi_1,\chi_2$ of $F^{\times}$, the integral
\begin{equation}\label{rank one integral}
q\int_{\CO_F} (\chi_1\cdot |\;|^{-1/2})(1+a)\cdot (\chi_2\cdot |\;|^{-1/2})(a)da
\end{equation}
is absolutely convergent and is equal to
$$q-2+(q-1)\cdot \frac{q^{-1/2}\chi_1(\varpi)+ q^{-1/2}\chi_2(\varpi)-2q^{-1}\chi_1\chi_2(\varpi)}{(1-q^{-1/2}\chi_1(\varpi))(1-q^{-1/2}\chi_2(\varpi))}.$$
The proof of this identity is similar to (and easier than) the two identities in Section \ref{sec:two identities}. We omit the details here.

\begin{rmk}
The set $\{\beta_{\alpha}^{\vee},\;\alpha^{\vee}-\beta_{\alpha}^{\vee}\mid \alpha\in \Delta(G)\}$ is the set of virtual weighted colors of $X=G/H$ defined in Section 7.1 of \cite{Sa}. There is another way to compute the virtual weighted colors using the Luna diagram of $X=G/H$. In \cite{LU}, Luna computed the Luna diagram for all the split reductive spherical varieties of Type A. The Luna diagram of all the split reductive spherical varieties was computed in \cite{BP}. In Section \ref{sec:GSp6-model}, we will use the model $(\GSp_6\times \GSp_4,(\GSp_4\times \GSp_2)^0)$ as an example to explain how to use the Luna diagram to compute the virtual weighted colors. We refer the reader to Remark \ref{Luna diagram} for details.
\end{rmk}

\begin{defn}
Let $\Theta^+$ be the unique subset of $\Theta$ satisfying the following condition:
\begin{itemize}
\item For every simple root $\alpha$, we have $\Theta^+ - w_\alpha \Theta^+=\{\beta_{\alpha}^{\vee},\;\alpha^{\vee}-\beta_{\alpha}^{\vee}\}$.  
\end{itemize}
Recall that for all the models in Table \ref{fig:1}, $\Theta$ is the set of weights of the representation $\rho_X$ of $\hat{G}$ listed in Table \ref{fig:1}. We define 
$$\beta(\theta)=\frac{\prod_{\alpha\in \Phi^+}1-q^{-1}e^{\alpha^\vee}}{
	\prod_{\gamma^\vee\in \Theta^+}1-q^{-\frac{1}{2}}e^{\gamma^\vee}}(\theta)
\text{ and }
	c_{WS}(\theta)=\frac{\prod_{\gamma^\vee\in \Theta^+}1-q^{-\frac{1}{2}}e^{\gamma^\vee}}
{\prod_{\alpha\in \Phi^+}1-e^{\alpha^\vee}}(\theta).$$
\end{defn}

For a Weyl element $w\in W$, the intertwining operator $T_w \colon I^G_B(\theta)\to I^G_B(w\theta)$ is defined by
\[
T_w(f)(g)=\int_{N(F)\cap wN(F)w^{-1}\bks N(F)}f(w^{-1}ng)\ud n,\; f\in I^G_B(\theta).
\] 
It is absolutely convergent when $\theta$ is positive enough and admits a meromorphic continuation (see Theorem IV.1.1 of \cite{W03}).  

By Theorem 1.2.1 of \cite{Sa08}, the space $\Hom_{H(F)}(I^G_B(\theta),1)$ is one-dimensional for almost all $\theta$ in the unitary line. In fact, for all the cases in Table \ref{fig:1}, the little Weyl group $W_X$ of the spherical variety is equal to the Weyl group $W$. 
This implies that the factor $(\CN_W(\delta^{-1/2}A_{X}^{\ast}):W_X)$ in loc. cit. is equal to 1. 
Moreover, since the spherical variety has a unique open Borel orbit, the factor $|H^1(k,A_X)|$ in loc. cit. is also equal to 1 (see Lemma 3.4.1 of \cite{Sa08}). 
This implies that the Hom-space is one dimensional. 
In Section \ref{sec multiplicity}, we will prove a multiplicity formula of the dimension of this Hom space for all the tempered representations which will imply that the Hom-space $\Hom_{H(F)}(I^G_B(\theta),1)$ is actually one dimensional for all unitary characters. But we don't need this result in our computation.

By the definition of $\CY_\theta$, we can define an element $\ell_\theta$ in the $\Hom$-space $\Hom_{H(F)}(I^G_B(\theta),1)$ to be 
\begin{equation}\label{eq:l-theta}
\ell_\theta(\CP_{\theta}(f))=\int_{G(F)}f(g)\CY_\theta(g)\ud g, 
\text{ for } f\in\CC_c^\infty(G),	
\end{equation}
where $\CP_\theta(f)=\int_{B(F)}(\theta^{-1}\delta_B^{\frac{1}{2}})(b)f(bg)\ud b$ is the canonical $G(F)$-equivariant map from $\CC_c^\infty(G(F))$ to $I^{G}_B(\theta)$. Since the Hom-space is one dimensional for almost all $\theta$, for each simple reflection $w_\alpha\in W$ associated to a simple root $\alpha$, there exists a rational function $b_{w_\alpha}(\theta)$ on $\theta$ such that
\begin{equation}\label{eq:a-w}
\ell_{w_\alpha\theta}\circ T_{w_\alpha}=c_{\alpha}(\theta)b_{w_\alpha}(\theta)\ell_{\theta}.
\end{equation}
Here $c_\alpha(\theta)=\frac{1-q^{-1}e^{\alpha^\vee}}{1-e^{\alpha^\vee}}(\theta)$ is the $c$-function defined in \cite{C80}. 

Our goal is to obtain a formula of  $b_{w_\alpha}(\theta)$. 
Similar to the proof of Theorem 10.5 in \cite{KMS03},
to obtain $b_{w_\alpha}(\theta)$, we evaluate both sides of  equation \eqref{eq:a-w} at
$\CP_{\theta}(\Phi_1+\Phi_{w_\alpha})$.
Note that 
\[
T_{w_\alpha}(\CP_{\theta}(\Phi_1+\Phi_{w_\alpha}))=c_{\alpha}(\theta)\CP_{w_\alpha \theta}(\Phi_1+\Phi_{w_\alpha}).
\]
Then, under the choice \eqref{eq:l-theta} of $\ell_\theta$, we have
\[
vol(\CI)I_{\alpha}(\theta)=\ell_{\theta}(\CP_{\theta}\circ R(\eta)(\Phi_1+\Phi_{w_\alpha})),
\]
where $R$ is the right translation of $G(F)$.
On the other hand, 
$$\ell_{w_\alpha\theta}\circ T_{w_\alpha}(\CP_{\theta}\circ R(\eta)(\Phi_1+\Phi_{w_\alpha}))=
c_\alpha(\theta) b_{w_\alpha}(\theta) \ell_{w_\alpha\theta}(\CP_{w_\alpha\theta}\circ R(\eta)(\Phi_1+ \Phi_{w_\alpha}))$$
$$=c_\alpha(\theta)vol(\CI)I_{\alpha}(w_\alpha\theta).$$
Following \eqref{eq:a-w}, we obtain
\begin{equation}
b_{w_\alpha}(\theta)=\frac{I_{\alpha}(w_\alpha\theta)}{I_\alpha(\theta)}.	
\end{equation}
  
Recall that $S_{\theta}(g)=\ell_\theta(R(g)\CP_\theta(1_K))$.
Plugging $T_{w_\alpha}(\CP_{w_\alpha\theta}(1_K))=c_{\alpha}(\theta)\CP_{\theta}(1_K)$ into the left hand side of \eqref{eq:a-w}, we have
\[
S_{w_\alpha\theta}(g)=\ell_{w_\alpha\theta}(R(g)\CP_{w_\alpha\theta}(1_K))
\]
\[=c_{\alpha}(\theta)^{-1}(\ell_{w_\alpha\theta}\circ T_{w_{\alpha}})(R(g)\CP_{\theta}(1_K)) =b_{w_\alpha}(\theta) S_{\theta}(g).
\] 
Thus for all simple roots $\alpha$ of $G$, we have
 \[
\frac{S_{w_\alpha\theta}(g)}{S_{\theta}(g)}=b_{w_\alpha}(\theta)=\frac{I_{\alpha}(w_\alpha\theta)}{I_{\alpha}(\theta)}=\frac{\beta(w_\alpha\theta)}{\beta(\theta)}.
\]
This implies that
$S_\theta(g)/\beta(\theta)$ is $W$-invariant as a function of $\theta$.	

\begin{prop}\label{prop:WS-reductive}
Let $T(F)^+=\{t\in T(F)|\;t^{-1}N(\CO_F)t\subset N(\CO_F)\}$ be the positive chamber of $T(F)$. Then
\begin{align*}
S_\theta(\eta^{-1}t)/\beta(\theta)=&q^{l(W)}
vol(\CI)\sum_{w\in W}c_{WS}(w\theta)((w\theta)^{-1}\delta_B^{\frac{1}{2}})(t^{-1}), \text{ for } t\in T(F)^+,
\end{align*}
where $l(W)$ is the length of the longest Weyl element in $W$.
\end{prop}

\begin{proof}
First, we show that 
\begin{equation}\label{eq:S-R}
S_\theta(\eta^{-1}t)
=vol(\CI t(\lam)\CI)^{-1}R(1_{\CI t\CI})	S_{\theta}(\eta^{-1}),
\end{equation}
where $R$ is the right convolution defined by
\[
R(1_{\CI t\CI})	S_{\theta}(\eta^{-1})
=\int_{G(F)}1_{\CI t\CI}(x)S_{\theta}(\eta^{-1}x)\ud x
=\int_{\CI t\CI}	S_{\theta}(\eta^{-1}x)\ud x.
\]
Now, it is enough to show that
\[
\eta^{-1}\CI t\CI\subset H(\CO_F)\eta^{-1}t \CI \subset H(\CO_F)\eta^{-1}tK,
\]
which follows from Lemma \ref{lem eta}.

Similar to Proposition 1.10 of \cite{KMS03}, there exists a basis $\{f_w\colon w\in W\}$ of $I^G_B(\theta)^{\CI}$ such that 
\begin{equation}\label{eq:Phi-K-f}
R(1_{\CI t\CI})f_w=vol(\CI t^{-1}\CI)(w\theta)^{-1}\delta_B^{\frac{1}{2}}(t^{-1})f_w,
\end{equation}
$$f_1=\CP_{\theta}(\Phi_1),\; \CP_{\theta}(1_{K})=q^{l(w_\ell)}\sum_{w\in W}c_w(\theta)f_w$$
for $t\in T(F)^+$ where $c_w(\theta)=\displaystyle \prod_{\alpha>0,w\alpha>0}c_\alpha(\theta)$.

Recall that $S_\theta(g)=\ell_\theta(R(g)\CP_\theta(1_K))$.
Substituting \eqref{eq:Phi-K-f} into $S_\theta$, we have
\[
S_\theta(\eta t)/\beta(\theta)=q^{l(w_\ell)}\beta(\theta)^{-1}
\sum_{w\in W}c_w(\theta)\ell_\theta(R(\eta)f_w)\cdot(w\theta)^{-1}\delta^{\frac{1}{2}}(t^{-1}). 
\]

By \eqref{eq:I-1} and $f_1=\CP_\theta(\Phi_1)$, we have the coefficient of $(w\theta)^{-1}\delta^{\frac{1}{2}}(t^{-1})$ for $w=1$ in $S_\theta(\eta t)/\beta(\theta)$ is equal to
\[
q^{l(w_\ell)}\beta(\theta)^{-1}c_{1}(\alpha)=q^{l(w_\ell)}vol(\CI)c_{WS}(\theta),
\]
where $c_1(\theta)=\prod_{\alpha\in \Phi^+}\frac{1-q^{-1}e^{\alpha^\vee}}{1-e^{\alpha^\vee}}(\theta)$.
Since $S_\theta(\eta t)/\beta(\theta)$ is $W$-invariant, by the linear independence of the characters $(w\theta)^{-1}\delta^{\frac{1}{2}}(t^{-1})$ for generic $\theta$, we obtain
\[ 
 S_\theta(\eta^{-1}t)/\beta(\theta)= q^{l(w_\ell)}
 vol(\CI)\sum_{w\in W}c_{WS}(w\theta)(w\theta)^{-1}\delta^{\frac{1}{2}}(t^{-1}).
\]

\end{proof}

Since $\eta\in K$, we have $S_\theta(1)=S_\theta(\eta^{-1})$. Combining with the proposition above, we have
$$S_\theta(1)/\beta(\theta)=q^{l(W)}vol(\CI)\sum_{w\in W}c_{WS}(w\theta).$$
Since $vol(\CI)=\Delta_G(1)\zeta(1)^{-rk(G)}\cdot q^{-l(W)}$, we have
$$S_\theta(1)/\beta(\theta)=\Delta_G(1)\zeta(1)^{-rk(G)}\sum_{w\in W}c_{WS}(w\theta).$$
Hence in order to prove Proposition \ref{main prop}, it is enough to prove the following lemma.

\begin{lem}\label{lem constant reductive}
The summation $\displaystyle\sum_{w\in W}c_{WS}(w\theta)$ is independent of the choice of $\theta$ and is equal to $\frac{1}{\Delta_{H/Z_{G,H}}(1)}$.
\end{lem}

\begin{proof}
Since the spherical varieties for the reductive cases are affine, the first part of this statement follows from Theorem 7.2.1 of \cite{Sa}. 
For the second part, since the summation is independent of the choice of $\theta$, we can compute it by plugging in some special $\theta$. We will compute it for each of our models in later sections.
\end{proof}

\begin{rmk}\label{rmk constant reductive}
For all reductive cases, if we set $\theta=\delta_{B}^{1/2}$ as in Lemma 4.2.3 of \cite{Sa} (which is used in the proof of \cite[Theorem 7.2.1]{Sa}), then the only nonvanishing term in the summation is the term corresponding to the longest Weyl element, 
which is equal to $\frac{1}{\Delta_{H/Z_{G,H}}(1)}$. We will show this for all the reductive models in Table \ref{fig:1} in later sections.
\end{rmk}

\subsubsection{The summary}\label{red summary}
By the discussion in the previous two subsections, in order to compute the relative character in the reductive case, we just need to perform the following steps:

\begin{enumerate}
\item \label{item:reductive-1} Show that the double cosets $B(F)\back G(F)/H(F)$ have a unique open orbit $B(F)\eta G(F)$ and the representative $\eta$ can be chosen to satisfy Lemma \ref{lem eta}.
\item \label{item:reductive-2} Verify the identity \eqref{main identity} by expressing the product $u_{-\alpha}(a)\eta$ in terms of the decomposition $B(F)\eta H(F)$. This gives us the set of virtual weighted colors of $X$.
\item \label{item:reductive-3} Compute the subset $\Theta^+$ of $\Theta$ and show that it satisfies \eqref{theta plus}.
\item \label{item:reductive-4} Compute the constant $\sum_{w\in W}c_{WS}(w\theta)$, i.e. Lemma \ref{lem constant reductive}. This computation is easy for the reductive case, see Remark \ref{rmk constant reductive}.
\end{enumerate}

\subsubsection{The trilinear model}\label{section triple product}
To end this subsection, we use the trilinear $\GL_2$ model as an example to explain the method. This example also appeared in Section 7.2.4 of \cite{Sa} and we will use it for the non-reductive cases in Table \ref{fig:1} , which are the Whittaker inductions of the trilinear $\GL_2$-model. Let $G=\GL_2\times \GL_2\times \GL_2$ and $H=\GL_2$ diagonally embedded into $G$. Let $B$ be the upper triangular Borel subgroup of $G$ and 
$\eta_0=(I_2,\begin{pmatrix}0&1\\1&0 \end{pmatrix},\begin{pmatrix}0&1\\1&0 \end{pmatrix}\begin{pmatrix}1&1\\0&1 \end{pmatrix}).$ It is easy to see that $B(F)\eta_0 H(F)$ is the unique open orbit and $\eta_0$ satisfies Lemma \ref{lem eta}. 

% {\it Step \eqref{item:reductive-2}.} 
Let $\Theta$ be the set of weights of the tensor product representation of $\GL_2(\BC)\times \GL_2(\BC)\times \GL_2(\BC)$.  
We can write it as
$\{e_i+e_j'+e_k'' \mid 1\leq i,j,k\leq 2\}.$
Let $\alpha_i\;(1\leq i\leq 3)$  be the simple root of the $i$-th copy of $\GL_2$. We have
\begin{align*}
 &u_{-\alpha_1}(a)\eta_0 =(\begin{pmatrix} \frac{1}{1-a} & 0 \\
 0 & 1 \end{pmatrix},\begin{pmatrix}1 & \frac{-a}{1-a} \\
  0 & \frac{1}{1-a} \end{pmatrix},\begin{pmatrix}\frac{1}{1-a} & \frac{-a}{1-a} \\ 0 & 1 \end{pmatrix})\cdot \eta_0 \cdot \begin{pmatrix}1-a&0\\ a&1\end{pmatrix}^{\text{diag}}, \\
&u_{-\alpha_2}(b)\eta_0 =(\begin{pmatrix} 1 & -\frac{b}{1-b} \\
0 & \frac{1}{1-b} \\\end{pmatrix}, \begin{pmatrix}\frac{1}{1-b} & 0 \\0 & 1 \\\end{pmatrix},\begin{pmatrix}\frac{1}{1-b} & 0 \\ 0 & 1 \end{pmatrix})\cdot \eta_0 \cdot \begin{pmatrix}1&b\\ 0&1-b\end{pmatrix}^{\text{diag}},\\
&u_{-\alpha_3}(c)\eta_0 =(\begin{pmatrix} 1 & 0 \\
0 & \frac{1}{1+c} \\\end{pmatrix},\begin{pmatrix}\frac{1}{1+c} & 0 \\0 & 1 \\\end{pmatrix},\begin{pmatrix}\frac{1}{1+c} & 0\\ 0 & 1 \end{pmatrix})\cdot \eta_0 \cdot \begin{pmatrix}1&0\\ 0&1+c\end{pmatrix}^{\text{diag}}.
\end{align*} 
This proves \eqref{main identity} and implies that (note that the representation has trivial central character)

$$\beta_{\alpha_1}^{\vee}=e_1+e_2'+e_1'',\;\alpha_{1}^{\vee}-\beta_{\alpha_1}^{\vee}=e_1+e_1'+e_2'',\beta_{\alpha_2}^{\vee}=e_2+e_1'+e_1'',$$
$$\alpha_{2}^{\vee}-\beta_{\alpha_2}^{\vee}=e_1+e_1'+e_2'',\beta_{\alpha_3}^{\vee}=e_2+e_1'+e_1'',\;\alpha_{3}^{\vee}-\beta_{\alpha_3}^{\vee}=e_1+e_2'+e_1''.$$

% {\it Step \eqref{item:reductive-3}.}
Then $\Theta^+$ will be the smallest subset of $\Theta$ satisfying the following two conditions:
\begin{itemize}
\item $e_1+e_1'+e_2'',e_1+e_2'+e_1'',e_2+e_1'+e_1''\in \Theta^+$.
\item $\Theta^+-w_{\alpha_1}\Theta^+=\{e_1+e_1'+e_2'',e_1+e_2'+e_1''\},\;\Theta^+-w_{\alpha_2}\Theta^+=\{e_1+e_1'+e_2'',e_2+e_1'+e_1''\},\;\Theta^+-w_{\alpha_3}\Theta^+=\{e_1+e_2'+e_1'',e_2+e_1'+e_1''\}.$
\end{itemize}
As a result, we have
$$\Theta^+=\{e_1+e_1'+e_1'',e_1+e_1'+e_2'',e_1+e_2'+e_1'',e_2+e_1'+e_1''\}.$$
It is easy to see that $\Theta^+$ satisfies \eqref{theta plus}.

% {\it Step \eqref{item:reductive-4}.}
Finally, if we let $\theta=\delta_{B}^{1/2}$, it is easy to see that for $w\in W$, $c_{WS}(w\theta)=0$ unless $w$ is the longest Weyl element. If $w$ is the longest Weyl element, we have $c_{WS}(w\theta)=1-q^{-2}=\zeta(2)^{-1}=\frac{1}{\Delta_{H/Z_{G,H}}(1)}$. This proves Lemma \ref{lem eta}. In conclusion, we have proved that the local relative character in this case is equal to 
$$\zeta(1)^3\zeta(2)\cdot \frac{L(1/2,\pi_1\times \pi_2\times \pi_2)}{L(1,\pi,\Ad)}$$
where $\pi=\pi_1\otimes \pi_2\otimes \pi_3$ is an unramified representation of $G(F)$.

\subsection{The Whittaker induced case: some reductions}\label{sec non-red strategy 1}
In this subsection, we consider the Whittaker induced case. Let $(G,H)$ be a Whittaker induction of a strongly tempered model $(G_0,H_0)$. In other words, there exists a parabolic subgroup $P=LU$ of $G$ and a generic character $\xi$ of $U(F)$ such that $G_0\simeq L$ and $H_0$ is the neutral component of the stabilizer of $\xi$ in $M$. Note that for all the cases we considered in Table \ref{fig:1}, the model $(G_0,H_0)$ is essentially the trilinear $\GL_2$-model.

Let $B_0=TN_0$ be a Borel subgroup of $G_0$, $\bar{B}_0=T\bar{N}_0$ be its opposite, and $\bar{P}=L\bar{U}$ be the opposite parabolic subgroup of $P$. Let $N=N_0U$ and $\bar{N}=\bar{N}_0 \bar{U}$. Then $B=TN$ is a Borel subgroup of $G$ and $\bar{B}=T\bar{N}$ is its opposite.

For all of our cases (as well as all the other Whittaker induced cases in Remark \ref{rmk non-red all cases}), there exists a Weyl element $w_0$ such that the $w_0$-conjugation map
\begin{itemize}
\item induces an isomorphism between $U$ and $\bar{U}$,
\item stabilizes $L$ and fixes $H_0\subset L$.
\end{itemize}
Also there exists a homomorphism $\lambda:U(F)\rightarrow F$ such that $\xi(u)=\psi(\lambda(u))$ for all $u\in U(F)$. We extend $\lambda$ to $H(F)$ by making it trivial on $H_0(F)$. We also have a map $a:\GL_1\rightarrow Z_L$ such that 
\begin{equation}\label{the map a} w_{0}^{-1}a(t)w_{0}=a(t)^{-1},
\text{ and }
\lambda(a(t)ua(t)^{-1})=t\lambda(u), \text{ for } t\in F^{\times}, u\in U(F).
\end{equation}

Let $B_0(F)\eta_0 H_0(F)$ be the unique open Borel orbit of the model $(G_0,H_0)$, and let $\eta=\eta_0 w_0$. Then $B(F)\eta H(F)$ is the unique open Borel orbit of the model $(G,H)$ and the stabilizer of this orbit is $Z_{G,H}=H\cap Z_G$. 
Note that we always assume $(G_0,H_0)$ does not have Type $N$ spherical root. 
The equation \eqref{the map a} implies that $\eta^{-1}a(t)\eta=a(t)^{-1}$.

We want to compute the local relative character
\begin{equation}\label{relative character non-red}
I(\phi_\theta)=\int_{H_0(F)/Z_{G,H}(F)}\int_{U(F)}\phi_\theta(hu)\xi(u)^{-1}\ud u\ud h
\end{equation}
where $\phi_\theta$ is the unramified matrix coefficient of $I_{B}^{G}(\theta)$ with $\phi_\theta(1)=1$, and $\theta$ is a unitary unramified character of $T(F)$.  
The general idea of the computation is the same as the reductive case, the only difference is some convergent issue.
Unlike the reductive case, the integral above is not absolutely convergent because of the extra unipotent integral. 
Hence we need to regularize the unipotent integral. 

There are three (equivalent) ways to regularize the unipotent integral. The first one is using the fact the the unipotent integral is stable, i.e. there exists a compact open subgroup $\CU$ of $U(F)$ such that $$\int_{\CU'-\CU} \phi_\theta(hu)\xi(u)^{-1}\ud u=0$$
for all compact open subgroup $\CU'$ of $\CU(F)$ with $\CU\subset \CU'$. Hence we can replace the integral over $U(F)$ by an integral over the compact subgroup $\CU$. 
This regularization has been used by Lapid--Mao \cite{LM} in their computation for the Whittaker model, and used by Liu \cite{L} in his computation of the non-reductive Gan--Gross--Prasad models. 
The advantage of this regularization is that it works for general matrix coefficients, not just the tempered ones. 
(Of course, in order for the integral of $H_0(F)$ to be convergent, we still need the character $\theta$ to be close to the unitary line.)

The rest two regularizations are only for the tempered case. It uses a critical fact that although the integral \eqref{relative character non-red} is not absolutely convergent, the integral will become absolutely convergent if we replace $U(F)$ by $U'(F)=\{u\in U(F)\mid \lambda(u)=0\}$. (For all the models considered in this paper, this can be proved by the same argument as Lemma 4.3.1 of \cite{Wan17b}.) 
As a result, we only need to regularize the integral over $\lambda(u)\in F$.

\begin{rmk}\label{rmk convergent}
In particular, the integral 
$$\int_{H(F)/Z_{G,H}(F)}\phi_\theta(h)\Phi(\lambda(h))\ud h$$
is absolutely convergent for all $\Phi\in C_{c}^{\infty}(F)$.
\end{rmk}

There are two ways to regularize the integral over $\lambda(u)$. The first way is to replace the the integral over $U(F)$ by the integral over 
$$U_n(F)=\{u\in U(F)\mid |\lambda(u)|\leq q^n\}.$$
Then one can show that for every matrix coefficient $\phi$ of $I_{B}^{G}(\theta)$, there exists $N>0$ (depends on the level of $\phi$, i.e. the open compact subgroup $K'\subset G(F)$ where $\phi$ is bi-$K'$-invariant) such that the integral
$$I_n(\phi_\theta)=\int_{H_0(F)/Z_{G,H}(F)}\int_{U_n(F)}\phi_\theta(hu)\xi(u)^{-1}\ud u\ud h$$
is independent of $n$ for $n>N$, i.e. the unipotent integral is stable on the sequence $U_n(F)$ (the difference between this regularization and the previous one is that the unipotent groups $U_n(F)$ we used here is not compact). 
Hence we can just replace the integral over $U(F)$ in the definition of $I(\phi_\theta)$ by the integral over $U_n(F)$ for some large $n$ (in fact, as $\phi_\theta$ is unramified, one can easily show that we can just replace $U(F)$ by $U_1(F)$). 
This regularization has been used by Waldspurger in Lemma 5.1 of \cite{W12} for the orthogonal Gan--Gross--Prasad model. 
The same arguments work for all the Whittaker induction cases in this paper. 

Another way is to replace the character $\xi(u)^{-1}=\psi(\lambda(u))^{-1}$ by some Schwartz function $\varphi_n(\lambda(u))$ ($n\geq 0$) of $\lambda(u)$ where 
$$\varphi_0=\varphi=1_{\CO_F}-\frac{1}{q-1}\cdot 1_{\varpi^{-1}\CO_{F}^{\times}}$$ 
is the Fourier transform of $\frac{1}{vol(\CO_{F}^{\times})}\cdot 1_{\CO_{F}^{\times}}$ and $\varphi_n$ is the Fourier transform of the function $\frac{1}{vol(1+\varpi^n\CO_{F}^{\times})}1_{1+\varpi^n\CO_{F}^{\times}}$ for $n\geq 1$. One can show that for every matrix coefficient $\phi$ of $I_{B}^{G}(\theta)$, there exists $N>0$ (depends on the level of $\phi$) such that the integral
$$I^n(\phi_\theta)=\int_{H_0(F)/Z_{G,H}(F)}\int_{U_n(F)}\phi_\theta(hu)\xi(u)^{-1}\ud u\ud h$$
is independent of $n$ for $n>N$.

This regularization has been used by Beuzart-Plessis (Proposition 7.1.1 of \cite{B15}) for the unitary Gan--Gross--Prasad model (note that the group $1+\varpi^n\CO_{F}^{\times}$ is just the group $K_a$ in loc. cit.) and by the first author (Proposition 5.1 of \cite{Wan16}) for the Ginzburg--Rallis model. The same argument works for all the Whittaker induction cases in this paper.  In the unramified case, we may just take $n=0$ and the regularized integral is given by the formula
$$I(\phi_\theta)=\int_{H(F)/Z_{G,H}(F)}\phi_\theta(h)\varphi_0(\lambda(h))\ud h.$$

In order to compute this regularized integral, we need another two regularized integrals.

\begin{lem}\label{normalization of induced representation}
For $f\in I_{B}^{G}(\theta)$, the integrals
$$\int_{H_0(F)/Z_{G,H}(F)}\int_{U_n(F)} f(\eta hu)\xi(u)^{-1}\ud u\ud h$$
and 
$$\int_{H(F)}f(\eta h) \varphi_n(\lambda(h))\ud h$$
are absolutely convergent for all $n$. 

Moreover, there exists $N\geq 0$ (depends on the level of $f$, i.e. the open compact subgroup $K'\subset G(F)$ where $f$ is right $K'$-invariant) such that both integrals are equal to each other and are independent of $n$ for $n>N$. 
\end{lem}

\begin{defn}
We use $\int_{H(F)/Z_{G,H}(F)}^{\ast} f(\eta h)\xi(h)^{-1}\ud h$ to denote the regularized integral 
\begin{align*}
&\lim_{n\rightarrow \infty}\int_{H_0(F)/Z_{G,H}(F)}\int_{U_n(F)} f(\eta hu)\xi(u)^{-1}\ud u\ud h\\
 =&\lim_{n\rightarrow \infty}\int_{H(F)/Z_{G,H}(F)}f(\eta h) \varphi_n(\lambda(h))\ud h.
\end{align*}
\end{defn}

\begin{proof}
We first prove the convergence. By replacing $f$ by $|f|$ and $\theta$ by $|\theta|$ we may assume that $f$ is a non-negative real valued function. Let $f_\theta$ be the unramified vector in $I_{B}^{G}(\theta)$ with $f_\theta(1)=1$. Then the matrix coefficient of $f$ and $f_\theta$ is given by 
$$\phi_{f,f_\theta}(g)=\int_{K}f(kg)\ud g.$$
By the discussion above, we know that the integral
$$\int_{H_0(F)/Z_{G,H}(F)}\int_{U_n(F)}\phi_{f,f_{\theta}}(hu)\xi(u)^{-1}\ud u\ud h$$
is absolutely convergent for all $n$. This implies that (note that $f$ is a non-negative real valued function) the triple integral 
$$\int_{H_0(F)/Z_{G,H}(F)}\int_{U_n(F)}\int_K f(khu)\xi(u)^{-1}\ud k\ud u\ud h$$
is absolutely convergent. In particular, the integral
$$\int_{H_0(F)/Z_{G,H}(F)}\int_{U_n(F)}f(khu)\xi(u)^{-1}\ud u\ud h$$
is absolutely convergent for almost all $k\in K$. 
But as a function on $k\in K$, this integral is left $B(\CO_F)$ and right $H(\CO_F)$ invariant. Combining with the fact that $\eta\in K$ and $B\eta H$ is Zariski open in $G$, we know that the integral 
$$\int_{H_0(F)/Z_{G,H}(F)}\int_{U_n(F)} f(\eta hu)\xi(u)^{-1}\ud u\ud h$$
is absolutely convergent for all $n$. 
This also implies that the integral 
$$\int_{H(F)}f(\eta h) \varphi_n(\lambda(h))\ud h$$
is absolutely convergent for all $n$ since $\varphi_n$ is a compactly supported function.

Now we prove the second part of the theorem. We first prove the following statement
\begin{itemize}
\item[(1)] there exists $N\geq 0$ such that for all $n>N$, we have
$$\int_{H_0(F)/Z_{G,H}(F)}\int_{U_1(F)} f(\eta hu_0u)\xi(u)^{-1}\ud u\ud h=0$$
for all $u_0\in U(F)-U_n(F)$.
\end{itemize}

In fact, for $n$ large, we have the function $f(g)$ is right $a(t)$-invariant for all $t\in 1+\varpi^n\CO_{F}^{\times}$. It is also left $a(t)$-invariant for all $t\in 1+\varpi^n\CO_{F}^{\times}$ since $\theta$ is an unramified character. Then for $u_0\in U(F)-U_n(F)$, there exists $t_0\in 1+\varpi^n\CO_{F}^{\times}$ such that $(t_0-1)\lambda(u_0)\in \varpi^{-1}\CO_{F}^{\times}$ (in particular, $\psi((t_0-1)\lambda(u_0))\neq 1$). This implies that 
$$\int_{H_0(F)/Z_{G,H}(F)}\int_{U_1(F)} f(\eta hu_0u)\xi(u)^{-1}\ud u\ud h$$
$$=\int_{H_0(F)/Z_{G,H}(F)}\int_{U_1(F)} f(a(t_{0}^{-1})\eta hu_0u)\xi(u)^{-1}\ud u\ud h$$
\begin{align*}
=&\int_{H_0(F)/Z_{G,H}(F)}\int_{U_1(F)} f(\eta hu_0(u_{0}^{-1}a(t_0)u_0)ua(t_0)^{-1})\xi(u)^{-1}\ud u\ud h\\
=&\int_{H_0(F)/Z_{G,H}(F)}\int_{U_1(F)} f(\eta hu_0(u_{0}^{-1}a(t_0)u_0a(t_0)^{-1})u)\xi(u)^{-1}\ud u\ud h.
\end{align*}
Here we use the fact that since $\psi$ is unramified, $\xi(u)=\xi(a(t)ua(t)^{-1})$ for $u\in U_1(F)$ and $t\in 1+\varpi^n\CO_F$. By our choice of $u_0$ and $t_0$, we know that $u_{0}^{-1}a(t_0)u_0a(t_0)^{-1}\in U_1(F)$, then an easy change of variable shows that $\int_{H_0(F)/Z_{G,H}(F)}\int_{U_1(F)} f(\eta hu_0u)\xi(u)^{-1}\ud u\ud h$ is equal to
$$\xi(u_{0}^{-1}a(t_0)u_0a(t_0)^{-1})\int_{H_0(F)/Z_{G,H}(F)}\int_{U_1(F)} f(\eta hu_0u)\xi(u)^{-1}\ud u\ud h.$$
Since $\xi(u_{0}^{-1}a(t_0)u_0a(t_0)^{-1})=\psi((t_0-1)\lambda(u_0))\neq 1$, we have
$$\int_{H_0(F)/Z_{G,H}(F)}\int_{U_1(F)} f(\eta hu_0u)\xi(u)^{-1}\ud u\ud h=0.$$
This proves (1).

Now we are ready to prove the theorem. By (1), we know that for all $n>N$, we have
$$\int_{H_0(F)/Z_{G,H}(F)}\int_{U_{n+1}(F)-U_n(F)} f(\eta hu)\xi(u)^{-1}\ud u\ud h=0.$$
In particular, the integral
$$\int_{H_0(F)/Z_{G,H}(F)}\int_{U_n(F)} f(\eta hu)\xi(u)^{-1}\ud u\ud h$$
is independent of $n$ for $n>N$.

For the second integral, choose $n$ large so that the function $f(g)$ is right $a(t)$-invariant for all $t\in 1+\varpi^n\CO_{F}^{\times}$. Then we have
\begin{align*}
 & \int_{H_0(F)/Z_{G,H}(F)}\int_{U_{n}(F)} f(\eta hu)\xi(u)^{-1}\ud u\ud h\\
=&\int_{1+\varpi^n\CO_{F}^{\times}} \int_{H_0(F)/Z_{G,H}(F)}\int_{U_{n}(F)} f(a(t)\eta hua(t))\xi(u)^{-1}\ud u\ud h \ud t\\
=&\int_{1+\varpi^n\CO_{F}^{\times}} \int_{H_0(F)/Z_{G,H}(F)}\int_{U_{n}(F)} f(\eta ha(t)^{-1}ua(t))\xi(u)^{-1}\ud u\ud h \ud t\\
=&\int_{1+\varpi^n\CO_{F}^{\times}}\int_{H_0(F)/Z_{G,H}(F)}\int_{U_{n}(F)} f(\eta hu)\psi(t\lambda(u))^{-1}\ud u\ud h\ud t\\
=&\int_{H_0(F)/Z_{G,H}(F)}\int_{U(F)}\int_{1+\varpi^{n}\CO_{F}^{\times}} f(\eta hu) \psi(t\lambda(u))^{-1} 1_{\varpi^{-n}\CO_F}(\lambda(u))\ud t\ud u\ud h. 
\end{align*} 
Here the measure ${\rm d}t$ on $1+\varpi^n\CO_{F}^{\times}$ is chosen so that the total volume is equal to 1. The function 
$$x\mapsto \int_{1+\varpi^{n}\CO_{F}^{\times}}  \psi(tx)^{-1} 1_{\varpi^{-n}\CO_F}(x)\ud t$$
$$=1_{\varpi^{-n}\CO_F}(x)\cdot \int_{F}\frac{1}{vol(1+\varpi^n\CO_{F}^{\times})}1_{1+\varpi^n\CO_{F}^{\times}}(t) \psi(tx)^{-1}\ud y$$
is just $1_{\varpi^{-n}\CO_F}\cdot \varphi_n$  (recall that $\varphi_n$ is the Fourier transform of the function $\frac{1}{vol(1+\varpi^n\CO_{F}^{\times})}1_{1+\varpi^n\CO_{F}^{\times}}$). A direct computation shows that the function $\varphi_n$ is supported on $\varpi^{-n}\CO_F$, hence $1_{\varpi^{-n}\CO_F}\cdot \varphi_n=\varphi_n$. As a result, we have
$$\int_{H_0(F)/Z_{G,H}(F)}\int_{U_{n}(F)} f(\eta hu)\xi(u)^{-1}\ud u\ud h=\int_{H(F)}f(\eta h) \varphi_n(\lambda(h))\ud h.$$
This proves the lemma.
\end{proof}

\begin{rmk}\label{rmk unramified 1}
As long as $f$ is right $T(\CO_F)$-invariant (for example when $f$ is unramified or when $f$ is an Iwahori fixed vector), we can just take $N=0$. We can also show that the integral 
$$\int_{H(F)}f(\eta h) \varphi_n(\lambda(h))\ud h$$
is independent of $n$ for $n\geq 0$.
\end{rmk}

Let $\CY_{\theta,\xi}$, $\CY_{\theta,\xi,n}$, $\CY_{\theta,\xi}^{n}$ be the function on $G(F)$ satisfying the following conditions:
\begin{itemize}
\item $\CY_{\theta,\xi}$, $\CY_{\theta,\xi,n}$, $\CY_{\theta,\xi}^{n}$ are supported on the open orbit $B(F)\eta H(F)$. 
\item For $b\in B(F)$ and $h\in H(F)$, we have
$$\CY_{\theta,\xi}(b\eta h)=\theta^{-1}\delta_{B}^{1/2}(b) \xi(h),\; \CY_{\theta,\xi,n}(b\eta h)=\theta^{-1}\delta_{B}^{1/2}(b) \xi(h)1_{\varpi^{-n}\CO_F}(\lambda(h)),$$ 
$$\CY_{\theta,\xi}^{n}(b\eta h)=\theta^{-1}\delta_{B}^{1/2}(b) \varphi_n(\lambda(h)).$$  .
\end{itemize}

\begin{lem}\label{normalization of Y}
For $\Phi\in C_{c}^{\infty}(G(F))$, the integrals
$$\int_{G(F)}\Phi(g)\CY_{\theta,\xi,n}(g)\ud g\text{ and } 
\int_{G(F)}\Phi(g)\CY_{\theta,\xi}^{n}(g)\ud g$$
are absolutely convergent for all $n$. Moreover, there exists $N\geq 0$ (depends on the level of $f$, i.e. the open compact subgroup $K'\subset G(F)$ where $f$ is right $K'$-invariant) such that both integrals are equal to each other and are independent of $n$ for $n>N$. 
\end{lem}

\begin{defn}
We use $\int_{G(F)}^{\ast} \Phi(g)\CY_{\theta,\xi}(g)\ud g$ to denote the regularized integral
$$\lim_{n\rightarrow \infty}\int_{G(F)}\Phi(g)\CY_{\theta,\xi,n}(g)\ud g=\lim_{n\rightarrow \infty}\int_{G(F)}\Phi(g)\CY_{\theta,\xi}^{n}(g)\ud g.$$
\end{defn}

\begin{proof}
By the same arguments as Lemma \ref{lem Haar measure}, we can prove the following statement
\begin{itemize}
\item For $\Phi\in C^{\infty}_{c}(G(F))$, we have
\begin{equation}\label{decomposition of measure}
\int_{G(F)}\Phi(g)\ud g=\frac{\Delta_{G}(1)}{\Delta_{H_0/Z_{G,H}}(1)}\zeta(1)^{-rk(G)}\int_{H(F)/Z_{G,H}(F)}\int_{B(F)}\Phi(b\eta h)\ud b\ud h.
\end{equation}
\end{itemize}
This implies that the integral $\int_{G(F)} \Phi(g) \CY_{\theta,\xi,n}(g)\ud g$ is equal to the product of $\frac{\Delta_{G}(1)}{\Delta_{H_0/Z_{G,H}}(1)}\zeta(1)^{-rk(G)}$ with 
\begin{align*}
 & \int_{H_0(F)Z_{G,H}(F)}\int_{U_n(F)}\int_{B(F)}\Phi(b\eta hu)\theta^{-1}\delta^{1/2}(b)\xi(u)^{-1}\ud b\ud u\ud h\\
=&\int_{H_0(F)Z_{G,H}(F)}\int_{H_0(F)/Z_{G,H}(F)}\int_{U_n(F)} f_{\Phi}(\eta hu)\xi(u)^{-1}\ud u\ud h
\end{align*} 
and the integral $\int_{G(F)} \Phi(g) \CY_{\theta,\xi}^{n}(g)\ud g$ is equal to the product of 
$$\frac{\Delta_{G}(1)}{\Delta_{H_0/Z_{G,H}}(1)}\zeta(1)^{-rk(G)}$$ 
with
\begin{align*}
 &\int_{H(F)/Z_{G,H}(F)}\int_{B(F)}\Phi(b\eta h)\theta^{-1}\delta^{1/2}(b)\varphi_n(\lambda(h))\ud b\ud h\\
=&\int_{H_0(F)Z_{G,H}(F)}\int_{H(F)/Z_{G,H}(F)}f_\Phi(\eta h) \varphi_n(\lambda(h))\ud h.  
\end{align*} 
Then the lemma just follows from the lemma above.
\end{proof}

\begin{rmk}\label{rmk unramified 2}
If $\Phi$ is right $T(\CO_F)$-invariant, then we can just take $N=0$. We can also show that the integral
$$\int_{G(F)}\Phi(g)\CY_{\theta,\xi}^{n}(g)\ud g$$
is independent of $n$ for $n\geq 0$. For any open compact subset $K'\subset G(F)$, we let
$$\int_{K'}^{\ast} \CY_{\theta,\xi}(k)\ud k:=\int_{G(F)}^{\ast} 1_{K'}(g) \CY_{\theta,\xi}(g)\ud g.$$
\end{rmk}

Recall that $f_\theta$ is the unramified vector in $I_{B}^{G}(\theta)$ with $f_\theta(1)=1$. We have
$$\phi_\theta(g)=\int_K f_\theta(kg)\ud k,\;I(\phi_\theta)=\int_{H(F)/Z_{G,H}(F)} \phi_\theta(h)\varphi(\lambda(h))\ud h$$
$$=\int_{H_0(F)/Z_{G,H}(F)}\int_{U_1(F)}\phi_\theta(hu)\xi(u)^{-1}\ud u.$$
This implies that 
\begin{equation}\label{I phi non-red}
I(\phi_\theta)=\int_{K}\int_{H(F)/Z_{G,H}(F)} f_\theta(kh)\varphi(\lambda(h)) \ud h\ud k
\end{equation}
$$=\int_{K}\int_{H_0(F)/Z_{G,H}(F)}\int_{U_1(F)} f_\theta(khu)\xi(u)^{-1}\ud u\ud h\ud k.$$
The next lemma follows from the proof of Lemma \ref{normalization of induced representation}.

\begin{lem}
For $u_0\in U(F)-U_1(F)$, we have
$$\int_{H_0(F)/Z_{G,H}(F)}\int_{U_1(F)} f_\theta(\eta u_0 hu)\xi(u)^{-1}\ud u\ud h=0.$$
\end{lem}

\begin{cor}
We have
$$I(\phi_\theta)=\frac{\Delta_{H_0/Z_{G,H}}(1)}{\Delta_{G}(1)}\zeta(1)^{rk(G)}\cdot \int_{K}^{\ast} \CY_{\theta^{-1},\xi}(k)\ud k\cdot \int_{K}^{\ast} \CY_{\theta,\xi^{-1}}(k)\ud k.$$
\end{cor}

\begin{proof}
We have
$$I(\phi_\theta)=\int_{K\cap B(F)\eta H_0(F)U_1(F)}\int_{H_0(F)/Z_{G,H}(F)}\int_{U_1(F)} f_\theta(khu)\xi(u)^{-1}\ud u\ud h\ud k.$$
The function 
$$k\mapsto \int_{H_0(F)/Z_{G,H}(F)}\int_{U_1(F)} f_\theta(khu)\xi(u)^{-1}\ud u\ud h$$
on $K\cap B(F)\eta H_0(F)U_1(F)$ is a scalar of the restriction of the function $\CY_{\theta^{-1},\xi}$ to $K\cap B(F)\eta H_0(F)U_1(F)$ and the scalar is equal to 
$$\int_{H_0(F)/Z_{G,H}(F)}\int_{U_1(F)} f_\theta(\eta hu)\xi(u)^{-1}\ud u\ud h.$$
This implies that 
\begin{eqnarray*}
I(\phi_\theta)&=&\int_{H_0(F)/Z_{G,H}(F)}\int_{U_1(F)} f_\theta(\eta hu)\xi(u)^{-1}\ud u\ud h \\
&&\cdot \int_{K\cap B(F)\eta H_0(F)U_1(F)}\CY_{\theta^{-1},\xi}(k)\ud k.
\end{eqnarray*}

By Lemma \ref{normalization of Y} and Remark \ref{rmk unramified 2}, we have 
$$\int_{K\cap B(F)\eta H_0(F)U_1(F)}\CY_{\theta^{-1},\xi}(k)\ud k=\int_{K}^{\ast} \CY_{\theta^{-1},\xi}(k)\ud k.$$
Hence it remains to show that 
\begin{equation}\label{regularization eq 1}
\frac{\Delta_{H_0/Z_{G,H}}(1)}{\Delta_{G}(1)}\zeta(1)^{rk(G)}\cdot \int_{K}^{\ast} \CY_{\theta,\xi^{-1}}(k)\ud k
\end{equation}
$$=\int_{H_0(F)/Z_{G,H}(F)}\int_{U_1(F)} f_\theta(\eta hu)\xi(u)^{-1}\ud u\ud h.$$

By Lemma \ref{normalization of Y}, Remark \ref{rmk unramified 2} and \eqref{decomposition of measure}, we have 
\begin{align*}
 &\frac{\Delta_{H_0/Z_{G,H}}(1)}{\Delta_{G}(1)}\zeta(1)^{rk(G)}\cdot \int_{K}^{\ast} \CY_{\theta,\xi^{-1}}(k)\ud k\\
 =&\frac{\Delta_{H_0/Z_{G,H}}(1)}{\Delta_{G}(1)}\zeta(1)^{rk(G)}\cdot \int_{K\cap B(F)\eta H_0(F)U_1(F)} \CY_{\theta,\xi^{-1}}(k)\ud k\\
=&\int_{H_0(F)}\int_{U_1(F)}\int_{B(F)} 1_K(b\eta hu)\theta^{-1}\delta_{B}^{1/2}(b) \xi(u)^{-1}du\\
=&\int_{H_0(F)/Z_{G,H}(F)}\int_{U_1(F)} f_\theta(\eta hu)\xi(u)^{-1}\ud u \ud h.  
\end{align*} 
This proves \eqref{regularization eq 1} and finishes the proof of the lemma.
\end{proof}

In the next subsection, we will explain how to compute the regularized integral $\int_{K}^{\ast} \CY_{\theta,\xi}(k)\ud k$. The result is summarized in the proposition below.

\begin{prop}\label{main prop nonred}
Let $\Phi^+$ be the set of positive roots of $G$. There is a decomposition of weights of a representation $\rho_X$ of $\hat{G}$ (denoted by $\Theta=\Theta^+\cup \Theta^-$) such that
$$\int_{K}^{\ast}\CY_{\theta,\xi}(k) \ud k=\frac{\Delta_{G}(1)}{\Delta_{H_0/Z_{G,H}}(1)}\zeta(1)^{-rk(G)}\cdot \beta(\theta)$$
where 
$$\beta(\theta)=\frac{\prod_{\alpha\in \Phi^+}1-q^{-1}e^{\alpha^\vee}}{
	\prod_{\gamma^\vee\in \Theta^+}1-q^{-\frac{1}{2}}e^{\gamma^\vee}}(\theta).$$
Also \eqref{theta plus} still holds.
\end{prop}

The proposition above implies that $I(\phi_\theta)$ is equal to
$$\frac{\Delta_{G}(1)}{\Delta_{H_0/Z_{G,H}}(1)}\zeta(1)^{-rk(G)}\cdot \beta(\theta)\cdot \beta(\theta^{-1})=\frac{\Delta_{G}(1)}{\Delta_{H_0/Z_{G,H}}(1)}\cdot \frac{L(1/2,\pi,\rho_X)}{L(1,\pi,\Ad)}.$$
This finishes the computation.

\subsection{The computation of $\int_{K}^{\ast} \CY_{\theta,\xi}(k)\ud k$}\label{sec non-red strategy 2}
Let
$$S_\theta(g):=\int_{K}^{\ast} \CY_{\theta,\xi}(kg^{-1})\ud k=\int_{G(F)}^{\ast} 1_K(g'g) \CY_{\theta,\xi}(g')\ud g'.  $$
Our goal is to prove Proposition \ref{main prop nonred}, i.e. compute $$S_\theta(1)=\int_{K}^{\ast} \CY_{\theta,\xi}(k)\ud k=\int_{K} \CY_{\theta,\xi}^{0}(k)\ud k.$$

Let $\CI=B(\CO_F)\bar{N}(\varpi \CO_F)$ (resp. $\CI_0=B_0(\CO_F)\bar{N}_0(\varpi \CO_F)$) be the Iwahori subgroup of $G(F)$ (resp. $G_0(F)=L(F)$). As in Lemma \ref{lem eta}, we can choose $\eta_0$ so that 
\begin{equation}\label{eta0 relation}
\bar{N}_0(\varpi \CO_F)\eta_0\subset T(\CO_F)N_0(\varpi \CO_F)\eta_0 H_0(\CO_F).
\end{equation}

\begin{lem}\label{lem eta nonred}
We have $\eta\in K$ and 
$$\bar{N}(\varpi \CO_F)\eta\subset T(\CO_F)N(\varpi \CO_F)\eta H_0(\CO_F)U(\varpi \CO_F).$$
\end{lem}

\begin{proof}
Since $\eta_0,w_0\in K$, we have $\eta=\eta_0w_0\in K$. For $\bar{n}\in \bar{N}(\varpi \CO_F)$, we write it as $\bar{n}'\bar{u}$ with $\bar{n}'\in \bar{N}_0(\varpi \CO_F)$ and $\bar{u}\in \bar{U}(\varpi \CO_F)$. Since $\eta=\eta_0 w_0$ and  $\eta_0\in L(\CO_F)$, we have $\eta^{-1}\bar{u}\eta\in U(\varpi \CO_F)$. Hence it is enough to consider the case when $\bar{n}\in \bar{N}_0(\varpi \CO_F)$. Then the lemma follows from \eqref{eta0 relation} and the fact that $H_0$ commutes with $w_0$.
\end{proof}

For $w\in W$, let $\Phi_w=1_{\CI w \CI}$. Let $\alpha$ be a simple root and $w_\alpha$ be the corresponding element in $W$. As in the reductive case, we would need to compute
$$I_\alpha(\theta)=vol(\CI)^{-1}\int_{G(F)}^{\ast}\CY_{\theta,\xi}(x )(\Phi_1(x\eta^{-1})+\Phi_{w_\alpha}(x\eta^{-1}))\ud x $$
$$=vol(\CI)^{-1}\int_{G(F)} \CY_{\theta,\xi}^{0}(x\eta )(\Phi_1(x)+\Phi_{w_\alpha}(x))\ud x.$$

First, by the lemma above, we have $\CI \eta\subset B(\CO_F)\eta H_0(\CO_F)U(\varpi \CO_F)$. Hence $\CY_{\theta,\xi}^{0}(x\eta )=1$ for all $x\in \CI$. This implies that 
$$vol(\CI)^{-1}\int_{G(F)} \CY_{\theta,\xi}^{0}(x\eta )\Phi_1(x)\ud x=1.$$
The next lemma follows from the same arguments as in the reductive case.
\begin{lem}
Let $u_{\alpha}:F\rightarrow N(F)$ be the homomorphism whose image is the root space of $\alpha$ (the root space is one dimensional since we assume that the group is split). Then 
$$I_\alpha(\theta)=1+q\int_{\CO_F} (\theta^{-1}\delta^{\frac{1}{2}})(e^{\alpha^\vee}(a^{-1}))\CY_{\theta,\xi}^0(u_{-\alpha}(a^{-1})\eta )\ud a,	$$
\end{lem}

Then for each model, by an explicit matrix computation, we will show that for $\alpha\in \Delta(G_0)$, there exists $\beta_\alpha^{\vee}\in \Theta$ such that $-\beta_\alpha^{\vee}+\alpha^{\vee}\in \Theta$ and  
\begin{equation}\label{main identity nonred}
\CY_{\theta,\xi}^{0}(u_{-\alpha}(a^{-1})\eta h^{-1})=\varphi_0(\lambda(h))\cdot \theta(e^{\beta_{\alpha}^{\vee}}(1+a^{-1}))\cdot |1+a^{-1}|^{-1/2}.
\end{equation}
As in the reductive case, this will imply that
\begin{equation}\label{I alpha non-red 1}
\alpha= (q-1)\cdot \frac{1-q^{-1}e^{\alpha^{\vee}}(\theta)}{(1-q^{-1/2} e^{\beta_{\alpha}^{\vee}}(\theta))(1-q^{-1/2}e^{-\beta_{\alpha}^{\vee}+\alpha^{\vee}}(\theta))}.
\end{equation}

\begin{rmk}\label{color in non-red}
In fact, by our choice of $\eta=\eta_0w_0$, we only need to verify the identity for the reductive model $(G_0,H_0)$ and we know that $\beta_{a}^{\vee}$ is just the color associated to $\alpha$ for the reductive model $(G_0,H_0)$. For all of our cases in Table \ref{fig:1}, since it is induced from the trilinear $\GL_2$-model, we can just use the computations in Section \ref{section triple product}. 
\end{rmk}

On the other hand, if $\alpha\in \Delta(G)-\Delta(G_0)$, we will show that 
\begin{equation}\label{nonred root}
\CY_{\theta,\xi}^{0}(u_{-\alpha}(a^{-1})\eta )=\varphi_0(a^{-1}).
\end{equation}
This implies that 
\begin{equation}\label{I alpha non-red 2}
I_\alpha(\theta)=1+q\int_{\CO_F} \theta(e^{\alpha^{\vee}}(a))\cdot |a|^{-1}\cdot \varphi_0(a^{-1})da=q(1-q^{-1}e^{\alpha^{\vee}}(\theta)).
\end{equation}

\begin{rmk}
Note that if we don't regularize the unipotent integral, the integral we get here will be 
$$I_\alpha(\theta)=1+q\int_{\CO_F} \theta(e^{\alpha^{\vee}}(a))\cdot |a|^{-1}\cdot \psi(a^{-1})\ud a.$$
This is not absolutely convergent (which is also the reason why the original unipotent integral is not absolutely convergent). There are two ways to regularize this integral which correspond to the two ways to regularize the unipotent integral. 

The first way is to use the fact that $\int_{\varpi^n\CO_{F}^{\times}} \theta(e^{\alpha^{\vee}}(a))\cdot |a|^{-1}\cdot \psi(a^{-1})\ud a=0$ for $n>1$, and regularize the integral as 
$$I_\alpha(\theta)=1+q\int_{\CO_{F}^{\times}\cup \varpi^{-1}\CO_{F}^{\times}} \theta(e^{\alpha^{\vee}}(a))\cdot |a|^{-1}\cdot \psi(a^{-1})\ud a,$$
which is equal to $q(1-q^{-1}e^{\alpha^\vee}(\theta))$. The second way is to replace $\psi(a^{-1})$ by $\varphi_0(a^{-1})$ as we did above which gives the same answer.
\end{rmk}

\begin{defn}
Let $\Theta^+$ be the unique subset of $\Theta$ satisfying the following two conditions:
\begin{itemize}
\item For every simple root $\alpha\in \Delta_{G_0}$, we have $\Theta^+ - w_\alpha \Theta^+=\{\beta_{\alpha}^{\vee},\;\alpha^{\vee}-\beta_{\alpha}^{\vee}\}$;
\item For every simple root $\alpha\in \Delta_G-\Delta_{G_0}$, $\Theta^+$ is stable under $w_\alpha$.
\end{itemize}
We then define 
$$\beta(\theta)=\frac{\prod_{\alpha\in \Phi^+}1-q^{-1}e^{\alpha^\vee}}{
	\prod_{\gamma^\vee\in \Theta^+}1-q^{-\frac{1}{2}}e^{\gamma^\vee}}(\theta) 
\text{ and }
	 c_{WS}(\theta)=\frac{\prod_{\gamma^\vee\in \Theta^+}1-q^{-\frac{1}{2}}e^{\gamma^\vee}}
{\prod_{\alpha\in \Phi^+}1-e^{\alpha^\vee}}(\theta).$$
\end{defn}

Now by the exactly same arguments as in the reductive case (the only difference is that for the definition of $l_\theta$ in \eqref{eq:l-theta}, we replace the integral $\int_{G(F)}$ by the regularized integral $\int_{G(F)}^{\ast}$), we can prove the following proposition.

\begin{prop}\label{prop:WS-non-reductive}
Let $T(F)^+=\{t\in T(F)|\;t^{-1}N(\CO_F)t\subset N(\CO_F)\}$ be the positive chamber of $T(F)$. Then
\begin{align*}
S_\theta(\eta^{-1}t)/\beta(\theta)=&q^{l(W)}
vol(\CI)\sum_{w\in W}c_{WS}(w\theta)(w\theta)^{-1}\delta^{\frac{1}{2}}(t^{-1}), \text{ for } t\in T(F)^+,
\end{align*}
where $l(W)$ is the length of the longest Weyl element in $W$.
\end{prop}

Since $\eta\in K$, we have $S_\theta(1)=S_\theta(\eta^{-1})$. Combining with the proposition above, we have
$$S_\theta(1)/\beta(\theta)=q^{l(W)}vol(\CI)\sum_{w\in W}c_{WS}(w\theta).$$
Since $vol(\CI)=\Delta_G(1)\zeta(1)^{-rk(G)}\cdot q^{-l(W)}$, we have
$$S_\theta(1)/\beta(\theta)=\Delta_G(1)\zeta(1)^{-rk(G)}\sum_{w\in W}c_{WS}(w\theta).$$
Hence in order to prove Proposition \ref{main prop nonred}, it is enough to prove the following lemma.

\begin{lem}\label{lem constant nonreductive}
The summation $\sum_{w\in W}c_{WS}(w\theta)$ is independent of the choice of $\theta$ and is equal to $\frac{1}{\Delta_{H_0/Z_{G,H}}(1)}$.
\end{lem}

This lemma is much more difficult than the reductive case for two reasons. First, Theorem 7.2.1 of \cite{Sa} only works for the reductive case, so we can not use it to imply that the summation is a constant. Secondly, in the reductive case, if we set $\theta=\delta_{B}^{1/2}$, then the only nonvanishing term in the summation is the term corresponding to the longest Weyl element, and it will be equal to $\frac{1}{\Delta_{H_0/Z_{G,H}}(1)}$. But for all the non-reductive cases in Table \ref{fig:1}, this is not true and actually it is impossible to choose a $\theta$ so that all the terms in the summation are equal to 0 except one. 
We believe that  this is related to the fact that for the models in Table \ref{fig:1}, the Type $T$ spherical roots and Type $(U,\psi)$ spherical roots will sometimes interlace each other. For example, for the model $(\GL_6,\GL_2\ltimes U)$, the roots $\alpha_i=e_i-e_{i+1}$ is of Type $T$ when $i=1,3,5$ and is of Type $(U,\psi)$ when $i=2,4$.

As a result, for each of these cases, we will prove this lemma by a direct computation. Our computation is based on some reductive steps. For example, for the model $(\GSO_{12},\GL_2\ltimes U)$, we will prove the identity by proving another identity which allows us to reduce to the identity for the model $(\GL_6,\GL_2\ltimes U)$; for the model $(\GL_6,\GL_2\ltimes U)$, we will prove the identity by proving another identity which allows us to reduce to an identity related to the group $\GL_4\times \GL_2$.

\subsubsection{The summary}\label{sec:6-steps}
By the discussion in the previous two subsections, in order to compute the local relative character in the non-reductive case, we just need to do the following steps:

\begin{enumerate}
\item Define the Weyl element $w_0$ so that the $w_0$-conjugation map
\begin{itemize}
\item induces an isomorphism between $U$ and $\bar{U}$,
\item stabilizes $L$ and fixes $H_0\subset L$.
\end{itemize}
\item Define the map $a:\GL_1\rightarrow Z_L$ so that it satisfies \eqref{the map a}.
\item Show that the double coset $B_0(F)\back G_0(F)/H_0(F)$ has a unique open Borel $B_0(F)\eta_0 G_0(F)$ and the representative $\eta_0$ can be chosen to satisfy \eqref{eta0 relation}. Since all the cases in Table \ref{fig:1} are Whittaker inductions of the trilinear $\GL_2$-model, this step has already been verified in Section \ref{section triple product}.
\item Verify the identity \eqref{main identity nonred} and \eqref{nonred root} by expressing the product $u_{-\alpha}(a)\eta$ in terms of the decomposition $B(F)\eta H(F)$. Since all the cases in Table \ref{fig:1} are Whittaker induction of the trilinear $\GL_2$-model, the identity \eqref{main identity nonred} and the colors have already been computed in Section \ref{section triple product} (see Remark \ref{color in non-red}), so we only need to verify \eqref{nonred root}.
\item Compute the subset $\Theta^+$ of $\Theta$ and show that it satisfies \eqref{theta plus}.
\item Compute the constant $\sum_{w\in W}c_{WS}(w\theta)$, i.e. Lemma \ref{lem constant nonreductive}. This is the most technical part of the computation.
\end{enumerate}

{\bf A final remark for the spherical roots.}
In Table \ref{fig:1},
if a model is reductive, then all the simple roots of the spherical variety are of Type $T$, and our computation of $I_\alpha(\theta)$ in \eqref{I alpha red} confirms Statement 6.3.1 of \cite{Sa};
If a model is non-reductive,  the Whittaker induction of the trilinear model $(G_0,H_0)$, then for a simple root $\alpha$ of the spherical variety, $\alpha$ is of Type $T$ if $\alpha$ is a simple root of $G_0$ (recall that $G_0$ is embedded as the Levi subgroup of $G$) and the remaining simple roots  are of Type $(U,\psi)$. 
In such case, our computation of $I_\alpha(\theta)$ in \eqref{I alpha non-red 1} and \eqref{I alpha non-red 2} also confirms Statement 6.3.1 of \cite{Sa}.

\section{The model $(\GSp_6\times \GSp_4,(\GSp_4\times \GSp_2)^0)$}\label{sec:GSp6-model}
In this section, we compute the local relative character for the model $(\GSp_6\times \GSp_4,(\GSp_4\times \GSp_2)^0)$. 
We closely follow the four steps in Section \ref{red summary}. 
In Section \ref{sec:GSp-model}, we will define this model and verify Step \eqref{item:reductive-1}, i.e. there is only one open orbit under the action of the Borel subgroup.  
Then in Section \ref{sec:GSp-computation}, we will first study the matrix identities of the product $u_{-\alpha}(x)\eta$ to get the set of virtual weighted colors (Step \eqref{item:reductive-2}). Then we will compute the set $\Theta^+$ (Step \eqref{item:reductive-3}) and finally we will compute the constant $ \sum_{w\in W} c_{WS}(w\theta)$ (Step \eqref{item:reductive-4}).

\subsection{The model and some orbit computation}\label{sec:GSp-model}
% Let $$ with $w_n$ be the longest Weyl element in $\GL_n$. 
Define the split symplectic similitude group
$$
\GSp_{2n}=\{g\in \GL_{2n}\mid g^t J_{2n}g=l(g)J_{2n}\}$$
where  $J_{2n}= \begin{pmatrix}0&-w_n\\w_n&0\end{pmatrix}  \text{ and }
w_{n}= \begin{pmatrix}
&&1\\&\iddots&\\1&& 	
\end{pmatrix}.$ Here $l$ is the similitude character.  
Let $B_{2n}$  be the Borel subgroup of $\GSp_{2n}$ consisting of all upper triangular matrices. 
Set $G=\GSp_6\times \GSp_4$ and
$$
H=(\GSp_4\times \GSp_2)^0:=\{(h_1,h_2)\in \GSp_4\times \GSp_2 \mid l(h_1)=l(h_2)\}
$$ embeds into $G$ via the map
$$(h_1, \begin{pmatrix}a&b\\c&d\end{pmatrix})\in (\GSp_4\times \GSp_2)^0=H$$
$$\mapsto (\begin{pmatrix}a&0&b\\0&h_1&0\\c&0&d \end{pmatrix},h_1)\in \GSp_6\times \GSp_4=G.
 $$

For the non-split version of this model, let $D/F$ be a quaternion algebra. Let 
$$\GSp_{n}(D)=\{g\in \GL_n(D)\mid \bar{g}^tw_ng=l(g)w_n\}$$ 
where $\bar{g}$ is the conjugation map on $\GL_n(D)$ induced by the conjugation map on $D$. Let $G_D(F)=\GSp_3(D)\times \GSp_2(D)$ and $H_D(F)=(\GSp_2(D)\times \GSp_1(D))^0=\{(h_1,h_2)\in \GSp_2(D)\times \GSp_1(D)\mid l(h_1)=l(h_2)\}$ embeds into $G_D(F)$ via the map
$$(h_1, h_2)=(\begin{pmatrix}a&b\\c&d \end{pmatrix},h_2)\in (\GSp_2(D)\times \GSp_1(D))^0$$
$$\mapsto (\begin{pmatrix}a&0&b\\0&h_2&0\\c&0&d \end{pmatrix},h_1)\in \GSp_3(D)\times \GSp_2(D).$$

Set $\eta=\begin{pmatrix}1&0&0&0&0&0\\0&1&0&0&0&0\\0&0&1&0&0&0\\0&-1&-1&1&0&0 \\-1&0&-1&0&1&0\\0&-1&0&0&0&1  \end{pmatrix}$ 
 and 
$\eta^{-1}=\begin{pmatrix}1&0&0&0&0&0\\0&1&0&0&0&0\\0&0&1&0&0&0\\0&1&1&1&0&0 \\1&0&1&0&1&0\\0&1&0&0&0&1  \end{pmatrix}$.

\begin{prop}
The double cosets $B(F)\back G(F)/H(F)$ contain a unique open orbit $B(F)(\eta, I_4) H(F)$. Here $B(F)=B_6(F)\times B_4(F)$.
\end{prop}

Let $H'(F)=\{h_1\times h_2\in (\GSp_4\times \GSp_2)^0=H\mid  h_1\in B_4(F)\}$ be a subgroup of $\GSp_6(F)$. Let $X(F)=GSp_6(F)/B_6(F)$ be the flag variety associated to $\GSp_6(F)$. We have a natural action of $\GSp_6(F)$ on $X(F)$ which induces an action of $H'(F)$ on $X(F)$. 

Let $W_6=Span\{w,w_1,w_2,w_{2}^{\perp},w_{1}^{\perp},w^{\perp}\}$ be the six dimensional symplectic space defining $\GSp_6$ 
where $\{w,w_1,w_2,w_{2}^{\perp},w_{1}^{\perp},w^{\perp}\}$ is the standard basis induced by $B_6$, i.e. 
$$w=(1,0,0,0,0,0)^T,w_1=(0,1,0,0,0,0)^T,\cdots$$ 
Then $X(F)$ is characterized by 
$$X'(F)=\{(v_1,v_2,v_3)\mid \langle v_i,v_j \rangle = 0,\;v_1,v_2,v_3\;\text{are linearly independent}\}.$$
More specifically, $X(F)=\{Span\{v_1\},Span\{v_1,v_2\}, Span\{v_1,v_2,v_3\} \mid (v_1,v_2,v_3)\in X'(F)\}$.
The $\GSp_6(F)$-action is just 
$$g\cdot(v_1,v_2,v_3)=(gv_1,gv_2,gv_3).$$
Note that
$$\eta^{-1}\cdot (w,w_1,w_2)=(w+w_{1}^{\perp},w_1+w^{\perp}+w_{2}^{\perp},w_2+w_{1}^{\perp}+w_{2}^{\perp}).$$
Hence in order to prove the proposition, it is enough to prove the following lemma.

\begin{lem}
The $H'(F)$-action on $X(F)$ contains a unique open orbit represented by $(w+w_{1}^{\perp},w_1+w^{\perp}+w_{2}^{\perp},w_2+w_{1}^{\perp}+w_{2}^{\perp})$.
\end{lem}

\begin{proof}
First, we assume that $(v_1,v_2,v_3)$ belongs to the Zariski open subset such that 
\begin{itemize}
\item $v_1$ has nonzero projection to the subspaces $Span\{w,w^{\perp}\}$ and $Span\{w_{1}^{\perp}\}$;
\item The projections of $v_1$ and $v_2$ to $Span\{w,w^{\perp}\}$ and $Span\{w_{2}^{\perp},w_{1}^{\perp}\}$ are linearly independent;
\item The projections of $v_1,v_2$ and $v_3$ to $Span\{w_2,w_{2}^{\perp},w_{1}^{\perp}\}$ are linearly independent;
\item The projections of $v_1,v_2$ and $v_3$ to  $Span\{w,w^\perp,w_{2}^{\perp}\}$ are linearly independent.
\end{itemize}
Up to the $H'(F)$-action and because of the first condition we may assume that $v_1=w+w_{1}^{\perp}$. Then by the second condition and since $\langle v_1,v_2\rangle =0$, we may assume $v_2=w_1+w^{\perp}+aw_{2}^{\perp}+bw_2+cw_{1}^{\perp}$ with $a,b,c\in F$ and $a\neq 0$. Up to the action of an element 
$$diag(I_2,\begin{pmatrix}t&x\\0&t^{-1} \end{pmatrix},I_2)\in H'(F),$$
we may assume that $v_2=w_1+w^{\perp}+w_{2}^{\perp}+cw_{1}^{\perp}$. Note that such an element fixes $v_1$. Now let $h$ be the element in $H'(F)$ that fixes $w,w_1,w_2,w_{2}^{\perp},w_{1}^{\perp}$ and maps $w^{\perp}$ to $w^{\perp}+cw$. Then 
$$hv_1=v_1,\;hv_2=cv_1+(w_1+w^{\perp}+w_{2}^{\perp}).$$
Hence we may assume that $v_2=w_1+w^{\perp}+w_{2}^{\perp}$.

Finally, because $(v_1,v_2)=(w+w_{1}^{\perp},w_1+w^{\perp}+w_{2}^{\perp})$ and by the third condition, we may assume that $v_3=w_2+aw_{1}^{\perp}+bw_1+cw_{2}^{\perp}$. Since $\langle v_1,v_3 \rangle = \langle v_2,v_3 \rangle =0$, we have $a=1,b=0$. Hence $v_3=w_2+w_{1}^{\perp}+cw_{2}^{\perp}$. By the fourth condition, we know that $c\neq 0$.

Now consider the element $h_0=diag(1,c^{-1},1,c^{-1},1,c^{-1})\in H'(F)$. We have
$$h_0v_1=v_1,\;h_0v_2=c^{-1}v_2,\;h_0v_3=w_2+w_{1}^{\perp}+w_{2}^{\perp}.$$
This proves the lemma.
\end{proof}

Now if we let $\bar{N}_6$ (resp. $\bar{N}_4$) be the lower triangular unipotent subgroup of $\GSp_6$ (resp. $\GSp_4$) and we embed $\bar{N}_4$ into $\bar{N}_6$ via the embedding of $\GSp_4$ to $\GSp_6$. We also let $T_6$ (resp. $T_4$) be the diagonal torus of $\GSp_6$ (resp. $\GSp_4$). The following lemma is a direct consequence of the proof of the previous lemma.

\begin{lem}
For all $n\in \bar{N}_6(\varpi\CO_F)$ and $n'\in \bar{N}_{4}(\varpi\CO_F)$, we have
$$n\eta n'\in B_6(F) \eta N_6(\varpi\CO_F)T_6(\CO_F)N'(\varpi\CO_F)$$
where $N'$ is the lower triangular unipotent subgroup of $\GSp_2$ via as a subgroup of $\GSp_6$. 
\end{lem}

We need a stronger result.
\begin{lem}
For all $n\in \bar{N}_6(\varpi\CO_F)$ and $n'\in \bar{N}_{4}(\varpi\CO_F)$, we have
$$n\eta n'\in T_6(\CO_F)N_6(\varpi\CO_F) \eta N_6(\varpi\CO_F)T_6(\CO_F)N'(\varpi\CO_F).$$
\end{lem}

\begin{proof}
The above lemma implies that 
$$n\eta n'\in T_6(\CO_F)N_6(\CO_F) \eta N_6(\varpi\CO_F)T_6(\CO_F)N'(\varpi\CO_F).$$
We just need to prove the element in $N_6(\CO_F)$ actually belongs to $N_6(\varpi\CO_F)$, which is equivalent to show that this element preserves the sets $V_4,V_5,V_6$ where
$$V_i=\{(a_1,a_2,a_3,a_4,a_6,a_6)^T\mid a_i\in \CO_{F}^{\times},a_j\in \varpi\CO_F\;\text{for all}\;j\neq i\}.$$
But this just follows from the fact that these three sets are fixed by 
$$\eta,\eta^{-1}, T_6(\CO_F),\bar{N}_6(\varpi\CO_F)\bar{N}_{4}(\varpi\CO_F),N_6(\varpi\CO_F),N'(\varpi\CO_F).$$ 
This proves the lemma.
\end{proof}

\noindent
The above lemma implies the following proposition which is Lemma \ref{lem eta} for the current case.

\begin{prop}
For all $n\in \bar{N}_6(\varpi\CO_F)$ and $n'\in \bar{N}_{4}(\varpi\CO_F)$, we have
$$(n,n')(\eta,I_4)\in T(\CO_F)N(\varpi\CO_F)(\eta,I_4)H(\CO_F)$$
with $B=TN$.
\end{prop}

\subsection{The computation}\label{sec:GSp-computation}
Let $\alpha_1=\varepsilon_1-\varepsilon_2,\alpha_2=\varepsilon_2-\varepsilon_3,\alpha_3=2\varepsilon_3$ be the simple roots of $\GSp_6$ and $\alpha_1'=\varepsilon_1'-\varepsilon_2',\alpha_2'=2\varepsilon_2'$ be the simple roots of $\GSp_4$. We want to compute the virtual weighted colors associated to these simple roots. Set
$$u_{-\alpha_1}(x)=\begin{pmatrix}1&0&0&0&0&0\\x&1&0&0&0&0\\0&0&1&0&0&0\\0&0&0&1&0&0 \\0&0&0&0&1&0\\0&0&0&0&-x&1  \end{pmatrix},\;u_{-\alpha_2}(x)=\begin{pmatrix}1&0&0&0&0&0\\0&1&0&0&0&0\\0&x&1&0&0&0\\0&0&0&1&0&0 \\0&0&0&-x&1&0\\0&0&0&0&0&1  \end{pmatrix},$$
$$u_{-\alpha_3}(x)=\begin{pmatrix}1&0&0&0&0&0\\0&1&0&0&0&0\\0&0&1&0&0&0\\0&0&x&1&0&0 \\0&0&0&0&1&0\\0&0&0&0&0&1  \end{pmatrix}.$$

We also let $\Theta$ the weights of the 32-dimensional representation $\Spin_7\otimes \Spin_5$ of $\GSpin_7(\BC)\times \GSpin_5(\BC)$. We can write it as
$$\Theta=\{\frac{\pm e_1 \pm e_2 \pm e_3}{2}\}+ \{\frac{\pm e_1' \pm e_2'}{2}\}.$$

For $\alpha_1$, we have
\begin{equation} 
(u_{-\alpha_1}(x)\eta,I_4)=(b,h^{-1})\cdot (\eta,I_4)\cdot (g,h),\; (b,h^{-1})\in B(F), (g,h)\in H(F)
\end{equation}
where
$$g=\begin{pmatrix}1&0&0&0&0&0\\0&\frac{1}{1+x}&\frac{-x}{1+x}&0&\frac{x}{1+x}&0\\0&0&1&0&0&0\\0&0&0&\frac{1}{1+x}&\frac{x}{1+x}&0 \\0&0&0&0&1&0\\\frac{x}{1+x}&0&0&0&0&\frac{1}{1+x}  \end{pmatrix},h=\begin{pmatrix}\frac{1}{1+x}&\frac{-x}{1+x}&0&\frac{x}{1+x}\\0&1&0&0\\0&0&\frac{1}{1+x}&\frac{x}{1+x}\\0&0&0&1 \end{pmatrix},$$
$$ b=\begin{pmatrix}1&0&0&0&0&0\\0&x+1&0&0&-x&0\\0&0&1&0&0&0\\0&0&0&x+1&0&0 \\0&0&0&0&1&0\\0&0&0&0&0&x+1  \end{pmatrix}.$$
This implies that (recall that $\beta_{\alpha_1}^{\vee}$ is defined by the equation \eqref{main identity}, also note that the representation has trivial central character)
$$\beta_{\alpha_1}^{\vee}=\frac{e_1-e_2+e_3}{2}+ \frac{-e_1'+e_2'}{2},\;\alpha_{1}^{\vee}-\beta_{\alpha_1}^{\vee}=\frac{e_1-e_2-e_3}{2}+ \frac{e_1'-e_2'}{2}.$$

For $\alpha_2$, we have
have
\begin{equation} 
(u_{-\alpha_2}(x)\eta,I_4)=(b,h^{-1})\cdot (\eta,I_4)\cdot (g,h),\;(b,h^{-1})\in B(F), (g,h)\in H(F)
\end{equation}
where
$$g=\begin{pmatrix}\frac{1}{1-x}&0&0&0&0&\frac{-x}{1-x}\\0&1&0&0&0&0\\0&0&1&\frac{x}{1-x}&0&0\\0&0&0&\frac{1}{1-x}&0&0 \\0&0&0&0&\frac{1}{1-x}&0\\0&0&0&0&0&1  \end{pmatrix},h=\begin{pmatrix}1&0&0&0\\0&1&\frac{x}{1-x}&0\\0&0&\frac{1}{1-x}&0\\0&0&0&\frac{1}{1-x} \end{pmatrix}, $$
$$b=\begin{pmatrix}1-x&x&0&0&0&x\\0&1&0&0&0&0\\0&0&1-x&-x&0&0\\0&0&0&1&0&0 \\0&0&0&0&1-x&-x\\0&0&0&0&0&1  \end{pmatrix}.$$
This implies that 
$$\beta_{\alpha_2}^{\vee}=\frac{-e_1+e_2-e_3}{2}+ \frac{e_1'+e_2'}{2},\;\alpha_{2}^{\vee}-\beta_{\alpha_2}^{\vee}=\frac{e_1+e_2-e_3}{2}+ \frac{-e_1'-e_2'}{2}.$$

For $\alpha_3$, we have
\begin{equation} 
(u_{-\alpha_3}(x)\eta,I_4)=(g,h)\cdot (\eta,I_4)\cdot (g^{-1},h^{-1}),\;(g,h)\in B(F)\cap H(F)
\end{equation}
where $h=diag(1-x,1,1-x,1)$ and $g=diag(1,1-x,1,1-x,1,1-x)$. This implies that 
$$\beta_{\alpha_3}^{\vee}=\frac{e_1-e_2+e_3}{2}+ \frac{-e_1'+e_2'}{2},\; \alpha_{3}^{\vee}-\beta_{\alpha_3}^{\vee}=\frac{-e_1+e_2+e_3}{2}+ \frac{e_1'-e_2'}{2}.$$

For $\alpha_1'$, we can reduce to the root $\alpha_2$ (because the open orbit is represented by the element $(\eta,I_4)$) but we need to change $u_{-\alpha_2}(x)\eta$ to $\eta u_{-\alpha_2}(x)^{-1}=\eta u_{-\alpha_2}(-x)$. We have
\begin{equation}
(\eta u_{-\alpha_2}(-x),I_4)=(b,h^{-1})\cdot (\eta,I_4)\cdot (g,h),\;(b,h^{-1})\in B(F), (g,h)\in H(F)
\end{equation}
where
$$g=\begin{pmatrix}\frac{1}{1-x}&0&0&0&0&\frac{-x}{1-x}\\0&1&0&0&0&0\\0&0&\frac{1}{1-x}&\frac{-x}{1-x}&0&0\\0&0&0&1&0&0 \\0&0&0&0&\frac{1}{1-x}&0\\0&0&0&0&0&1  \end{pmatrix},h=\begin{pmatrix}1&0&0&0\\0&\frac{1}{1-x}&\frac{-x}{1-x}&0\\0&0&1&0\\0&0&0&\frac{1}{1-x} \end{pmatrix}, $$
$$b=\begin{pmatrix}1-x&x&0&0&0&x\\0&1&0&0&0&0\\0&0&1&x&0&0\\0&0&0&1-x&0&0 \\0&0&0&0&1-x&-x\\0&0&0&0&0&1  \end{pmatrix}.$$
This implies that 
$$\beta_{\alpha_1'}^{\vee}=\frac{-e_1+e_2+e_3}{2}+ \frac{e_1'-e_2'}{2},\;\alpha_{1}'^{\vee}-\beta_{\alpha_1'}^{\vee}=\frac{e_1-e_2-e_3}{2}+ \frac{e_1'-e_2'}{2}.$$

For $\alpha_2'$, we can reduce to the root $\alpha_3$ but we need to change $u_{-\alpha_3}(x)\eta$ to $\eta u_{-\alpha_3}(x)^{-1}=\eta u_{-\alpha_3}(-x)$. We have
\begin{equation}
(\eta u_{-\alpha_3}(-x),I_4)=(g,h)\cdot (\eta,I_4)\cdot (g^{-1},h^{-1}),\;(g,h)\in B(F)\cap H(F)
\end{equation}
where $h=diag(1+x,1,1+x,1)$ and $g=diag(1,1+x,1,1+x,1,1+x)$. This implies that 
$$\beta_{\alpha_2'}^{\vee}=\frac{e_1-e_2+e_3}{2}+ \frac{-e_1'+e_2'}{2},\; \alpha_{2}'^{\vee}-\beta_{\alpha_2'}^{\vee}=\frac{-e_1+e_2-e_3}{2}+ \frac{e_1'+e_2'}{2}.$$

\begin{rmk}\label{Luna diagram}
In this remark, we explain how to use the Luna diagram to compute the virtual colors. We recall the following Luna diagram of the model $(\GSp_6\times \GSp_4,(\GSp_4\times \GSp_2)^0)$ in Case (48) of \cite{BP}:
\begin{center}
\includegraphics[width=0.2\textwidth]{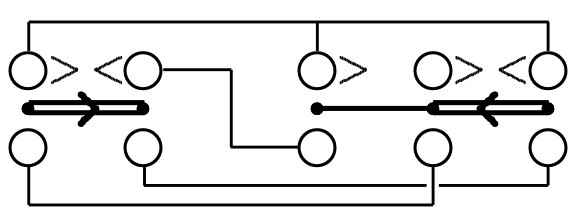} 
\end{center}
The middle row is the Dynkin diagram of $G$, from left to right we have the simple roots $\alpha_2',\alpha_1',\alpha_1,\alpha_2,\alpha_3$. For each simple root, there are two colors associated to it (represented by the two $\circ$ above and below the simple root). There is a line connecting two colors if and only if they are equal to each other. For $\alpha=\alpha_1,\alpha_3,\alpha_2'$ (resp. $\alpha=\alpha_2,\alpha_1'$), we use $\beta_{\alpha}^{\vee}$ to denote the color above (resp. below) $\alpha$ in the Luna diagram and we use $\alpha^\vee-\beta_{\alpha}^{\vee}$ to denote the color below (resp. above) $\alpha$ in the Luna diagram. The Luna diagram above implies that
$$\beta_{\alpha_1}^{\vee}=\beta_{\alpha_3}^{\vee}=\beta_{\alpha_2'}^{\vee},\;\beta_{\alpha_1'}^{\vee}=\alpha_{3}^{\vee}-\beta_{\alpha_3}^{\vee},\;\alpha_1'^\vee-\beta_{\alpha_1'}^{\vee}=\alpha_{1}^{\vee}-\beta_{\alpha_1}^{\vee},\;\alpha_{2}'^\vee-\beta_{\alpha_2'}^{\vee}=\beta_{\alpha_2}^{\vee}.$$
Combining the first three equations, we have
$$\alpha_1'^\vee=\beta_{\alpha_1'}^{\vee}+(\alpha_1'^\vee-\beta_{\alpha_1'}^{\vee})=(\alpha_{3}^{\vee}-\beta_{\alpha_3})+(\alpha_{1}^{\vee}-\beta_{\alpha_1}^{\vee})=\alpha_{1}^{\vee}+\alpha_{3}^{\vee}-2\beta_{\alpha_1}^{\vee}.$$
This implies that $$\beta_{\alpha_1}^{\vee}=\beta_{\alpha_3}^{\vee}=\beta_{\alpha_2'}^{\vee}=\frac{\alpha_{1}^{\vee}+\alpha_{3}^{\vee}-\alpha_1'^\vee}{2}=\frac{e_1-e_2+e_3}{2}+\frac{-e_1'+e_2'}{2}.$$ 
Combining with the last three equations, we have
$$\beta_{\alpha_1'}^{\vee}=\alpha_{3}^{\vee}-\beta_{\alpha_3}^{\vee}=\frac{-e_1+e_2+e_3}{2}+\frac{e_1'-e_2'}{2},$$
$$\alpha_1'^\vee-\beta_{\alpha_1'}^{\vee}=\alpha_{1}^{\vee}-\beta_{\alpha_1}^{\vee}=\frac{e_1-e_2-e_3}{2}+\frac{e_1'-e_2'}{2},$$
$$\alpha_{2}'^\vee-\beta_{\alpha_2'}^{\vee}=\beta_{\alpha_2}^{\vee}=\frac{-e_1+e_2-e_3}{2}+\frac{e_1'+e_2'}{2},$$
$$\alpha_{2}^{\vee}-\beta_{\alpha_2}^{\vee}=\frac{e_1+e_2-e_3}{2}+\frac{-e_1'-e_2'}{2}.$$
This recovers the above computation of colors using matrix identities.
\end{rmk}

\begin{prop}\label{theta proposition}
$\Theta^+$ is consisting of the following 16 elements:
\begin{equation}\label{gsp6 first line}
\frac{e_1+e_2\pm e_3}{2}+ \frac{\pm e_1' \pm e_2'}{2},\;\frac{e_1-e_2+e_3}{2}+ \frac{e_1'\pm e_2'}{2},\;\frac{e_1-e_2+e_3}{2}+ \frac{-e_1'+ e_2'}{2}, 
\end{equation}
$$\frac{\pm( e_1-e_2-e_3) }{2}+ \frac{e_1'\pm e_2'}{2},\;\frac{-e_1+e_2-e_3}{2}+ \frac{e_1'+e_2'}{2}.$$
\end{prop}

\begin{proof}
By the computations of the virtual weighted colors above, we know that $\Theta^+$ is the smallest subset of $\Theta$ satisfies the following 6 conditions:
\begin{enumerate}
\item $\{\frac{e_1+e_2-e_3}{2}+ \frac{-e_1'-e_2'}{2}\;\frac{e_1-e_2+e_3}{2}+ \frac{-e_1'+e_2'}{2},\;\frac{-e_1+e_2+e_3}{2}+ \frac{e_1'-e_2'}{2},\;\frac{e_1-e_2-e_3}{2}+ \frac{e_1'-e_2'}{2},\;\frac{-e_1+e_2-e_3}{2}+ \frac{e_1'+e_2'}{2}\}\subset \Theta^+$.

\item $\Theta^+-(\Theta^+\cap w_{\alpha_1}\Theta^+)=\{ \frac{e_1-e_2+e_3}{2}+ \frac{-e_1'+e_2'}{2},\frac{e_1-e_2-e_3}{2}+ \frac{e_1'-e_2'}{2}\}$.

\item $\Theta^+-(\Theta^+\cap w_{\alpha_2}\Theta^+)=\{\frac{e_1+e_2-e_3}{2}+ \frac{-e_1'-e_2'}{2},\;\frac{-e_1+e_2-e_3}{2}+ \frac{e_1'+e_2'}{2}\}$.

\item $\Theta^+-(\Theta^+\cap w_{\alpha_3}\Theta^+)=\{\frac{e_1-e_2+e_3}{2}+ \frac{-e_1'+e_2'}{2},\; \frac{-e_1+e_2+e_3}{2}+ \frac{e_1'-e_2'}{2}\}$.

\item $\Theta^+-(\Theta^+\cap w_{\alpha_1'}\Theta^+)=\{\frac{-e_1+e_2+e_3}{2}+ \frac{e_1'-e_2'}{2},\;\frac{e_1-e_2-e_3}{2}+ \frac{e_1'-e_2'}{2}\}$.

\item $\Theta^+-(\Theta^+\cap w_{\alpha_2'}\Theta^+)=\{\frac{e_1-e_2+e_3}{2}+ \frac{-e_1'+e_2'}{2},\;\frac{-e_1+e_2-e_3}{2}+ \frac{e_1'+e_2'}{2}\}$.
\end{enumerate}

It is clear that the set in the statement satisfies these conditions. So we just need to show that the set is the unique subset of $\Theta$ satisfying these 6 conditions. Let $\Theta'^+$ be another subset of $\Theta$ satisfies these 6 conditions. Then the set $\Theta^+\cap \Theta'^+$ also satisfies these 6 conditions. This implies that $\Theta^+-(\Theta^+\cap \Theta'^+)$ and $\Theta'^+-(\Theta^+\cap \Theta'^+)$ are $W$-invariant subsets of $\Theta$ (recall that $W$ is the Weyl group of $\hat{G}$). But the only $W$-invariant subsets of $\Theta$ are $\Theta$ and the empty set. Hence we must have $\Theta^+=\Theta'^+$. This proves the proposition.
\end{proof}

It is clear that $\Theta^+$ satisfies \eqref{theta plus}. So the last thing remains  to prove Lemma \ref{lem constant reductive} for the current case. 

\begin{lem}
With the notation above, we have 
$$\sum_{w\in W}c_{WS}(w\theta)=\frac{1}{\Delta_{H/Z_{G,H}}(1)}=\frac{1}{\zeta(2)^2\zeta(4)}=(1-q^{-2})^2(1-q^{-4}).$$
Recall that $c_{WS}(\theta)=\frac{\prod_{\gamma^\vee\in \Theta^+}1-q^{-\frac{1}{2}}e^{\gamma^\vee}}
{\prod_{\alpha\in \Phi^+}1-e^{\alpha^\vee}}(\theta)$.
\end{lem}

\begin{proof}
Since the summation is independent of $\theta$ (see Lemma \ref{lem constant reductive}), we set $\theta=\delta_{B}^{1/2}$. The lemma follows from the following two claims:
\begin{enumerate}
\item $c_{WS}(w\theta)$ is zero unless $w$ is the longest Weyl element.
\item If $w$ is the longest Weyl element, we have $c_{WS}(w\theta)=(1-q^{-2})^2(1-q^{-4})$.
\end{enumerate}
The second claim is easy to prove so we will focus on the first one. Let $w=(s,s')\in W$ so that $c_{WS}(w\theta)$ is nonzero. Here $s$ is a Weyl element of $\GSp_6$ and $s'$ is a Weyl element of $\GSp_4$. By abuse of language, we can also view $w=(s,s')$ as a Weyl element of the dual group $\GSpin_7(\BC)\times \GSpin_5(\BC)$. Then
\begin{equation}\label{GSp6 1}
e^{\gamma^\vee}(w\theta)=e^{w^{-1}\gamma^\vee}(\theta)\neq q^{1/2},\;\gamma^{\vee}\in \Theta^+.
\end{equation}

The values of the modular character $\delta_{B_6}^{1/2}$ on the weights 
$$\frac{e_1+e_2+e_3}{2},\frac{e_1+e_2-e_3}{2},\frac{e_1-e_2+e_3}{2},\frac{-e_1+e_2+e_3}{2},\frac{e_1-e_2-e_3}{2},$$
$$\frac{-e_1+e_2-e_3}{2},\frac{-e_1-e_2+e_3}{2},\frac{-e_1-e_2-e_3}{2}$$
are equal to
$$q^3,q^2,q,1,1,q^{-1},q^{-2},q^{-3}.$$
The values of the modular character $\delta_{B_4}^{1/2}$ on the weights 
$$\frac{e_1'+e_2'}{2},\frac{e_1'-e_2'}{2},\frac{-e_1'+e_2'}{2},\frac{-e_1'-e_2'}{2}$$
are equal to
$$q^{3/2},q^{1/2},q^{-1/2},q^{-3/2}.$$

Apply \eqref{GSp6 1} to the first eight weights in $\Theta^+$, we know that
$$s^{-1}(\frac{e_1+e_2\pm e_3}{2})\in\{\frac{e_1+e_2+e_3}{3},\frac{-e_1-e_2\pm e_3}{2}\}.$$
Since $s$ is a Weyl element, we must have
$s^{-1}(\frac{e_1+e_2\pm e_3}{2})=\{\frac{-e_1-e_2\pm e_3}{2}\}.$
This implies that
$$\{s^{-1}(e_1),s^{-1}(e_2)\}=\{-e_1,-e_2\},\;s(e_3)\in \{\pm e_3\}.$$

If $s^{-1}(e_1)=-e_2,s^{-1}(e_2)=-e_1$ and $s^{-1}(e_3)=e_3$, then $s^{-1}$ fixes $\frac{e_1-e_2+e_3}{2}$ and $\frac{\pm(e_1-e_2-e_3)}{2}$. Combining with \eqref{GSp6 1} and the fact that $\Theta^+$ contains the following 7 elements
$$\frac{e_1-e_2+e_3}{2}+ \frac{e_1'\pm e_2'}{2},\;\frac{e_1-e_2+e_3}{2}+ \frac{-e_1'+ e_2'}{2},\;\frac{\pm( e_1-e_2-e_3) }{2}+ \frac{e_1'\pm e_2'}{2},$$
we know that
$$s'^{-1}\{\frac{e_1'\pm e_2'}{2},\frac{-e_1'+e_2'}{2}\}= \{\frac{e_1'\pm e_2'}{2},\frac{-e_1'-e_2'}{2}\},$$
$$s'^{-1}(\frac{e_1'\pm e_2'}{2})\in \{\frac{\pm e_1'+ e_2'}{2},\frac{-e_1'-e_2'}{2}\}.$$ 
It is easy to see that such an $s'$ does not exist, so we get a contradiction. Similarly we can also get a contradiction when $s^{-1}(e_1)=-e_2,s^{-1}(e_2)=-e_1,s^{-1}(e_3)=-e_3$ or $s^{-1}(e_1)=-e_1,s^{-1}(e_2)=-e_2,s^{-1}(e_3)=e_3$.

Now the only case left is when $s^{-1}(e_1)=-e_1,s^{-1}(e_2)=-e_2$ and $s^{-1}(e_3)=-e_3$. In this case, we have $s^{-1}(\alpha^\vee)=-\alpha^{\vee}$ for all $\alpha^{\vee}\in \{\frac{\pm e_1\pm e_2\pm e_3}{2}\}$.
By the same argument as in the previous cases, we know that
$$s'^{-1}\{\frac{e_1'\pm e_2'}{2},\frac{-e_1'+e_2'}{2}\}= \{\frac{-e_1'\pm e_2'}{2},\frac{e_1'-e_2'}{2}\},$$
$$s'^{-1}(\frac{e_1'\pm e_2'}{2})\in \{\frac{\pm e_1'+ e_2'}{2},\frac{-e_1'-e_2'}{2}\}.$$ 
This implies that $s'^{-1}(e_1')=-e_1'$ and $s'^{-1}(e_2')=-e_2'$. In particular $w=(s,s')$ is the longest Weyl element. This proves the lemma.
\end{proof}

To sum up, we have proved that the local relative character is equal to 
$$\frac{\Delta_{G}(1)}{\Delta_{H/Z_{G,H}}(1)}\cdot \frac{L(1/2,\pi,\rho_X)}{L(1,\pi,\Ad)}=\zeta(1)^2\zeta(4)\zeta(6)\frac{L(1/2,\pi,\Spin_7\otimes \Spin_5)}{L(1,\pi,\Ad)}$$
where $\pi=\pi_1\otimes \pi_2$ is an unramified representation of $\GSp_6(F)\times \GSp_4(F)$.

\section{The model $(\GL_4\times \GL_2,\GL_2\times \GL_2)$}\label{sec:GL4}
In this section, we compute the local relative character for the model $(\GL_4\times \GL_2,\GL_2\times \GL_2)$. We again closely follow four steps in Section \ref{red summary}. 

Let $G=\GL_4\times \GL_2$ and $H=\GL_2\times \GL_2$ embed into $G$ via the map $(a,b)\mapsto (diag(a,b),b)$. Similarly, we can also define the quaternion version of this model. Let $D/F$ be a quaternion algebra, and let $G_D(F)=\GL_2(D)\times \GL_1(D)$ and $H_D(F)=\GL_1(D)\times \GL_1(D)$ embed into $G_D$ via the map $(a,b)\mapsto (diag(a,b),b)$.

We let $\Theta=\Theta_1\cup \Theta_2\cup \Theta_3$ with $\Theta_1$ being the weights of the representation $\wedge^2\otimes {\rm std}$ of $\GL_4(\BC)\times \GL_2(\BC)$, $\Theta_2$ being the weights of the standard representation of $\GL_4(\BC)$ and $\Theta_3$ being the weights of the dual of the standard representation of $\GL_4(\BC)$. We can write $\Theta_i$ as
$$\Theta_1=\{e_i+e_j+e_k'\mid 1\leq i<j\leq 4,1\leq k\leq 2\},$$
$$\Theta_2=\{e_i\mid 1\leq i\leq 4\},\;\Theta_3=\{-e_i\mid 1\leq i\leq 4\}.$$

Set $\eta=\begin{pmatrix}1&0&0&0\\0&1&0&0\\0&-1&1&0\\-1&1&-1&1\end{pmatrix}$ and $\eta^{-1}=\begin{pmatrix}1&0&0&0\\0&1&0&0\\0&1&1&0\\1&0&1&1\end{pmatrix}$. 
The proofs of the following two lemmas are similar to the $(\GSp_6\times \GSp_4,(\GSp_4\times \GSp_2)^0)$ case, and we will skip them here.

\begin{lem}
The double cosets $B(F)\back G(F)/H(F)$ contain a unique open orbit $B(F)(\eta, I_2) H(F)$. Here $B(F)=B_4(F)\times B_2(F)$ is the upper triangular Borel subgroup.
\end{lem}

\begin{lem}
For all $n\in \bar{N}_4(\varpi\CO_F)$ and $n'\in \bar{N}_{2}(\varpi\CO_F)$, we have
$$(n,n')(\eta,I_2)\in T(\CO_F)N(\varpi\CO_F)(\eta,I_2)H(\CO_F)$$
with $B=TN$.
\end{lem}

Then we compute the colors for this case. 
Let $\alpha_1=\varepsilon_1-\varepsilon_2,\alpha_2=\varepsilon_2-\varepsilon_3,\alpha_3=\varepsilon_3-\varepsilon_4$ be the simple roots of $\GL_4$ and $\alpha'=\varepsilon_1'-\varepsilon_2'$ be the simple root of $\GL_2$. Set
$$u_{-\alpha_1}(x)=\begin{pmatrix}1&0&0&0\\x&1&0&0\\0&0&1&0\\0&0&0&1\end{pmatrix},\;u_{-\alpha_2}(x)=\begin{pmatrix}1&0&0&0\\0&1&0&0\\0&x&1&0\\0&0&0&1\end{pmatrix},$$
$$u_{-\alpha_3}(x)=\begin{pmatrix}1&0&0&0\\0&1&0&0\\0&0&1&0\\0&0&x&1\end{pmatrix}.$$

We first study $\alpha_1$, we have
\begin{equation}\label{alpha 1} 
(u_{-\alpha_1}(x)\eta,I_2)=(b,h^{-1})\cdot (\eta,I_2)\cdot (g,h),\;(b,h^{-1})\in B(F), (g,h)\in H(F)
\end{equation}
where
$$g=\begin{pmatrix}1&0&0&0\\\frac{x}{x+1}&\frac{1}{x+1}&0&0\\0&0&\frac{1}{x+1}&\frac{x}{x+1}\\0&0&0&1\end{pmatrix},h=\begin{pmatrix}\frac{1}{x+1}&\frac{x}{x+1}\\0&1 \end{pmatrix},$$ $$b=\begin{pmatrix}1&0&0&0\\0&x+1&0&0\\0&0&1&-x\\0&0&0&x+1\end{pmatrix}.$$
This implies that $\beta_{\alpha_1}^{\vee}=e_1+e_3+e_2'$ and $\alpha_{1}^{\vee}-\beta_{\alpha_1}^{\vee}=e_1+e_4+e_1'$ (note that the representation has trivial central character).

For $\alpha_2$, we have
\begin{equation}\label{alpha 2} (u_{-\alpha_2}(x)\eta,I_2)=(b,h^{-1})\cdot (\eta,I_2)\cdot (g,h),\;(b,h^{-1})\in B(F), (g,h)\in H(F)
\end{equation}
where
$$g=\begin{pmatrix}\frac{1}{1-x}&\frac{-x}{1-x}&0&0\\0&1&0&0\\0&0&\frac{1}{1-x}&0\\0&0&0&\frac{1}{1-x}\end{pmatrix},h=\begin{pmatrix}\frac{1}{1-x}&0\\0&\frac{1}{1-x} \end{pmatrix},$$
$$ b=\begin{pmatrix}1-x&x&0&0\\0&1&0&0\\0&0&1-x&0\\0&0&0&1-x\end{pmatrix}.
$$
This implies that $\beta_{\alpha_2}^{\vee}=e_2$ and $\alpha_{2}^{\vee}-\beta_{\alpha_2}^{\vee}=-e_3$.

For $\alpha_3$, we have
\begin{equation}\label{alpha 3} (u_{-\alpha_3}(x)\eta,I_2)=(b,h^{-1})\cdot (\eta,I_2)\cdot (g,h),\;(b,h^{-1})\in B(F), (g,h)\in H(F)
\end{equation}
where
$$g=\begin{pmatrix}\frac{1}{1-x}&0&0&0\\0&1&0&0\\0&0&1&0\\0&0&0&\frac{1}{1-x}\end{pmatrix},h=\begin{pmatrix}1&0\\0&\frac{1}{1-x} \end{pmatrix}, b=\begin{pmatrix}1-x&0&0&0\\0&1&0&0\\0&0&1&0\\0&0&0&1-x\end{pmatrix}.
$$
This implies that $\beta_{\alpha_3}^{\vee}=e_2+e_3+e_1'$ and $\alpha_{3}^{\vee}-\beta_{\alpha_3}^{\vee}=e_1+e_3+e_2'$.

For the root $\alpha'$ on $\GL_2$, we can reduce to the root $\alpha_3$ on $\GL_4$ but we need to  change $u_{-\alpha_3}(x)\eta$ to $\eta u_{-\alpha_3}(-x)$. We have
\begin{equation}\label{beta} (\eta u_{-\alpha_3}(-x),I_2)=(b,h^{-1})\cdot (\eta,I_2)\cdot (g,h),\;(b,h^{-1})\in B(F), (g,h)\in H(F)
\end{equation}
where
$$g=\begin{pmatrix}\frac{1}{1+x}&\frac{x}{1+x}&0&0\\0&1&0&0\\0&0&1&0\\0&0&0&\frac{1}{1+x}\end{pmatrix},h=\begin{pmatrix}1&0\\0&\frac{1}{1+x} \end{pmatrix}, b=\begin{pmatrix}1+x&-x&0&0\\0&1&0&0\\0&0&1&0\\0&0&0&1+x\end{pmatrix}.
$$
This implies that $\beta_{\alpha'}^{\vee}=e_2+e_3+e_1'$ and $\alpha'^{\vee}-\beta_{\alpha'}^{\vee}=e_1+e_4+e_1'$.

\begin{prop}\label{theta for GL(4)xGL(2)}
$\Theta^+$ is consisting of the following 10 elements:
$$e_1+e_i+e_j',2\leq i\leq 3,1\leq j\leq 2;\;e_1+e_4+e_1',e_2+e_3+e_1';\;e_1,e_2,-e_3,-e_4.$$
\end{prop}

\begin{proof}
By the computation of colors above, we know that $\Theta^+$ is the smallest subset of $\Theta$ satisfying the following 5 conditions:
\begin{enumerate}
\item $e_1+e_3+e_2',e_1+e_4+e_1',e_2+e_3+e_1',e_2,-e_3\in \Theta^{+}$.
\item $\Theta^{+}-(\Theta^+\cap w_{\alpha_1}\Theta^+)=\{e_1+e_3+e_2',e_1+e_4+e_1'\}$.
\item $\Theta^{+}-(\Theta^+\cap w_{\alpha_2}\Theta^+)=\{e_2,-e_3\}$.
\item $\Theta^{+}-(\Theta^+\cap w_{\alpha_3}\Theta^+)=\{e_1+e_3+e_2',e_2+e_3+e_1'\}$.
\item $\Theta^{+}-(\Theta^+\cap w_{\alpha_1'}\Theta^+)=\{e_2+e_3+e_1',e_1+e_4+e_1'\}$.
\end{enumerate}
It is clear that the set in the statement satisfies these conditions. So we just need to show that the set is the unique subset of $\Theta$ satisfying these conditions. The argument is exactly the same as the case $(\GSp_6\times \GSp_4, (\GSp_4\times \GSp_2)^0)$ in Proposition \ref{theta proposition}. We will skip it here.
\end{proof}

It is clear that $\Theta^+$ satisfies \eqref{theta plus}. The last thing remains to prove Lemma \ref{lem constant reductive} for the current case. 

\begin{lem}
With the notation above, we have 
$$\sum_{w\in W}c_{WS}(w\theta)=\frac{1}{\Delta_{H/Z_{G,H}}(1)}=\frac{1}{\zeta(1)\zeta(2)^2}=(1-q^{-1})(1-q^{-2})^2.$$
\end{lem}

\begin{proof}
Since the summation is independent of $\theta$, we set $\theta=\delta_{B}^{1/2}$. The lemma follows from the following two claims:
\begin{enumerate}
\item $c_{WS}(w\theta)$ is zero unless $w$ is the longest Weyl element.
\item If $w$ is the longest Weyl element, we have $c_{WS}(w\theta)=(1-q^{-1})(1-q^{-2})^2$.
\end{enumerate}
The second claim is easy to prove so we will focus on the first one. For $w=(s,s')\in W=S_4\times S_2$, we know that $c_{WS}(w\theta)$ is nonzero if and only if
\begin{equation}\label{GL4 1}
1-q^{-1/2} \theta_{s(i)},\;1-q^{-1/2} \theta_{s(1)}\theta_{s(j)}\theta_{s'(k)}',
\end{equation}
$$1-q^{-1/2} \theta_{s(1)}\theta_{s(4)} \theta_{s'(1)}',\;1-q^{-1/2} \theta_{s(2)}\theta_{s(3)} \theta_{s'(1)}'$$
are nonzero for $1\leq i\leq 4,2\leq j\leq 3,1\leq k\leq 2$ where $\theta_1=q^{3/2},\theta_2=\theta_1'=q^{1/2},\theta_3=\theta_2'=q^{-1/2},\theta_4=q^{-3/2}$. 

Using the four terms $1-q^{-1/2} \theta_{s(i)}$ in \eqref{GL4 1}, we have $s(1),s(2)\in \{1,3,4\},\;s(3),s(4)\in \{1,2,4\}$. This implies that $\{s(1),s(2)\}$ is equal to $\{1,3\}$ or $\{3,4\}$. If it is equal to $\{1,3\}$, then $\theta_{s(1)}\theta_{s(2)}=q$. Hence $\theta_{s(1)}\theta_{s(2)}\theta_{s'(1)}'$ or $\theta_{s(1)}\theta_{s(2)}\theta_{s'(2)}'$ is equal to $q^{1/2}$. This is a contradiction. So we must have $\{s(1),s(2)\}=\{3,4\}$.

If $s(1)=3$, then $\theta_{s(1)}\theta_{s(3)}$ is equal to $1$ or $q$ (depends on whether $s(3)=2$ or $s(3)=1$). In both cases, we have $\theta_{s(1)}\theta_{s(3)} \theta_{s'(1)}'$ or $\theta_{s(1)}\theta_{s(3)} \theta_{s'(2)}'$ is equal to $q^{1/2}$. This is a contradiction. So we must have $s(1)=4$ and $s(2)=3$.

Now if $s(3)=1$, then $\theta_{s(1)}\theta_{s(3)}=1$, which implies that $\theta_{s(1)}\theta_{s(3)} \theta_{s'(1)}'$ or $\theta_{s(1)}\theta_{s(3)} \theta_{s'(2)}'$ is equal to $q^{1/2}$. This is a contradiction. So we must have $s(3)=2$ and $s(4)=1$.

Finally, using the fact that $1-q^{-1/2} \theta_{s(1)}\theta_{s(4)} \theta_{s'(1)}'\neq 0$ we know that $s'(1)=2$ and $s'(2)=1$. Hence $w$ is the longest Weyl element. This proves the lemma.
\end{proof}

To sum up, we have proved that the local relative character is equal to 
$$\zeta(1)\zeta(3)\zeta(4)\frac{L(1/2,\pi,\wedge^2\otimes {\rm std}_2)L(1/2,\pi_1,{\rm std}_4)L(1/2,\pi_1,{\rm std}_{4}^\vee)}{L(1,\pi,\Ad)}$$
where $\pi=\pi_1\otimes \pi_2$ is an unramified representation of $\GL_4(F)\times \GL_2(F)$.

\section{The model $(\GL_6,\GL_2\ltimes U)$}\label{sec:GL6}
In this section, we compute the local relative character for the model $(\GL_6,\GL_2\ltimes U)$. We closely follow the six steps in Section \ref{sec:6-steps}.

Let $G=\GL_6$, $H=H_0\ltimes U$ with 
$$H_0=\{diag(h,h,h) \mid h\in \GL_2\},$$
$$U=\{u(X,Y,X)=\begin{pmatrix} I_2&X&Z\\0&I_2&Y\\0&0&I_2 \end{pmatrix} \mid X,Y,Z\in Mat_{2\times 2}\}.$$
Let $P=LU$ be the parabolic subgroup of $G$ with $L=\{(h_1,h_2,h_3) \mid h_i\in \GL_2\}$. We define a generic character $\xi$ on $U(F)$ to be $\xi(u(X,Y,Z))=\psi(\lambda(u(X,Y,Z)))$ where $\lambda(u(X,Y,Z))=\tr(X)+\tr(Y)$. It is easy to see that $H_0$ is the stabilizer of this character and $(G,H)$ is the Whittaker induction of the trilinear $\GL_2$ model $(L,H_0,\xi)$. The model $(G,H,\xi)$ is the so called Ginzburg--Rallis model introduced by Ginzburg and Rallis in \cite{GR}.

We can also define the quaternion version of this model. Let $D/F$ be a quaternion algebra, and let $G_D(F)=\GL_3(D)$, $H_D=H_{0,D}\ltimes U_D$ with 
$$H_{0,D}(F)=\{diag(h,h,h)\mid h\in \GL_1(D)\},$$
$$U_D(F)=\begin{pmatrix} 1&X&Z\\0&1&Y\\0&0&1 \end{pmatrix}\mid X,Y,Z\in D\}.$$
Like the split case, we can define the character $\xi_D$ on $U_D(F)$ by replacing the trace map of $Mat_{2\times 2}$ by the trace map of $D$.

Let $w_0=\begin{pmatrix}0&0&I_2\\0&I_2&0\\I_2&0&0 \end{pmatrix}$ be the Weyl element that sends $U$ to its opposite. It is clear that the $w_0$-conjugation map stabilizes $L$ and fixes $H_0$. We define the map $a:\GL_1\rightarrow Z_L$ to be
$$a(t)=(tI_2,I_2,t^{-1}I_2).$$
This clearly satisfies \eqref{the map a}. For the open Borel orbit, let
$$\eta_0=diag(I_2,\begin{pmatrix}0&1\\1&0 \end{pmatrix},\begin{pmatrix}0&1\\1&0 \end{pmatrix}\begin{pmatrix}1&1\\0&1 \end{pmatrix})$$ 
be the representative of the open Borel orbit for the model $(L,H_0)$ as in Section \ref{section triple product}, and $\eta=\eta_0 w_0$. The relation \eqref{eta0 relation} has already been verified in Section \ref{section triple product}. This finishes the first three steps in Section \ref{sec:6-steps}.

Now we compute the set of colors and also the set $\Theta^+$. Let $\Theta$ be the weights of the exterior cube representation of $\GL_6(\BC)$. We can write it as 
$$\Theta=\{e_i+e_j+e_k \mid 1\leq i<j<k\leq 6\}.$$

Let $\alpha_i=\varepsilon_i-\varepsilon_{i+1}$ be the simple roots for $1\leq i\leq 5$. By the computation of the trilinear $\GL_2$-model in Section \ref{section triple product} and the discussion in Section \ref{sec non-red strategy 2} (in particular, Remark \ref{color in non-red}), we get the set of colors for this case:
$$\beta_{\alpha_1}^{\vee}=e_1+e_4+e_5,\;\alpha_{1}^{\vee}-\beta_{\alpha_1}^{\vee}=e_1+e_3+e_6,$$ $$\beta_{\alpha_3}^{\vee}=e_2+e_3+e_5,\;\alpha_{3}^{\vee}-\beta_{\alpha_3}^{\vee}=e_1+e_3+e_6,$$
$$\beta_{\alpha_5}^{\vee}=e_2+e_3+e_5,\;\alpha_{5}^{\vee}-\beta_{\alpha_3}^{\vee}=e_1+e_4+e_5.$$
Then we verify \eqref{nonred root} for $\alpha_2$ and $\alpha_4$.

For $\alpha_2$, let $u_{-\alpha_2}(a)=(x_{ij})_{1\leq i,j\leq 6}$ with $x_{ii}=1,\;x_{32}=a$ and $x_{ij}=0$ for all the other $(i,j)$. We have
$$u_{-\alpha_2}(a)\eta=\eta \begin{pmatrix}I_2&0&0\\0&I_2&X\\0&0&I_2\end{pmatrix},\;X=\begin{pmatrix}0&0\\0&a \end{pmatrix}.$$
This proves \eqref{nonred root} for $\alpha_2$.

For $\alpha_4$, let $u_{-\alpha_4}(a)=(x_{ij})_{1\leq i,j\leq 6}$ with $x_{ii}=1,\;x_{54}=a$ and $x_{ij}=0$ for all the other $(i,j)$. We have
$$u_{-\alpha_4}(a)\eta=\eta \begin{pmatrix}I_2&X&0\\0&I_2&0\\0&0&I_2\end{pmatrix},\;X=\begin{pmatrix}-a&0\\a&0 \end{pmatrix}.$$
This proves \eqref{nonred root} for $\alpha_4$. Next, we compute the set $\Theta^+$.

\begin{prop}\label{prop theta GL(6)}
$\Theta^+$ is consisting of the following 10 elements:
$$e_1+e_2+e_i,e_1+e_3+e_j,e_1+e_4+e_5,e_2+e_3+e_4,e_2+e_3+e_5$$
where $3\leq i\leq 6$ and $4\leq j\leq 6.$
\end{prop}

\begin{proof}
By the computations above, we know that $\Theta^+$ is the smallest subset of $\Theta$ satisfying the following 5 conditions:
\begin{enumerate}
\item $e_1+e_4+e_5,e_1+e_3+e_6,e_2+e_3+e_5\in \Theta^{+}$.
\item $\Theta^{+}-(\Theta^+\cap w_{\alpha_1}\Theta^+)=\{e_1+e_4+e_5,e_1+e_3+e_6\}$.
\item $\Theta^{+}-(\Theta^+\cap w_{\alpha_3}\Theta^+)=\{e_2+e_3+e_5,e_1+e_3+e_6\}$.
\item $\Theta^{+}-(\Theta^+\cap w_{\alpha_5}\Theta^+)=\{e_2+e_3+e_5,e_1+e_4+e_5\}$.
\item $\Theta^{+}$ is stable under $w_{\alpha_2}$ and $w_{\alpha_4}$.
\end{enumerate}
It is clear that the set in the statement satisfies these conditions. So we just need to show that the set is the unique subset of $\Theta$ satisfying these conditions. The argument is exactly the same as the case $(\GSp_6\times \GSp_4, (\GSp_4\times \GSp_2)^0)$ in Proposition \ref{theta proposition}. We will skip it here.
\end{proof}

It is clear that $\Theta^+$ satisfies \eqref{theta plus}. The last thing remains  to prove Lemma \ref{lem constant nonreductive} for the current case. 

\begin{lem}\label{GR case constant lem}
With the notation above, we have 
$$\sum_{w\in W}c_{WS}(w\theta)=\frac{1}{\Delta_{H_0/Z_{G,H}}(1)}=\frac{1}{\zeta(2)}=(1-q^{-2}).$$
\end{lem}

\begin{proof}
Recall that $W=S_6$ is the permutation group of 6 variables. The goal is to show that
$$\sum_{s\in S_6} \frac{\Pi_{e_i+e_j+e_k\in \Theta^+} (1-q^{-1/2}\theta_{s(i)}\theta_{s(j)}\theta_{s(k)})}{\Pi_{1\leq i<j\leq 6} (1-\theta_{s(i)}/\theta_{s(j)})}=1-q^{-2}.$$
Here $\theta_i$ are arbitrary variables satisfying the equation $\Pi_{i=1}^{6}\theta_i=1$. We define the subset $\Theta_{0}^+$ of $\Theta^+$ to be
$$\Theta_{0}^+=\{e_1+e_i+e_j,e_1+e_4+e_5,e_2+e_3+e_5 \mid 2\leq i\leq 3,5\leq j\leq 6\}.$$
It contains those weights in the $\wedge^2\otimes {\rm std}_2$ representation of $\GL_4\times \GL_2$. We need a lemma.

\begin{lem}\label{constant for GL(4)xGL(2) non-red}
We embed $S_4\times S_2$ into $S_6$ by letting $S_4$ act  on the first four elements and $S_2$ acts on the last two elements. Then
$$\sum_{s\in S_4\times S_2} \frac{\Pi_{e_i+e_j+e_k\in \Theta_{0}^+} (1-q^{-1/2}\theta_{s(i)}\theta_{s(j)}\theta_{s(k)})}{\Pi_{\{(i,j) \mid 1\leq i<j\leq 4\;\text{or}\;5\leq i<j\leq 6\}} (1-\theta_{s(i)}/\theta_{s(j)})}=1-q^{-2}.$$
\end{lem}

\begin{proof}
By the identity for the triple product in Section \ref{section triple product}, we have (here we embed $S_2\times S_2$ into $S_4$ by letting the first $S_2$-copy act  on the first two elements and the second $S_2$-copy act on the last two elements)
$$\sum_{s\in S_2\times S_2\times S_2} \frac{\Pi_{e_i+e_j+e_k\in \Theta_{1}^+} (1-q^{-1/2}\theta_{s(i)}\theta_{s(j)}\theta_{s(k)})}{ (1-\theta_{s(1)}/\theta_{s(2)})(1-\theta_{s(3)}/\theta_{s(4)})(1-\theta_{s(5)}/\theta_{s(6)})}=1-q^{-2},$$
where $\Theta_{1}^+=\{e_1+e_3+e_5,e_1+e_3+e_6,e_1+e_4+e_5,e_2+e_3+e_5\}.$ Hence in order to prove the lemma, it is enough to show that 
$$\sum_{s\in S_4/S_2\times S_2} \frac{(1-q^{-1/2}\theta_{s(1)}\theta_{s(2)}\theta_{5})(1-q^{-1/2}\theta_{s(1)}\theta_{s(2)}\theta_{6})}{\Pi_{1\leq i\leq 2, 3\leq j\leq 4} (1-\theta_{s(i)}/\theta_{s(j)})}=1.$$
This follows from an easy computation.
\end{proof}

By the lemma above, we have
$\sum_{s\in S_6} \frac{\Pi_{e_i+e_j+e_k\in \Theta^+} (1-q^{-1/2}\theta_{s(i)}\theta_{s(j)}\theta_{s(k)})}{\Pi_{1\leq i<j\leq 6} (1-\theta_{s(i)}/\theta_{s(j)})}$
is equal to
$$(1-q^{-2})\cdot \sum_{s\in S_6/S_4\times S_2} \frac{\Pi_{e_i+e_j+e_k\in \Theta^+-\Theta_{0}^{+}} (1-q^{-1/2}\theta_{s(i)}\theta_{s(j)}\theta_{s(k)})}{\Pi_{1\leq i\leq 4, 5\leq j\leq 6} (1-\theta_{s(i)}/\theta_{s(j)})}.$$
So it is enough to show that 
\begin{equation}\label{GL(6) constant 1}
\sum_{s\in S_6/S_4\times S_2} \frac{\Pi_{e_i+e_j+e_k\in \Theta^+-\Theta_{0}^{+}} (1-q^{-1/2}\theta_{s(i)}\theta_{s(j)}\theta_{s(k)})}{\Pi_{1\leq i\leq 4, 5\leq j\leq 6} (1-\theta_{s(i)}/\theta_{s(j)})}=1.
\end{equation}
The set $\Theta^+-\Theta_{0}^{+}$ is equal to $\{e_i+e_j+e_k \mid 1\leq i<j<k\leq 4\}$. It is easy to see that the constant coefficient of the left hand side of \eqref{GL(6) constant 1} is equal to 1. So we just need to show that the $q^{-1/2},q^{-1},q^{-3/2},q^{-2}$-coefficients are equal to 0. For this, we can replace the summation over $S_6/S_4\times S_2$ by the summation over $S_6$, and we need to show that the $q^{-1/2},q^{-1},q^{-3/2},q^{-2}$-coefficients of 
$$\sum_{s\in S_6} \frac{\Pi_{1\leq i<j<k\leq 4} (1-q^{-1/2}\theta_{s(i)}\theta_{s(j)}\theta_{s(k)})}{\Pi_{1\leq i\leq 4, 5\leq j\leq 6} (1-\theta_{s(i)}/\theta_{s(j)})}$$
are equal to 0. We can rewrite the function in the summation as
\begin{eqnarray*}
 &&\frac{\Pi_{1\leq i<j<k\leq 4}  (1-q^{-1/2}\theta_{s(i)}\theta_{s(j)}\theta_{s(k)})}{\Pi_{1\leq i\leq 4, 5\leq j\leq 6} (1-\theta_{s(i)}/\theta_{s(j)})}\\
 &=&\frac{\theta_{s(5)}^4\theta_{s(6)}^4\cdot \Pi_{1\leq i<j<k\leq 4} (1-q^{-1/2}\theta_{s(i)}\theta_{s(j)}\theta_{s(k)})}{\Pi_{1\leq i\leq 4, 5\leq j\leq 6}  (\theta_{s(j)}-\theta_{s(i)})}\\
&=&\frac{\theta_{s(5)}^4\theta_{s(6)}^4\cdot \Pi_{\{(i,j)\mid 1\leq i<j\leq 4\;\text{or}\;5\leq i<j\leq 6\}} (\theta_{s(j)}-\theta_{s(i)}) }{\Pi_{1\leq i< j\leq 6}  (\theta_{s(j)}-\theta_{s(i)})}\\
&&\cdot \Pi_{1\leq i<j<k\leq 4} (1-q^{-1/2}\theta_{s(i)}\theta_{s(j)}\theta_{s(k)}).
\end{eqnarray*}
Since the denominator is $(S_6,\sgn)$-invariant ($\sgn$ is the sign character of $S_6$), we just need to show that the $(S_6,\sgn)$-summation of the $q^{-1/2},q^{-1},q^{-3/2},q^{-2}$-coefficients of
\begin{equation}\label{GL(6) constant 2}
\theta_{5}^4\theta_{6}^4\cdot \Pi_{\{(i,j)|\;1\leq i<j\leq 4\;\text{or}\;5\leq i<j\leq 6\}} (\theta_{j}-\theta_{i}) \cdot \Pi_{1\leq i<j<k\leq 4} (1-q^{-1/2}\theta_{i}\theta_{j}\theta_{k})
\end{equation}
are equal to 0. A direct computation shows that $\theta_{5}^4\theta_{6}^4\cdot \Pi_{\{(i,j)|\;1\leq i<j\leq 4\;\text{or}\;5\leq i<j\leq 6\}} (\theta_{j}-\theta_{i})$ is consisting of elements of the form
$$\Pi_i \theta_{i}^{a_i},\;\{a_1,a_2,a_3,a_4\}=\{0,1,2,3\},\;\{a_5,a_6\}=\{4,5\}.$$
Then any term $\Pi_i \theta_{i}^{b_i}$ appeared in the $q^{-1/2}$-coefficient of \eqref{GL(6) constant 2} must satisfy the condition $\{b_5,b_6\}=\{4,5\}$, and also satisfies at least one of the following two conditions
\begin{itemize}
\item $b_i=4$ for some $1\leq i\leq 4$;
\item $b_i=b_j$ for some $1\leq i<j\leq 4$.
\end{itemize}
In both cases, we have $b_i=b_j$ for some $i\neq j$. This implies that the $(S_6,\sgn)$-summation of the $q^{-1/2}$-coefficient is equal to 0.

Moreover, any term $\Pi_i \theta_{i}^{b_i}$ appeared in the $q^{-1},q^{-3/2},q^{-2}$-coefficients of \eqref{GL(6) constant 2} must satisfy the following two conditions 
\begin{itemize}
\item $b_i\in \{4,5\}$ for some $1\leq i\leq 4$;
\item $\{b_5,b_6\}=\{4,5\}$.
\end{itemize}
This implies that $b_i=b_j$ for some $i\neq j$. This implies that the $(S_6,\sgn)$-summation of the $q^{-1},q^{-3/2},q^{-2}$-coefficients are equal to 0. This finishes the proof of the lemma.
\end{proof}

To sum up, we have proved that the local relative character is equal to 
$$\zeta(1)\zeta(3)\zeta(4)\zeta(5)\zeta(6)\frac{L(1/2,\pi,\wedge^3)}{L(1,\pi,\Ad)}$$
where $\pi$ is an unramified representation of $\GL_6(F)$.

\section{The models $(\GU_6,\GU_2\ltimes U)$ and $(\GU_4\times \GU_2,(\GU_2\times \GU_2)^0)$}\label{sec GU}
\subsection{The models}
Let $E=F(\sqrt{\eps})$ be a quadratic extension of $F$, $\eta_{E/F}$ be the quadratic character associated to $E$,  $N_{E/F}$ (resp. $\tr_{E/F}$) be the norm map (resp. trace map), and $x\rightarrow \bar{x}$ be the Galois action on $E$. Denote  $w_{n}$ to  be the longest Weyl element of $\GL_n$. Define the quasi-split even unitary similitude group $\GU_{n,n}(F)$ to be
\begin{equation}\label{eq:GU-n}
\GU_{n,n}(F)=\{ g\in \GL_{2n}(E)\colon  {}^{t}\bar{g}w_{2n}g=l(g) w_{2n}\}	
\end{equation}
where $l(g)\in F^{\times}$ is the similitude factor of $g$. 

We first define the model $(\GU_6,\GU_2\ltimes U)$. Let $G=\GU_{3,3}$, and $P=LU$ be the standard parabolic subgroup of $G$ with ($g^*=w_2{}^t\bar{g}^{-1}w_2$)
\begin{align*}
L(F)=&\{m(g,h)=\begin{pmatrix}
g&&\\&h&\\&&l(h)g^*	\end{pmatrix} \mid g\in \GL_2(E),\; ~h\in\GU_{1,1}(F)\},
\end{align*}
$$U(F)=\{u(X,Y)=\begin{pmatrix}
I_2&X&Y\\&I_{2}&X'\\&&I_2	\end{pmatrix} \mid  X,Y\in Mat_{2\times 2}(E),$$
$$X'=-w_2{}^t\!Xw_2,w_2Y + {}^t\!Yw_2+{}^t\!X'w_2X'=0\}.$$
Let $\xi$ be a generic character of $U(F)$ given by
$$
\xi(u(X,Y))=\psi(\lambda(u(X,Y))),\;\lambda(u(X,Y))=\tr_{E/F}(\tr(X)).
$$
Then the stabilizer of $\xi$ under the adjoint action of $L(F)$ is
$$
H_{0}(F):=\{m(h,h)\mid h\in \GU_{1,1}(F)\}=\{diag(h,h,h) \mid h\in \GU_{1,1}(F)\}.
$$
Let $H=H_0\ltimes U$ and we extend the character $\xi$ to $H$ by making it trivial on $H_0$. The model $(G,H,\xi)$ is the analogue of the Ginzburg--Rallis model in the previous section for unitary similitude group. We can also define the quaternion (non quasi-split) version of this model by letting $G_D$ be the non quasi-split unitary similitude group (in the archimedean case $G_D=\GU_{4,2}$).

Now we define the model $(\GU_4\times \GU_2,(\GU_2\times \GU_2)^0)$.  
Let $G=\GU_{2,2}\times \GU_{1,1}$ and $H=(\GU_{1,1}\times \GU_{1,1})^0=\{(h_1,h_2)\in \GU_{1,1}\times \GU_{1,1}\mid l(h_1)=l(h_2)\}$. We can embed  H into $G$ via the map
$$(h_1,h_2)\in H\mapsto \begin{pmatrix}a&0&b\\0&h_1&0\\c&0&d \end{pmatrix} ,h_1)\in G,\;h_2=\begin{pmatrix}a&b\\c&d\end{pmatrix}.$$
For the pure inner forms of this model, we use $\GU_{2,0}=\GU_{0,2}$ to denote the non quasi-split unitary similitude group of rank 2, and we use $\GU_{3,1}$ to denote the non quasi-split unitary similitude group of rank 4 and split rank 1 (we use these notation in order to be compatible with the standard notation in the archimedean case). In the $p$-adic case, the pure inner forms are $(\GU_{2,2}\times \GU_{2,0},(\GU_{2,0}\times \GU_{0,2})^0),\;(\GU_{3,1}\times \GU_{1,1},(\GU_{1,1}\times \GU_{2,0})^0),\;(\GU_{3,1}\times \GU_{2,0},(\GU_{2,0}\times \GU_{1,1})^0)$. In the archimedean case, there is an extra compact pure inner form $(\GU_{4,0}\times \GU_{2,0},(\GU_{2,0}\times \GU_{2,0})^0)$.

The goal of this section is to compute the local relative character $I(\phi_\theta)$ for these two models. As we mentioned in Section \ref{sec:strategy}, the difference between these models and all the other models is that since $G$ is not split, the root space maybe two-dimensional. In the next subsection, we will prove two identities that will be used in our computation. Then  we will compute the relative character in the last two subsections. From now on, we assume that $E/F$ is unramified and $\eps\in \CO_{F}^{\times}$.

\subsection{Two identities}\label{sec:two identities}

\begin{lem}\label{rank one quadratic}
Let $\eta$ (resp. $\sigma$) be a unitary unramified character of $E^\times$ (resp. $F^\times$). We have
\begin{equation}\label{GU 1}
1+q^2\int_{\CO_F^2}\sig(x^2-\eps y^2-x)\eta(x+y\sqrt{\eps})\ud x\ud y=\frac{q^{2}(1-q^{-1})(1-q^{-4}\sig^2\eta(\varpi))}{(1-q^{-1}\sig(\varpi))(1-q^{-2}\sig\eta(\varpi))}.
\end{equation}
\end{lem}
This integral is an analogy of \eqref{rank one integral}. We need this identity when the root space is two dimensional.
To compute it, we need the following lemma.

\begin{lem}
The equation $x^2-\eps y^2-x=0$ has $q$ nonzero solutions in $\BF_q\times \BF_q$.
\end{lem}

\begin{proof}
The equation is equivalent to $(2x-1)^2-\eps(2y)^2=1$. So the number of solutions is equal to $|U_1(\BF_q)|=q+1$. In particular, there are $q$ nonzero solutions.
\end{proof}

Now we prove Lemma \ref{rank one quadratic}. Set $X=\CO_F\times \CO_{F}^{\times}\cup \CO_{F}^{\times}\times \CO_F$. Then for $k\geq 0$, we have
$$\varpi^k X=\{(x,y)\in \CO_{F}^{2}|\; max\{|x|,|y|\}=q^{-k}\}.$$
This implies that $\CO_F\times \CO_F$ is a disjoint union of $\varpi^k X$ for $k\geq 0$. Also for $(x,y)\in \varpi^k X$, we have
$$x+y\sqrt{\eps}\in \varpi^k\CO_{E}^{\times}.$$
As a result, the left hand side of \eqref{GU 1} is equal to
$$1+\sum_{k\geq 0}q^2\int_{\varpi^k X}\sig(x^2-\eps y^2-x)\eta(x+y\sqrt{\eps})\ud x\ud y $$
$$=1+\sum_{k\geq 0}q^{2-2k}\int_{X}\sig(\varpi^{2k}x^2-\varpi^{2k}\eps y^2-\varpi^{k} x)\eta^k(\varpi) \ud x\ud y$$
$$=1+\sum_{k\geq 0}q^{2-2k}\eta^k\sig^k(\varpi)\int_{X}\sig(\varpi^{k}x^2-\varpi^{k}\eps y^2- x) \ud x\ud y.$$

Now we study the integral $\int_{X}\sig(\varpi^{k}x^2-\varpi^{k}\eps y^2- x) \ud x\ud y$. When $k>0$, by Theorem 10.2.1 in \cite{I00}, the integral is equal to
$$q^{-2}(((q^2-1)-(q-1))+(q-1)\frac{(1-q^{-1})\sigma(\varpi)}{1-q^{-1}\sigma(\varpi)})$$
$$=q^{-2}((q^2-q)+(q-1)\frac{(1-q^{-1})\sigma(\varpi)}{1-q^{-1}\sigma(\varpi)}).$$
When $k=0$, by Theorem 10.2.1 in \cite{I00} and the lemma above, the integral is equal to
$$q^{-2}(((q^2-1)-(q))+q\frac{(1-q^{-1})\sigma(\varpi)}{1-q^{-1}\sigma(\varpi)})$$
$$=q^{-2}((q^2-q)+(q-1)\frac{(1-q^{-1})\sigma(\varpi)}{1-q^{-1}\sigma(\varpi)})-q^{-2}+q^{-2}\frac{(1-q^{-1})\sigma(\varpi)}{1-q^{-1}\sigma(\varpi)}.$$
This implies that the left hand side of \eqref{GU 1} is equal to
\begin{align*}
 &\frac{(1-q^{-1})\sigma(\varpi)}{1-q^{-1}\sigma(\varpi)}+\sum_{k\geq 0}\eta^k\sig^k(\varpi)q^{-2k}\cdot ((q^2-q)+(q-1)\frac{(1-q^{-1})\sigma(\varpi)}{1-q^{-1}\sigma(\varpi)})\\
=&\frac{(1-q^{-1})\sigma(\varpi)}{1-q^{-1}\sigma(\varpi)}+\frac{1}{1-\eta\sig(\varpi) q^{-2}}((q^2-q)+(q-1)\frac{(1-q^{-1})\sigma(\varpi)}{1-q^{-1}\sigma(\varpi)})\\
=&\frac{q^{2}(1-q^{-1})(1-q^{-4}\sig^2\eta(\varpi))}{(1-q^{-1}\sig(\varpi))(1-q^{-2}\sig\eta(\varpi))}.
\end{align*} 
This proves Lemma \ref{rank one quadratic}. We also need the following identity.

\begin{lem}\label{rank one quadratic 2}
We have (recall that $\varphi=\varphi_0=1_{\CO_F}-\frac{1}{q-1}\cdot 1_{\varpi^{-1}\CO_{F}^{\times}}$)
\[
1+q^2\int_{\CO_{F}^{2}} \eta(x+\sqrt{\varepsilon}y)\cdot |x^2-\varepsilon y^2|^{-1}\cdot \varphi(\frac{2x}{x^2-\varepsilon y^2})\ud x\ud y=q^2(1-q^{-2}\eta(\varpi)).
\]
\end{lem}

\begin{proof}
A direct computation shows that 
$$\int_{X} \eta(x+\sqrt{\varepsilon}y)\cdot |x^2-\varepsilon y^2|^{-1}\cdot \varphi(\frac{2x}{x^2-\varepsilon y^2})\ud x\ud y=\frac{q^2-1}{q^2},$$
$$\int_{\varpi X} \eta(x+\sqrt{\varepsilon}y)\cdot |x^2-\varepsilon y^2|^{-1}\cdot \varphi(\frac{2x}{x^2-\varepsilon y^2})\ud x\ud y=-\frac{\eta(\varpi)}{q^2},$$
$$\int_{\varpi^k X} \eta(x+\sqrt{\varepsilon}y)\cdot |x^2-\varepsilon y^2|^{-1}\cdot \varphi(\frac{2x}{x^2-\varepsilon y^2})\ud x\ud y=0,\;k\geq 2.$$
This proves the lemma.
\end{proof}

\subsection{The computation for $(\GU_6,\GU_2\ltimes U)$} 
In this subsection, we compute the local relative character for the model $(\GU_6,\GU_2\ltimes U)$. First, all the arguments in Section \ref{sec non-red strategy 1} still work for the current case, the only exception is that the equation \eqref{decomposition of measure} will become
\begin{equation}
\int_{G(F)}\Phi(g)\ud g=\frac{\Delta_{G}(1)}{\Delta_{H_0/Z_{G,H}}(1)}\zeta_E(1)^{-3}\zeta(1)^{-1}
\end{equation}
$$\cdot \int_{H(F)/Z_{G,H}(F)}\int_{B(F)}\Phi(b\eta h)\ud b\ud h.$$
This is because in the split case, $|T(\BF_q)|=(q-1)^{\dim(T)}$; for our current model, $|T(\BF_q)|=(q^2-1)^3(q-1)$. This implies that 
$$I(\phi_\theta)=\frac{\Delta_{H_0/Z_{G,H}}(1)}{\Delta_{G}(1)}\zeta_E(1)^{3}\zeta(1)\cdot \int_{K}^{\ast} \CY_{\theta^{-1},\xi}(k)\ud k\cdot \int_{K}^{\ast} \CY_{\theta,\xi^{-1}}(k)\ud k.$$

Now we compute the integral $\int_{K}^{\ast} \CY_{\theta,\xi}(k)\ud k$. 
Let $\alpha_1=\varepsilon_1-\varepsilon_2,\;\alpha_2=\varepsilon_2-\varepsilon_3$ and $\alpha_3=2\varepsilon_3$ be the simple roots of $G(F)$. Note that the root spaces of $\alpha_1$ and $\alpha_2$ are two dimensional and the root space of $\alpha_3$ is one dimensional. Let $w_0=\begin{pmatrix}0&0&I_2\\0&I_2&0\\I_2&0&0 \end{pmatrix}$ be the Weyl element that sends $U$ to its opposite. It is clear that the $w_0$-conjugation map stabilizes $L$ and fixes $H_0$. We define the map $a:\GL_1\rightarrow Z_L$ to be
$a(t)=(tI_2,I_2,t^{-1}I_2).$
This clearly satisfies \eqref{the map a}.

For the open Borel orbit, let
$$\eta_0=diag(\begin{pmatrix}0&1\\1&0 \end{pmatrix}\begin{pmatrix}1&1\\0&1 \end{pmatrix},I_2,\begin{pmatrix}1&0\\-1&1 \end{pmatrix}\begin{pmatrix}0&1\\1&0 \end{pmatrix})$$ 
be the representative of the open Borel orbit for the model $(L,H_0)$, and $\eta=\eta_0 w_0$. 
The relation \eqref{eta0 relation} can be easily verified as in the trilinear $\GL_2$-model case in Section \ref{section triple product}.

Now we compute the colors. Let $\Theta$ be the weights of the exterior cube representation of $\hat{G}(\BC)$. We can write it as 
$$\Theta=\{e_i,-e_i,\frac{\pm e_1\pm e_2\pm e_3}{2} \mid 1\leq i\leq 3\}.$$
The weight spaces of $e_i,-e_i$ are two dimensional and the weight spaces of $\frac{\pm e_1\pm e_2\pm e_3}{2}$ are one dimensional.
More precisely, the exterior representation of the $L$-group ${}^L\GU_6$ of $\GU_6$ is explicated in Section 3.1 \cite{Z}. 
More details on  the exterior cubic $L$-function of $\GU_{6}$ are also given there. 

For $\alpha_1$, as in the split case, we let
$I_{\alpha_1}(\theta)=vol(\CI)^{-1}\int_{G(F)} \CY_{\theta,\xi}^{0}(x\eta )(\Phi_1(x)+\Phi_{w_{\alpha_1}}(x)) \ud x.$
Since the root space is two dimensional, the same argument in the split case implies that
\begin{eqnarray*}
I_{\alpha_1}(\theta)&=&1+q^2\int_{\CO_{F}^{2}} (\theta^{-1}\delta^{1/2})(e^{\alpha_{1}^{\vee}}((x+y\sqrt{\varepsilon})^{-1})) \\
&&\CY_{\theta,\xi}^{0}(u_{-\alpha_1}((x+y\sqrt{\varepsilon})^{-1})\eta) \ud x\ud y.
\end{eqnarray*}
Here $u_{-\alpha_1}(a)=diag(\begin{pmatrix}1&0\\a&1 \end{pmatrix},I_2,\begin{pmatrix}1&0\\-\bar{a}&1 \end{pmatrix})$. Meanwhile, a direct computation shows that $u_{-\alpha_1}(x+y\sqrt{\eps})\eta$ is equal to
$$diag(\begin{pmatrix}\frac{1}{x+1}&0\\0&1\end{pmatrix},\begin{pmatrix}1&\frac{-y\sqrt{\eps}}{x+1}\\0&\frac{1}{x+1}\end{pmatrix},\begin{pmatrix}\frac{1}{x+1}&0\\0&1\end{pmatrix})\cdot \eta $$
$$\cdot diag(\begin{pmatrix}1&y\sqrt{\eps}\\0&x+1\end{pmatrix},\begin{pmatrix}1&y\sqrt{\eps}\\0&x+1\end{pmatrix},\begin{pmatrix}1&y\sqrt{\eps}\\0&x+1\end{pmatrix}).$$
Since $\frac{1}{x+y\sqrt{\eps}}=\frac{x}{x^2-y^2\eps}-\frac{y\sqrt{\eps}}{x^2-y^2\eps}$ and $1+\frac{x}{x^2-y^2\eps}=\frac{x+x^2-y^2\eps}{x^2-y^2\eps}$, we have
$$I_{\alpha_1}(\theta)=1+q^2\int_{\CO_{F}^2} \sigma(x^2-\varepsilon y^2-x)\eta(x+y\sqrt{\varepsilon}) \ud x\ud y$$
where $\eta=\theta(e^{\alpha_{1}^{\vee}})\cdot (|\;|^{-1}\cdot \sigma^{-1})\circ N_{E/F}$ and $\sigma=\theta(e^{\beta_{\alpha^{\vee}}})|_{F^{\times}}\cdot |\;|^{-1/2}$ with $\beta_{\alpha_1}^{\vee}=\frac{e_1-e_2-e_3}{2}$. Combing with Lemma \ref{rank one quadratic}, we have
$$I_{\alpha_1}(\theta)=q^2(1-q^{-1})\cdot \frac{1-q^{-2}e^{\alpha_{1}^{\vee}}(\theta)}{(1-q^{-1/2}e^{\beta_{\alpha_1}^{\vee}}(\theta))(1-q^{-1/2}e^{\alpha_{1}^{\vee}-\beta_{\alpha_1}^{\vee}}(\theta))}$$
with
$\beta_{\alpha_1}^{\vee}=\frac{e_1-e_2-e_3}{2},\;\alpha_{1}^{\vee}-\beta_{\alpha_1}^{\vee}=\frac{e_1-e_2+e_3}{2}.$

For $\alpha_2$, as in the Ginzburg--Rallis model case, it is easy to see that 
$$\CY_{\theta,\xi}^{0}(u_{-\alpha_2}(a)\eta)=\varphi(a+\bar{a}), u_{-\alpha_2}(a)=diag(1, \begin{pmatrix}1&0\\a&1 \end{pmatrix},\begin{pmatrix}1&0\\-\bar{a}&1 \end{pmatrix},1).$$
Then we have (note that the root space in this case is also 2-dimensional)
$$I_{\alpha_2}(\theta)=1+q^2\int_{\CO_{F}^{2}}\theta(e^{\alpha_{2}^{\vee}}(x+\sqrt{\varepsilon} y))\cdot |x^2-\varepsilon y^2|^{-1} \cdot \varphi(\frac{2x}{x^2-\varepsilon y^2})\ud x\ud y.$$
By Lemma \ref{rank one quadratic 2}, we know that
$$I_{\alpha_2}(\theta)=q^2\cdot (1-q^{-2}e^{\alpha_{2}^{\vee}}(\theta)).$$

For $\alpha_3$, the root space is one dimensional, so we have the identity
$$I_{\alpha_3}(\theta)=1+q\int_{\CO_F} (\theta^{-1}\delta^{1/2})(e^{\alpha_{3}^{\vee}}(a^{-1})) \CY_{\theta,\xi}^{0}(u_{-\alpha_3}(a^{-1})\eta)\ud a$$
where $u_{-\alpha_3}(x)=diag(I_2,\begin{pmatrix}1&0\\x\sqrt{\eps}&1 \end{pmatrix},I_2)$. On the other hand, $u_{-\alpha_3}(x)\eta$ is equal to
\begin{align*}
&diag(\begin{pmatrix}1+x\sqrt{\eps}&-x\sqrt{\eps}\\0&1-x\sqrt{\eps} \end{pmatrix},\begin{pmatrix}1& -x\sqrt{\eps}\\ 0&1-x^2\eps \end{pmatrix},\begin{pmatrix}1-x\sqrt{\eps}&-x\sqrt{\eps}\\0&1+x\sqrt{\eps} \end{pmatrix})\cdot \eta\\ 
&\times diag(\begin{pmatrix}\frac{1}{1-x^2\eps}&\frac{x\sqrt{\eps}}{1-x^2\eps}\\ \frac{x\sqrt{\eps}}{1-x^2\eps}&\frac{1}{1-x^2\eps} \end{pmatrix},\begin{pmatrix}\frac{1}{1-x^2\eps}&\frac{x\sqrt{\eps}}{1-x^2\eps}\\ \frac{x\sqrt{\eps}}{1-x^2\eps}&\frac{1}{1-x^2\eps} \end{pmatrix},\begin{pmatrix}\frac{1}{1-x^2\eps}&\frac{x\sqrt{\eps}}{1-x^2\eps}\\ \frac{x\sqrt{\eps}}{1-x^2\eps}&\frac{1}{1-x^2\eps} \end{pmatrix}).
\end{align*}
This implies that (note that all the characters are unramified and hence their values at $a+\sqrt{\varepsilon},\;a-\sqrt{\varepsilon}$ are equal to 1 for all $a\in \CO_F$)
$$I_{\alpha_3}(\theta)=q+1=(q+1)\cdot \frac{1-q^{-1}e^{\alpha_{3}^{\vee}}(\theta)}{1-q^{-1}e^{\beta_{\alpha_3}^{\vee}}(\theta)},$$
with $\beta_{\alpha_3}^{\vee}=\alpha_{3}^{\vee}-\beta_{\alpha_3}^{\vee}=e_3.$ Then we compute the set $\Theta^+$.

\begin{lem}
Let $W=S_3\ltimes (\BZ/2\BZ)^3$ be the Weyl group of $G$ and let $\Theta^+$ be the smallest subset of $\Theta$ satisfying the following two conditions:
\begin{enumerate}
\item $\frac{e_1-e_2\pm e_3}{2},e_3\in \Theta^+$.
\item $\Theta^+-(\Theta^+\cap w_{\alpha_1}\Theta^+)=\{\frac{e_1-e_2\pm e_3}{2}\}$, $\Theta^+= w_{\alpha_2}\Theta^+,$ $\Theta^+-(\Theta^+\cap w_{\alpha_3}\Theta^+)=\{e_3\}$.
\end{enumerate}
Then we have $\Theta^+=\{e_1,e_2,e_3,\frac{e_1\pm e_2\pm e_3}{2}\}$.
\end{lem}

\begin{proof}
It is clear that the set $\{e_1,e_2,e_3,\frac{e_1\pm e_2\pm e_3}{2}\}$ satisfies both conditions. So we just need to show that the set is the unique subset of $\Theta$ satisfying these conditions. The argument is exactly the same as the case $(\GSp_6\times \GSp_4, (\GSp_4\times \GSp_2)^0)$ in Proposition \ref{theta proposition}. We will skip it here.
\end{proof}

Now we decompose $\Theta$ as $\Theta_1\cup \Theta_2$ and $\Phi=\Phi_1\cup \Phi_2$ where $\Theta_i,\Phi_i$ contain the weights/roots whose weight spaces/root spaces are $i$ dimensional. More specifically, 
$$\Phi_1=\{\pm 2e_i\},\;\Phi_2=\{\pm e_i\pm e_j\},\;\Theta_1=\{\frac{\pm e_1\pm e_2\pm e_3}{2}\},$$
$$\Theta_2=\{\pm e_i\},\;1\leq i,j\leq 3,i\neq j.$$
Similarly, we can define $\Phi_{i}^{+}$ and $\Theta_{i}^{+}$ for $i=1,2$. Set
$$\beta(\theta)=\frac{\Pi_{i\in \{1,2\}}\Pi_{\alpha\in \Phi_{i}^{+}}1-q^{-i}e^{\alpha^{\vee}} }{\Pi_{i\in \{1,2\}}\Pi_{\gamma^{\vee}\in \Theta_{i}^{+}}1-q^{-i/2}e^{\gamma^{\vee}}}.$$
Then it is clear that 
$$\zeta(1)^{-1}\zeta_{E}^{-3}(1)\beta(\theta)\beta(\theta^{-1})=\frac{L(1/2,\pi,\wedge^3)}{L(1,\pi,\Ad)}.$$
The next lemma is an analogue of Lemma \ref{lem constant nonreductive} for the current case.

\begin{lem}
Set $c_{WS}(\theta)=\frac{\Pi_{i\in \{1,2\}}\Pi_{\gamma^{\vee}\in \Theta_{i}^{+}}1-q^{-i/2}e^{\gamma^{\vee}}}{\Pi_{i\in \{1,2\}}\Pi_{\alpha\in \Phi_{i}^{+}}1-e^{\alpha^{\vee}} } (\theta)$. Then $$\sum_{w\in W}c_{WS}(w\theta)$$ 
is independent of $\theta$ and is equal to $\frac{1}{\Delta_{H_0/Z_{G,H}}(1)}=\zeta(2)^{-1}=1-q^{-2}$.
\end{lem}

\begin{proof}
Our goal is to show that ($\theta_i$ are arbitrary variables)
$$\sum_{w\in W}w\big(\frac{\Pi_{\varepsilon_1,\varepsilon_2\in \{\pm 1\}}(1-q^{-1/2}\sqrt{\theta_1\theta_{2}^{\varepsilon_1}\theta_{3}^{\varepsilon_2}})\cdot \Pi_{i=1}^{3}(1-q^{-1}\theta_i)}{(1-\frac{\theta_1}{\theta_2})(1-\frac{\theta_1}{\theta_3})(1-\frac{\theta_2}{\theta_3})(1-\theta_1\theta_2)(1-\theta_1\theta_3)(1-\theta_2\theta_3)}$$
$$\cdot \frac{1}{(1-\theta_1)(1-\theta_2)(1-\theta_3)}\big)$$
is equal to $\zeta(2)^{-1}=1-q^{-2}$. 
Multiplying both the denominator and the numerator by $\theta_{1}^{-5/2}\theta_{2}^{-3/2}\theta_{3}^{-1/2}$, the denominator will be $(W,\sgn)$-invariant. Hence it is enough to show that
$$\sum_{w\in W}\sgn(w)\cdot w\big(\theta_{1}^{-5/2}\theta_{2}^{-3/2}\theta_{3}^{-1/2}\cdot \Pi_{i=1}^{3}(1-q^{-1}\theta_i)$$
$$\cdot \Pi_{\varepsilon_1,\varepsilon_2\in \{\pm 1\}}(1-q^{-1/2}\sqrt{\theta_1\theta_{2}^{\varepsilon_1}\theta_{3}^{\varepsilon_2}})\big)$$
is equal to $1-q^{-2}$ times
\begin{equation}\label{unitary GR constant}
\theta_{1}^{-5/2}\theta_{2}^{-3/2}\theta_{3}^{-1/2}(1-\frac{\theta_1}{\theta_2})(1-\frac{\theta_1}{\theta_3})(1-\frac{\theta_2}{\theta_3})
\end{equation}
$$(1-\theta_1\theta_2)(1-\theta_1\theta_3)(1-\theta_2\theta_3)(1-\theta_1)(1-\theta_2)(1-\theta_3).$$
We need to study the $q^{-k/2}$-coefficients ($0\leq q\leq 10$) of
$$\theta_{1}^{-5/2}\theta_{2}^{-3/2}\theta_{3}^{-1/2}\cdot \Pi_{\varepsilon_1,\varepsilon_2\in \{\pm 1\}}(1-q^{-1/2}\sqrt{\theta_1\theta_{2}^{\varepsilon_1}\theta_{3}^{\varepsilon_2}})\cdot \Pi_{i=1}^{3}(1-q^{-1}\theta_i).$$

For $k=1,3,5,7,9$, the $q^{-k/2}$-coefficients are combinations of 
$$\theta_{1}^{a_1}\theta_{2}^{a_2}\theta_{3}^{a_3},\;a_1,a_2\in\{0,-1,-2\},a_3\in \{1,0,-1\}. $$
For any such triple $(a_1,a_2,a_3)$, we either have $a_i=\pm a_j$ for some $i\neq j$ or we have $a_i=0$ for some $i$. Hence the $(W,\sgn)$-summation of the $q^{-k/2}$-coefficients are all equal to 0 for $k=1,3,5,7,9$.

The $q^0$-coefficient is equal to $\theta_{1}^{-5/2}\theta_{2}^{-3/2}\theta_{3}^{-1/2}$, 
and the $(W,\sgn)$-summation of it is equal to the denominator \eqref{unitary GR constant}.

The $q^{-5}$-coefficient is equal to $-\theta_{1}^{3}\theta_{2}\theta_{3}$,
and the $(W,\sgn)$-summation of it is equal to zero since the powers of $\theta_2$ and $\theta_3$ are equal.

The $q^{-1}$-coefficient is equal to
$$\theta_{1}^{-3/2}\theta_{2}^{-1/2}\theta_{3}^{-1/2}+\theta_{1}^{-3/2}\theta_{2}^{-3/2}\theta_{3}^{1/2}+\theta_{1}^{-3/2}\theta_{2}^{-3/2}\theta_{3}^{-1/2}+\theta_{1}^{-3/2}\theta_{2}^{-3/2}\theta_{3}^{-3/2}$$
$$-\theta_{1}^{-5/2}\theta_{2}^{-1/2}\theta_{3}^{-1/2}+\theta_{1}^{-3/2}\theta_{2}^{-5/2}\theta_{3}^{-1/2}-\theta_{1}^{-5/2}\theta_{2}^{-3/2}\theta_{3}^{1/2}.$$
The $(W,\sgn)$-summation of all the terms except the last two are equal to zero because either two of the powers are equal to other or two of the powers are opposite to each other. The $(W,\sgn)$-summation of $\theta_{1}^{-3/2}\theta_{2}^{-5/2}\theta_{3}^{-1/2}$ and $\theta_{1}^{-5/2}\theta_{2}^{-3/2}\theta_{3}^{1/2}$ are both equal to $-1$ times the denominator \eqref{unitary GR constant}. As a result, the $(W,\sgn)$-summation of the $q^{-1}$-coefficient is equal to 0.

The $q^{-4}$-coefficient is equal to $$\theta_{1}^{1/2}\theta_{2}^{-1/2}\theta_{3}^{-1/2}+\theta_{1}^{1/2}\theta_{2}^{-3/2}\theta_{3}^{1/2}-\theta_{1}^{-1/2}\theta_{2}^{-3/2}\theta_{3}^{1/2}-\theta_{1}^{-1/2}\theta_{2}^{-1/2}\theta_{3}^{-1/2}$$
$$-\theta_{1}^{-1/2}\theta_{2}^{-1/2}\theta_{3}^{1/2}-\theta_{1}^{-1/2}\theta_{2}^{-1/2}\theta_{3}^{3/2}-\theta_{1}^{-1/2}\theta_{2}^{1/2}\theta_{3}^{1/2}.$$
The $(W,\sgn)$-summation of all the terms is equal to zero because either two of the powers are equal to other or two of the powers are opposite to each other. As a result, the $(W,\sgn)$-summation of the $q^{-4}$-coefficient is equal to 0.

The $q^{-2}$-coefficient is equal to $$-\theta_{1}^{-1/2}\theta_{2}^{-1/2}\theta_{3}^{-1/2}-\theta_{1}^{-1/2}\theta_{2}^{-3/2}\theta_{3}^{1/2}-\theta_{1}^{-3/2}\theta_{2}^{1/2}\theta_{3}^{-1/2}-\theta_{1}^{-3/2}\theta_{2}^{-3/2}\theta_{3}^{3/2}$$
$$-2\theta_{1}^{-3/2}\theta_{2}^{-1/2}\theta_{3}^{1/2}-\theta_{1}^{-3/2}\theta_{2}^{-1/2}\theta_{3}^{-1/2}-\theta_{1}^{-3/2}\theta_{2}^{-3/2}\theta_{3}^{1/2}+ \theta_{1}^{-5/2}\theta_{2}^{-1/2}\theta_{3}^{1/2}$$
$$- \theta_{1}^{-1/2}\theta_{2}^{-3/2}\theta_{3}^{-1/2}-2\theta_{1}^{-3/2}\theta_{2}^{-3/2}\theta_{3}^{-1/2}-\theta_{1}^{-1/2}\theta_{2}^{-5/2}\theta_{3}^{-1/2}$$
$$-\theta_{1}^{-1/2}\theta_{2}^{-3/2}\theta_{3}^{-3/2}-\theta_{1}^{-3/2}\theta_{2}^{-1/2}\theta_{3}^{-3/2}-\theta_{1}^{-3/2}\theta_{2}^{-5/2}\theta_{3}^{1/2} .$$
The $(W,\sgn)$-summation of all the terms except the last term is equal to zero because either two of the powers are equal to other or two of the powers are opposite to each other. 
The $(W,\sgn)$-summation of $\theta_{1}^{-3/2}\theta_{2}^{-5/2}\theta_{3}^{1/2}$ is equal to  the denominator \eqref{unitary GR constant}. 

The $q^{-3}$-coefficient is equal to $$\theta_{1}^{-3/2}\theta_{2}^{-3/2}\theta_{3}^{1/2}+\theta_{1}^{-3/2}\theta_{2}^{-1/2}\theta_{3}^{-1/2}+\theta_{1}^{-1/2}\theta_{2}^{-5/2}\theta_{3}^{1/2}+\theta_{1}^{-1/2}\theta_{2}^{-1/2}\theta_{3}^{-3/2}$$
$$+2\theta_{1}^{-1/2}\theta_{2}^{-3/2}\theta_{3}^{-1/2}+\theta_{1}^{-1/2}\theta_{2}^{-3/2}\theta_{3}^{1/2}+\theta_{1}^{-1/2}\theta_{2}^{-1/2}\theta_{3}^{-1/2}- \theta_{1}^{1/2}\theta_{2}^{-3/2}\theta_{3}^{-1/2}$$
$$+ \theta_{1}^{-3/2}\theta_{2}^{-1/2}\theta_{3}^{1/2}+2\theta_{1}^{-1/2}\theta_{2}^{-1/2}\theta_{3}^{1/2}+\theta_{1}^{-3/2}\theta_{2}^{1/2}\theta_{3}^{1/2}$$
$$+\theta_{1}^{-3/2}\theta_{2}^{-1/2}\theta_{3}^{3/2}+\theta_{1}^{-1/2}\theta_{2}^{1/2}\theta_{3}^{-1/2}+\theta_{1}^{-1/2}\theta_{2}^{-3/2}\theta_{3}^{3/2} .$$
The $(W,\sgn)$-summation of all the terms is equal to zero because either two of the powers are equal to other or two of the powers are opposite to each other. 
Hence the $(W,\sgn)$-summation of the $q^{-3}$-coefficient is equal to $0$.

This finishes the proof of the lemma.
\end{proof}

Now by a very similar argument as in Section \ref{sec:S-theta}, our computation of the colors and the lemma above implies that
$$\int_{K}^{\ast} \CY_\theta(k)\ud k=\frac{\Delta_G(1)}{\Delta_{H_0/Z_{G,H}}} \zeta(1)^{-1}\zeta_E(1)^{-3}\cdot \beta(\theta).$$
There are only two differences 
\begin{itemize}
\item The $c$-function function for $\GU_6$ is defined to be 
$$c_{\alpha}(\theta)=\frac{1-q^{-1}e^{\alpha^\vee}}{1-e^{\alpha^{\vee}}}(\theta)$$ 
if the root space of $\alpha$ is one dimensional and is defined to be $$c_{\alpha}(\theta)=\frac{1-q^{-2}e^{\alpha^\vee}}{1-e^{\alpha^{\vee}}}(\theta)$$ 
if the root space of $\alpha$ is two dimensional. This matches our definition of $\beta(\theta)$ and $c_{WS}(\theta)$ for this case.
\item The volume of Iwahori subgroup of $\GU_6$ is equal to $$\Delta_G(1)\zeta(1)^{-1}\zeta_{E}(1)^{-3}\cdot q^{-l(W)}.$$ 
This is why we get $\zeta(1)^{-1}\zeta_E(1)^{-3}$ instead of $\zeta(1)^{-rk(G)}$ for this case.
\end{itemize}

This implies that 
$$I(\phi_\theta)=\frac{\Delta_G(1)}{\Delta_{H_0/Z_{G,H}}(1)}\zeta(1)^{-1}\zeta_E(1)^{-3}\cdot \beta(\theta)\cdot \beta(\theta^{-1})$$
$$=\frac{\Delta_G(1)}{\Delta_{H_0/Z_{G,H}}(1)}\cdot \frac{L(1/2,\pi,\wedge^{3})}{L(1,\pi,\Ad)}.$$

\subsection{The computation for $(\GU_4\times \GU_2,(\GU_2\times \GU_2)^0)$}
In this subsection, we compute the local relative character for the model $(\GU_4\times \GU_2,(\GU_2\times \GU_2)^0)$. We first study the open Borel orbit. Let $B_{2n}$ be the upper triangular Borel subgroup of $\GU_{n,n}$ and $B=B_4\times B_2$ be a Borel subgroup of $G$. We write $B=TN$ and let $\bar{B}=T\bar{N}$ be the opposite Borel subgroup.

Set $\eta^{-1}=\begin{pmatrix}1&0&0&0\\-1&1&0&0\\1&0&1&0\\1&-1&1&1\end{pmatrix}$ and $\eta=\begin{pmatrix}1&0&0&0\\1&1&0&0\\-1&0&1&0\\1&1&-1&1\end{pmatrix}$. 
The proofs of the following two lemmas are similar to the $(\GSp_6\times \GSp_4,(\GSp_4\times \GSp_2)^0)$ case, and we will skip them here.

\begin{lem}
The double cosets $B(F)\back G(F)/H(F)$ contain a unique open orbit $B(F)(\eta, I_2) H(F)$.
\end{lem}

\begin{lem}
For all $n\in \bar{N}(\varpi\CO_F)$, we have
$$n(\eta,I_2)\in T(\CO_F)N(\varpi\CO_F)(\eta,I_2)H(\CO_F).$$
\end{lem}

Now all the arguments in Section \ref{sec:reductive-case} still work for the current case, the only exception is that the equation in Lemma \ref{lem Haar measure} will become
\begin{equation}
\int_{G(F)}\Phi(g)\ud g=\frac{\Delta_{G}(1)}{\Delta_{H/Z_{G,H}}(1)}\zeta_E(1)^{-3}\zeta(1)^{-2}
\end{equation}
$$\cdot \int_{H(F)/Z_{G,H}(F)}\int_{B(F)}\Phi(b\eta h)\ud b\ud h.$$
This implies that 
$$I(\phi_\theta)=\frac{\Delta_{H/Z_{G,H}}(1)}{\Delta_{G}(1)}\zeta_E(1)^{3}\zeta(1)^2\cdot \int_{K} \CY_{\theta^{-1}}(k)\ud k\cdot \int_{K} \CY_{\theta}(k)\ud k.$$

Next we compute the colors. For this model, since the representation $\pi$ of $G(F)$ is of trivial central character, the associated $L$-parameter factors through the L-group of $\GU_4\times \GU_2/(Res_{E/F}\GL_1)^{diag}$, which is a subgroup of the L-group of $\GU_6$. The 20-dimensional representation $\rho_X=\wedge^2\otimes {\rm std}_2\oplus {\rm std}_4\oplus  {\rm std}_{4}^{\vee}$ in this case is the restriction of the 20 dimensional exterior cube representation of ${}^L\GU_6$ to ${}^L(\GU_4\times \GU_2/(Res_{E/F}\GL_1)^{diag})$. Let $\Theta$ be the weights of the representation $\wedge^2\otimes {\rm std}_2\oplus {\rm std}_4\oplus {\rm std}_{4}^{\vee}$. We can write it as $$\Theta=\{\frac{\pm e_1\pm e_2\pm e_1'}{2},\;\pm e_1',\;\pm e_i \mid 1\leq i\leq 2\}.$$ 
The weight spaces of $\pm e_1',\;\pm e_i$ are two dimensional and the weight spaces of $\frac{\pm e_1\pm e_2\pm 2e_1'}{2}$ are one dimensional.

Let $\alpha_1=\varepsilon_1-\varepsilon_2,\alpha_2=2\varepsilon_2$ and $\alpha'=2\varepsilon_1'$ be the simple roots of $G$. We can define $I_{\alpha_1},I_{\alpha_2}$ and $I_{\alpha'}$ as in the previous case. For $\alpha_1$, the root space is two dimensional and we have the matrix identity
\begin{equation}
(u_{-\alpha_1}(x+y\sqrt{\varepsilon})\eta,I_2)=(b,h^{-1})\cdot (\eta,I_2)\cdot (g,h)
\end{equation}
with $(b,h^{-1})\in B(F), (g,h)\in H(F)$ where
$$u_{-\alpha_1}(x+y\sqrt{\varepsilon})=\begin{pmatrix}1&0&0&0\\x+y\sqrt{\varepsilon}&1&0&0\\0&0&1&0\\0&0&-x+y\sqrt{\varepsilon}&1\end{pmatrix},\;h=\begin{pmatrix}1&y\sqrt{\varepsilon}\\0&1+x \end{pmatrix},$$
$$g=\begin{pmatrix}1+x&0&0&0\\0&1&y\sqrt{\varepsilon}&0\\0&0&1+x&0\\ -y\sqrt{\varepsilon}&0&0&1\end{pmatrix},\; b=\begin{pmatrix}\frac{1}{1+x}&0&0&0\\0&1&\frac{-y\sqrt{\varepsilon}}{1+x}&0\\0&0&\frac{1}{1+x}&0\\0&0&0&1\end{pmatrix}.$$
By the same argument as in the $\alpha_1$ case in the previous subsection, we have
$$I_{\alpha_1}(\theta)=q^2(1-q^{-1})\cdot \frac{1-q^{-2}e^{\alpha_{1}^{\vee}}(\theta)}{(1-q^{-1/2}e^{\beta_{\alpha_1}^{\vee}}(\theta))(1-q^{-1/2}e^{\alpha_{1}^{\vee}-\beta_{\alpha_1}^{\vee}}(\theta))}$$
with $\beta_{\alpha_1}^{\vee}=\frac{e_1-e_2-e_1'}{2}$ and $\alpha_{1}^{\vee}-\beta_{\alpha_1}^{\vee}=\frac{e_1-e_2+e_1'}{2}$.

For $\alpha_2$, the root space is one dimensional and we have the matrix identity
\begin{equation}
(u_{-\alpha_2}(x)\eta,I_2)=(b,h^{-1})\cdot (\eta,I_2)\cdot (g,h)
\end{equation}
with $(b,h^{-1})\in B(F), (g,h)\in H(F)$ where
$$u_{-\alpha_2}(x)=\begin{pmatrix}1&0&0&0\\0&1&0&0\\0&x\sqrt{\varepsilon}&1&0\\0&0&0&1\end{pmatrix},\;g=\begin{pmatrix}1&0&0&x\sqrt{\varepsilon}\\0&1+x\sqrt{\varepsilon}&0&0\\0&0&1+x\sqrt{\varepsilon}&0\\ x\sqrt{\varepsilon}&0&0&1\end{pmatrix}, $$
$$h=(1+x\sqrt{\varepsilon})I_2,\;b=\frac{1}{1-x^2\varepsilon}\begin{pmatrix}1-x\sqrt{\varepsilon}&x\sqrt{\varepsilon}&-x\sqrt{\varepsilon}&-x\sqrt{\varepsilon}\\0&1&-x\sqrt{\varepsilon}&-x\sqrt{\varepsilon}\\0&0&1-x^2\varepsilon&-x^2\varepsilon+x\sqrt{\varepsilon}\\0&0&0&1-x\sqrt{\varepsilon}\end{pmatrix}.$$
By the same argument as in the $\alpha_3$ case in the previous subsection (use the fact that all the unramified characters have value 1 at $a\pm \sqrt{\varepsilon}$ for $a\in \CO_F$), we have
$$I_{\alpha_2}(\theta)=q+1=(q+1)\cdot \frac{1-q^{-1}e^{\alpha_{2}^{\vee}}(\theta)}{1-q^{-1}e^{\beta_{\alpha_2}^{\vee}}(\theta)},$$
with $\beta_{\alpha_2}^{\vee}=\alpha_{2}^{\vee}-\beta_{\alpha_2}^{\vee}=e_2.$

For $\alpha'$, the root space is one dimensional and it can be reduced to $\alpha_2$ but we need to change $u_{-\alpha_2}(x)\eta$ to $\eta u_{-\alpha_2}(-x)$. We have the matrix identity
\begin{equation}
(\eta u_{-\alpha_2}(-x),I_2)=(b,h^{-1})\cdot (\eta,I_2)\cdot (g,h)
\end{equation}
with $(b,h^{-1})\in B(F), (g,h)\in H(F)$ where
$$g=\begin{pmatrix}\frac{1}{1-x\sqrt{\varepsilon}}&0&0&\frac{-x\sqrt{\varepsilon}}{1-x\sqrt{\varepsilon}}\\0&1+x\sqrt{\varepsilon}&\frac{x\sqrt{\varepsilon}}{1-x\sqrt{\varepsilon}}&0\\0&0&\frac{1}{1-x\sqrt{\varepsilon}}&0\\ \frac{-x\sqrt{\varepsilon}}{1-x\sqrt{\varepsilon}}&0&0&\frac{1}{1-x\sqrt{\varepsilon}}\end{pmatrix},h=\begin{pmatrix}1-x\sqrt{\varepsilon}&\frac{-x\sqrt{\varepsilon}}{1+x\sqrt{\varepsilon}}\\0&\frac{1}{1+x\sqrt{\varepsilon}} \end{pmatrix}, $$
$$b=\begin{pmatrix}1&\frac{-x\sqrt{\varepsilon}}{1+x\sqrt{\varepsilon}}&\frac{x\sqrt{\varepsilon}}{1+x\sqrt{\varepsilon}}&\frac{x\sqrt{\varepsilon}}{1+x\sqrt{\varepsilon}}\\0&\frac{1-x\sqrt{\varepsilon}}{1+x\sqrt{\varepsilon}}&0&\frac{x\sqrt{\varepsilon}}{1+x\sqrt{\varepsilon}}\\0&0&\frac{1-x\sqrt{\varepsilon}}{1+x\sqrt{\varepsilon}}&\frac{-x\sqrt{\varepsilon}}{1+x\sqrt{\varepsilon}}\\0&0&0&1\end{pmatrix}.$$
By the same argument as in the $\alpha_3$ case in the previous subsection, we have 
$$I_{\alpha'}(\theta)=q+1=(q+1)\cdot \frac{1-q^{-1}e^{\alpha'^{\vee}}(\theta)}{1-q^{-1}e^{\beta_{\alpha'}^{\vee}}(\theta)},$$
with $\beta_{\alpha'}^{\vee}=\alpha'^{\vee}-\beta_{\alpha'}^{\vee}=e_1'.$ Then we compute the set $\Theta^+$.

\begin{lem}
Let $W=(S_2\ltimes (\BZ/2\BZ)^2)\times (\BZ/2\BZ)$ be the Weyl group of $G$ and let $\Theta^+$ be the smallest subset of $\Theta$ satisfying the following two conditions:
\begin{enumerate}
\item $\frac{e_1-e_2\pm e_1'}{2},e_2,e_1'\in \Theta^+$.
\item $\Theta^+-(\Theta^+\cap w_{\alpha_1}\Theta^+)=\{\frac{e_1-e_2\pm e_1'}{2}\}$, $\Theta^+-(\Theta^+\cap w_{\alpha_2}\Theta^+)=\{e_2\},$ $\Theta^+-(\Theta^+\cap w_{\alpha'}\Theta^+)=\{e_1'\}$.
\end{enumerate}
Then we have $\Theta^+=\{e_1,e_2,e_1',\frac{e_1\pm e_2\pm e_1'}{2}\}$.
\end{lem}

\begin{proof}
It is clear that the set  $\{e_1,e_2,e_1',\frac{e_1\pm e_2\pm e_1'}{2}\}$ satisfies the two conditions. So we just need to show that the set is the unique subset of $\Theta$ satisfying these conditions. The argument is exactly the same as the case $(\GSp_6\times \GSp_4, (\GSp_4\times \GSp_2)^0)$ in Proposition \ref{theta proposition}. We will skip it here.
\end{proof}

Now as in the previous case, we decompose $\Theta$ as $\Theta_1\cup \Theta_2$ and $\Phi=\Phi_1\cup \Phi_2$ where $\Theta_i,\Phi_i$ contain the weights/roots whose weight spaces/root spaces are $i$ dimensional: 
$$\Phi_1=\{\pm 2e_i,\pm 2e_1'\},\;\Phi_2=\{\pm e_1 \pm e_2\},$$
$$\Theta_1=\{\frac{\pm e_1\pm e_2\pm e_1'}{2}\},\;\Theta_2=\{\pm e_i,\pm e_1'\},\;1\leq i\leq 2.$$
Similarly, we can define $\Phi_{i}^{+}$ and $\Theta_{i}^{+}$ for $i=1,2$. Set
$$\beta(\theta)=\frac{\Pi_{i\in \{1,2\}}\Pi_{\alpha\in \Phi_{i}^{+}}1-q^{-i}e^{\alpha^{\vee}} }{\Pi_{i\in \{1,2\}}\Pi_{\gamma^{\vee}\in \Theta_{i}^{+}}1-q^{-i/2}e^{\gamma^{\vee}}}.$$
Then it is clear that 
$$\zeta(1)^{-2}\zeta_{E}^{-3}(1)\beta(\theta)\beta(\theta^{-1})=\frac{L(1/2,\pi,\rho_X)}{L(1,\pi,\Ad)}.$$
The next lemma is an analogue of Lemma \ref{lem constant reductive} for the current case.

\begin{lem}
Set $c_{WS}(\theta)=\frac{\Pi_{i\in \{1,2\}}\Pi_{\gamma^{\vee}\in \Theta_{i}^{+}}1-q^{-i/2}e^{\gamma^{\vee}}}{\Pi_{i\in \{1,2\}}\Pi_{\alpha\in \Phi_{i}^{+}}1-e^{\alpha^{\vee}} }(\theta)$. Then 
$$\sum_{w\in W}c_{WS}(w\theta)$$ 
is independent of $\theta$ and is equal to 
$$\frac{1}{\Delta_{H/Z_{G,H}}(1)}=\zeta(2)^{-2}L(1,\eta_{E/F})^{-1}=(1-q^{-2})^2(1+q^{-1}).$$
\end{lem}

\begin{proof}
Since $H$ is reductive, Theorem 7.2.1 of \cite{Sa} implies that the summation is independent of $\theta$. Now we let $\theta=\delta_{B}^{1/2}$. The lemma follows from the following two claims:
\begin{enumerate}
\item $c_{WS}(w\theta)$ is zero unless $w$ is the longest Weyl element.
\item If $w$ is the longest Weyl element, we have $c_{WS}(w\theta)=(1-q^{-2})^2(1+q^{-1})$.
\end{enumerate}
The second claim is easy to prove so we will focus on the first one. Let $w=(s,s')\in W$ with $s\in S_2\ltimes (\BZ/2\BZ)^2$ and $s'\in \BZ/2\BZ$ so that $c_{WS}(w\theta)$ is nonzero. 

The factor $1-q^{-1}e^{e_1'}(w\theta)$ in the numerator of $c_{WS}(w\theta)$ forces $s'$ to be the longest Weyl element of $\GU_{1,1}$. The factors $1-q^{-1}e^{e_i}(w\theta),\;i=1,2$ in the numerator force  $s(e_1),s(e_2)\in \{\pm e_1,-e_2\}$. Hence there are four possibilities of $s$: $s(e_1)=\pm e_1, s(e_2)=-e_2$ or $s(e_2)=\pm e_1,s(e_1)=-e_2$. If $s(e_2)=\pm e_1,s(e_1)=-e_2$ or $s(e_1)=e_1,s(e_2)=-e_2$, one of the factors $1-q^{-1/2}e^{e_1\pm e_2 + e_1'}(w\theta)$ in the numerator is equal to 0. Hence we must have $s(e_1)=-e_1, s(e_2)=-e_2$, i.e. $w$ is the longest Weyl element. This proves the lemma.
\end{proof}

As in the previous case, our computation of the colors and the lemma above imply that
$$\int_{K} \CY_\theta(k)\ud k=\frac{\Delta_G(1)}{\Delta_{H/Z_{G,H}}} \zeta(1)^{-2}\zeta_E(1)^{-3}\cdot \beta(\theta).$$
This implies that 
$$I(\phi_\theta)=\frac{\Delta_G(1)}{\Delta_{H/Z_{G,H}}(1)}\zeta(1)^{-2}\zeta_E(1)^{-3}\cdot \beta(\theta)\cdot \beta(\theta^{-1})$$
$$=\frac{\Delta_G(1)}{\Delta_{H/Z_{G,H}}(1)}\cdot \frac{L(1/2,\pi,\rho_X)}{L(1,\pi,\Ad)}.$$

\section{The model $(E_7,\PGL_2\ltimes U)$}\label{sec:E7}
 In this section, we compute the local relative character of the model $(E_7,\PGL_2\ltimes U)$. We closely follow the six steps in Section \ref{sec:6-steps}.

To define this model, we recall a description of the adjoint group  of type $E_{7}$, following notation in \cite{P20}.  
Let $H_3(\BH)$ be the degree three central simple Jordan algebra over $k$.
Here $\BH$ is a quaternion algebra over $k$ and denote by $N$ its norm map, $\tr$ the trace, and $x\mapsto x^*$ its conjugation. 
More precisely, one may realize $H_3(\BH)$ as  the vector space of all $3\times 3$ Hermitian symmetric matrices over $\BH$, which are of form
\begin{equation}\label{eq:J}
J=\begin{pmatrix}
a&z&y^*\\
z^*&b&x\\
y&x^*&c 	
\end{pmatrix},
\end{equation}
where $x,y,z\in \BH$ and $a,b,c\in k$.
The Jordan algebra on $H_3(\BH)$ is defined by the composition $J_1\circ J_2:=\frac{1}{2}(J_1J_2+J_2J_1)$ for $J_1,J_2\in H_3(\BH)$,
where $J_1J_2$ and $J_2J_1$ are under the matrix multiplications.
The cubic norm $\det$ on $H_3(\BH)$ is defined by
\begin{equation} \label{eq:cubic}
 \det(J):=abc-aN(x)-bN(y)-cN(z)+\tr(xyz),  
\end{equation}
and the adjoint map $\sharp$ is
\[
J^{\sharp}:=\begin{pmatrix}
bc-N(x)&y^*x^*-cz&zx-by^*\\
xy-cz^*&ac-N(y)&z^*y^*-ax\\
x^*z^*-by&yz-ax^*&ab-N(z)
\end{pmatrix}.
\]
Denote by $(\cdot,\cdot,\cdot)$ the symmetric trilinear form corresponding to the cubic norm $\det$ with $(A,A,A)=\det(A)$ for $A\in H_3(\BH)$.

In \cite{R97}, Rumelhart constructed the Lie algebra $\Fg(H_3(\BH))$ through a $\BZ_3$-grading. (Here we following the notation in \cite[Section 4.2]{P20}.)
More precisely, define
\begin{equation}\label{eq:Fg-E7}
\Fg={\mathfrak {sl}}_3\oplus \Fm^0\oplus V_3\otimes H_3(\BH)\oplus V_3^\vee\otimes H_3(\BH)^\vee 	
\end{equation}
where $V_3$ and $V_3^\vee$ are the standard representation of ${\mathfrak {sl}}_3$ and its dual representation, respectively. 
Here let $\Fm^0$ be the Lie algebra consisting of all linear transformations $\phi$ on $H_3(\BH)$ such that 
 \[
(\phi(z_1),z_2,z_3)+(z_1,\phi(z_2),z_3)+(z_1,z_2,\phi(z_3))=0 \quad 
 \] 
for all  $z_1, z_2, z_3\in H_3(\BH)$. And we refer the reader to Section 4.2.1 in \cite{P20} for the description of the Lie bracket on $\Fg(H_3(\BH))$.

Now, let us consider the identity component of   the automorphism group $\Aut(\Fg(H_3(\BH)))$, which is the quaternionic adjoint group of type $E_7$.
In particular, if $\BH$ is split, then it is the split adjoint group of $E_7$, denoted by $G$.
If $\BH$ is not split, then we denote it by $G_D$, which is of type $E_{7,4}$ and of $k$-rank 4.

Next, let us explicate this model for the split case.
In this case, the quaternion $\BH$ is split and take $\BH=M_{2\times 2}(F)$ with
 \[
x^*=\begin{pmatrix}d&-b\\-c&a \end{pmatrix},\quad
{\rm tr}(x)=a+d,\quad
N(x)=\det(x)=ad-bc,
\]
for $x=\begin{pmatrix}a&b\\c&d \end{pmatrix}\in\BH$.
We may identify $H_3(\BH)$ to $\{A\in M_{6\times 6}(F)\colon A=\Gamma A^{t}\Gamma^{-1}\}$ as follows
\begin{align*}
\begin{pmatrix}
a&z&y^*\\
z^*&b&x\\
y&x^*&c 	
\end{pmatrix}\mapsto  
\begin{pmatrix}
aI_2&z&y^*\\
z^*&bI_2&x\\
y&x^*&cI_2 	
\end{pmatrix}\in M_{6\times 6}(F),
\end{align*}
where  
$\Gamma=\diag\{\begin{pmatrix}0&1\\-1&0	\end{pmatrix},\begin{pmatrix}0&1\\-1&0	\end{pmatrix},\begin{pmatrix}0&1\\-1&0	\end{pmatrix}\}$.
Then the cubic norm $\det$ in \eqref{eq:cubic} on $H_3(\BH)$ is given by 
$\det(A)={\rm Pf}(\Gamma A)$
where ${\rm Pf}$ is the Pfaffian of the skew-symmetric matrices.

The Lie algebra $\Fm^0(F)$ is isomorphic to ${\frak {sl}}_6(F)$ via the action of ${\frak {sl}}_6(F)$ on $H_3(\BH)$ given by $A\cdot X:=AX+XA^*$ for $X\in H_3(\BH)$
where $A^*=\Gamma {}^{t}\!A\Gamma^{-1}$.
Consider $V_3$ and $V_3^\vee$ in \eqref{eq:Fg-E7} as the 3-dimensional vector spaces of column vectors. 
The action of ${\mathfrak {sl}}_3(F)$ on $V_3$ and $V_3^\vee$ are given by: for $v\in V_3$, $\delta\in V_3^\vee$, and $\phi\in {\mathfrak {sl}}_3(F)$,
\begin{equation}\label{eq:sl3-V}
\phi(v)=\phi v \text{ and } \phi(\delta)=-{}^t\phi\delta	
\end{equation}
where the products in the right hand sides are the matrix multiplications.

For $A\in\GL_6(F)$, define $\Phi_A\in End(H_3(\BH))$ by 
\[
\Phi_A(X):=AXA^* \text{ and }\Phi_A^\vee(X):=(A^*)^{-1}XA^{-1}.
\]
Write
\[
(\GL_3\times\GL_6\times \GL_1)^0=\{(a,g,\lambda)\in \GL_3\times\GL_6\times \GL_1\mid \lambda^3{\rm Det}(g){\rm Det}(a)=1 \}
\]
where ${\rm Det}$ is the usual determinant of $\GL_n$.
Define the map $\iota$ from $(\GL_3\times\GL_6\times \GL_1)^0$ to $\GL(\Fg)$ as
\begin{align*}
\iota\colon &\phi\mapsto a\phi a^{-1}\\
&A\mapsto g A g^{-1}\\
&v\otimes X\mapsto ( a v)\otimes \lambda\Phi_A(X)\\
&\delta\otimes \gamma\mapsto   (a^{t})^{-1} \delta \otimes \lambda^{-1}\Phi_A^\vee(\gamma). 
\end{align*}
By a straightforward computation, we have the image of $\iota$ lies in $G$.
Moreover, the kernel of $\iota$ is $\ker \iota=\{(wI_{3},zI_{6},(wz^2)^{-1})\mid w,z\in F^\times\}\cong F^\times\times F^\times$.

We take the unipotent subgroup $U$ of Lie algebra $\Fu$ consisting of elements
\[
\{\begin{pmatrix}0&v_1&v_3\\&0&v_2\\&&0\end{pmatrix}\in {\frak {sl}}_3\}
\oplus \{\begin{pmatrix}0_{2\times 2}&x&y\\&0_{2\times 2}&z\\&&0_{2\times 2} 	\end{pmatrix}\in {\frak {sl}}_6\}\]
\[\oplus Fw_1\otimes H_3(\BH)\oplus Fw_2\otimes H_3(\BH)\oplus Fw_3\otimes H_3(\BH)^\vee,
\]
where $\{w_1,w_2,w_3\}$ is the standard basis of $F^3$.
Then its corresponding Levi subgroup  $L$ is given by the image
\[
\iota(\{(\begin{pmatrix} a&&\\&b&\\&&c 	\end{pmatrix}, 
  \begin{pmatrix}
    g_1&&\\& g_2&\\&& g_3\\  	
   \end{pmatrix},\lambda)\mid {\rm Det}(g_1){\rm Det}(g_2){\rm Det}(g_3)abc=\lambda^{-3} \}).
\]
For $u\in\Fu$, define the character $\xi$ of $U$ by
\[
\xi(\exp(u))=\psi(v_1+{\rm Tr}(x)+{\rm Tr}(z)+e)
\]
where $e$ is the entry corresponding to the simple root $\alpha_3=e_2-e_1$, i.e. the coefficient of $w_2\otimes (E_{5,5}+E_{6,6})$.
($E_{i,j}$ are the elementary matrices in $M_{6\times 6}(F)$.)
The stabilizer $H_0$ of $\xi$ is given by the image
\[
\iota(\{(\begin{pmatrix} a&&\\&a&\\&&c 	\end{pmatrix}, 
  \begin{pmatrix}
    g&&\\& g&\\&& g\\  	
   \end{pmatrix},\lambda)\mid a\lambda {\rm Det}(g)=1,~ {\rm Det}(g)^3a^2c=\lambda^{-3}\})\]
\[=\iota(\{(aI_3, 
  \begin{pmatrix}
    g&&\\& g&\\&& g\\  	
   \end{pmatrix},\lambda)\mid a\lambda {\rm Det}(g)=1\}),
\]
which is isomorphic to $\PGL_2(F)$.
Let $H=H_0\ltimes U$ and we extend the character $\xi$ to $H$ by making it trivial on $H_0$. 
The model $(G,H,\xi)$ is the Whittaker induction of the trilinear $\GL_2$ model $(L,H_0,\xi)$.
We can also define the quaternion (non-split) version of this model by letting $G_D$ be of type $E_{7,4}$. 
In the non-split case, $L_D\rtimes U_D$ is a minimal parabolic subgroup of $G_D$ defined over $F$ and $\xi_D$ is a generic character of $U_D$. 
Then the stabilizer $H_{0,D}$ of $\xi_D$ in $L_D$ is isomorphic to $PD^\times$.
Thus we obtain the quaternion (non-split) version $(G_D,H_D,\xi_D)$ with $H_D=H_{0,D}\rtimes U_D$.

Define the Weyl element $w_0$ of $E_7$ by
\begin{align*}
w_0\colon & \phi\in {\frak {sl}}_3\mapsto -\phi^t\in {\frak {sl}}_3\\
& A\in \mathfrak{sl}_6\mapsto -A^*\in \mathfrak{sl}_6\\
&v\otimes X\in V_3\otimes H_6\mapsto v\otimes X\in V^\vee_3\otimes H^\vee_6\\
&\delta\otimes\gamma\in V^\vee_3\otimes H^\vee_6\mapsto \delta\otimes\gamma\in V_3\otimes H_6.
\end{align*}
Then $w_0^2=1$ and $w_0$ sends $U$ to its opposite.
It is clear that the $w_0$-conjugation map stabilizes $L$ and fixes $H_0$.
We define the map $a:\GL_1\rightarrow Z_L$ to be
\[
a(t)=\iota(\begin{pmatrix}t&&\\&1&\\&&t^{-4}  \end{pmatrix}, \begin{pmatrix}tI_2&&\\&I_2&\\&&t^{-1} I_2 \end{pmatrix}, t).
\]
This clearly satisfies \eqref{the map a}. For the open Borel orbit, let $\eta_0=w'\gamma_0$ where
\[
w'=\iota(I_3, 
 \begin{pmatrix}
	I_2&&&&\\&0&1&&\\&1&0&&\\&&&0&1\\&&&1&0  	
	\end{pmatrix},1),\;
\gamma_0=\iota(I_3, 
 \begin{pmatrix}
 	I_4&&\\&1&1\\&0&1\\	
 	\end{pmatrix},1),
\]
be the representative of the open Borel orbit for the model $(L,H_0)$ as in Section \ref{section triple product}, 
and $\eta=\eta_0 w_0$. The relation \eqref{eta0 relation} has already been verified in Section \ref{section triple product}. This finishes the first three steps in Section \ref{sec:6-steps}.

Now we compute the set of colors and also the set $\Theta^+$. 
Following the notation in \cite{B02}, let $\alpha_1=\frac{1}{2}(\varepsilon_1+\varepsilon_8)-\frac{1}{2}\sum_{i=2}^7\varepsilon_i$, $\alpha_2=\varepsilon_1+\varepsilon_2$ and
 $\alpha_{i+1}=\varepsilon_{i}-\varepsilon_{i-1}$ for $3\leq i \leq 6$ be the simple roots.
Let $\Theta$ be the weights of the 56-dimensional irreducible representation of $E_7(\BC)$, corresponding to the 7-th fundamental weight $\omega_7$,
where $\omega_7=e_6+\frac{1}{2}(e_8-e_7)$.
We can write it as 
\[
\Theta=\{\pm e_i \pm \frac{1}{2}(e_8 -e_7 )\mid 1\leq i\leq 6\}
\]
\[
\cup\{\frac{1}{2}\sum_{i=1}^{6} a_i e_i \mid \#\{i\colon a_i=1\}
\text{ is  even and $a_i=\pm 1$}\}.
\]

By the computation of the trilinear $\GL_2$-model in Section \ref{section triple product} and the discussion in Section \ref{sec non-red strategy 2} (in particular, Remark \ref{color in non-red}), we get the set of colors for this case:
$$\beta_{\alpha_7}^{\vee}=\frac{e_1+e_2+e_3-e_4-e_5+e_6}{2},$$
$$\alpha_{7}^{\vee}-\beta_{\alpha_1}^{\vee}=\frac{-e_1-e_2-e_3+e_4-e_5+e_6}{2},$$
$$\beta_{\alpha_5}^{\vee}=\frac{e_1+e_2-e_3+e_4+e_5-e_6}{2}$$
$$\alpha_{5}^{\vee}-\beta_{\alpha_3}^{\vee}=\frac{-e_1-e_2-e_3+e_4-e_5+e_6}{2},$$
$$\beta_{\alpha_2}^{\vee}=\frac{e_1+e_2-e_3+e_4+e_5-e_6}{2},$$
$$\alpha_{2}^{\vee}-\beta_{\alpha_5}^{\vee}=\frac{e_1+e_2+e_3-e_4-e_5+e_6}{2}.$$
Then we verify \eqref{nonred root} for $\alpha_1$, $\alpha_3$, $\alpha_4$ and $\alpha_6$. Let
$$u_{-\alpha_1}(a)=\iota(\begin{pmatrix}
1&&\\a&1&\\&&1	
\end{pmatrix},I_6,1),\;u_{-\alpha_3}(a)=Id+a\cdot {\rm ad}_{w_2\otimes (E_{5,5}+E_{6,6})^\vee},$$
$$u_{-\alpha_4}(a)=\iota(I_3,I_6+aE_{5,4},1),\;u_{-\alpha_6}(a)=\iota(I_3,I_6+aE_{3,2},1).$$
We have the following 4 identities
$$u_{-\alpha_1}(a)\eta=\eta u_{\alpha_1}(a),\;u_{-\alpha_3}(a)\eta=\eta u_{\alpha_3}(a),$$
$$u_{-\alpha_4}(a)\eta=\eta u_{\alpha_4}(-a)u_{e_3+e_1}(-a),\;u_{-\alpha_6}(a)\eta=\eta u_{e_{6}-e_{4}}(-a).$$
This proves \eqref{nonred root} for $\alpha_1$, $\alpha_3$, $\alpha_4$ and $\alpha_6$. In addition, we label the type of each simple root in the following weighted Dynkin Diagram:
\begin{figure}[h!]
\begin{tikzpicture}[inner sep=1mm,scale=1.75]
\node [circle,draw,label=above:$0$,label=below:${\alpha_7,T}$] (1)  at ( 0,0)  {}; 
\node [circle,draw,label=above:$2$,label=below:${\alpha_6,(U,\psi)}$] (3) at ( 1,0) {};
\node [circle,draw,label=above:$0$,label=below:${\alpha_5,T}$] (4) at ( 2,0) {};
\node [circle,draw,label=above:$2$,label=below:${\alpha_4,(U,\psi)}$] (5) at ( 3,0) {}; 
\node [circle,draw,label=above:$2$,label=below:${\alpha_3,(U,\psi)}$] (6)  at (4,0) {};
\node [circle,draw,label=above:$2$,label=below:${\alpha_1,(U,\psi)}$] (7) at ( 5,0) {};
\node [circle,draw,label=left:$0$,label=right:${\alpha_2,T}$] (2)  at ( 3,1) {}; 
\draw  (1) -- (3);
\draw  (3) -- (4);
\draw  (2) -- (5);
\draw  (4) -- (5);
\draw  (5) -- (6);
\draw  (6) -- (7);
\end{tikzpicture}
\caption{Weighted Dynkin Diagram of $E_7$}\label{digram:E7}
\end{figure}
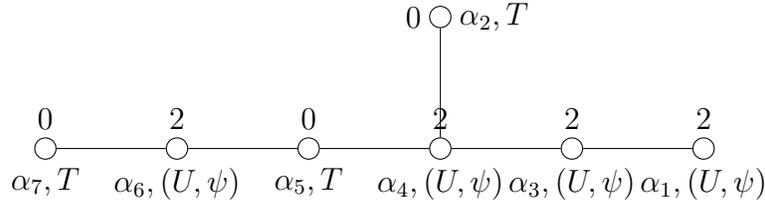

Note that this weighted Dynkin Diagram is associated to the special nilpotent stable orbit of Balar-Carter label $E_6$.
Its corresponding unipotent stable orbit is the maximal unipotent orbit with a non-empty intersection with the unipotent subgroup $U$. 

% \begin{rmk}
% Weighted Dynkin diagram:
% Label: $E_6$, $\dim=120$, $\pi_1(\CO)=1$ and it is special.
% \end{rmk}

Next, we compute the set $\Theta^+$.
\begin{prop}
$\Theta^+$ is consisting of the following 28 elements:
\begin{align}
&\frac{\sum_{l=1}^{6}e_l-2e_i-2e_j}{2},\frac{-\sum_{l=1}^{6}e_l+2e_{i'}+2e_{j'}}{2},\label{eq:E7-Theta-S0-1}\\
&\frac{e_1+e_2+e_3+e_4+e_5+e_6}{2}, \frac{-e_1+e_2+e_3+e_4+e_5+e_6-2e_k}{2},  \label{eq:E7-Theta-S0-3}\\ 
&\pm e_{m}+\frac{e_8-e_7}{2}, \text{ for $1\leq m\leq 6$},\label{eq:E7-Theta-S0-c}
\end{align}
where $(i,j)\in\{(23),(24),(34),(25),(35),(45),(26),(36)\}$, $2\leq k\leq 6$ and $(i',j')\in\{(56),(46)\}$.
% \begin{align}
% & 
% \frac{\sqrt{\chi_6\chi_5\chi_4\chi_1}}{\sqrt{\chi_3\chi_2}},
% \frac{\sqrt{\chi_6\chi_5\chi_3\chi_1}}{\sqrt{\chi_4\chi_2}},
% \frac{\sqrt{\chi_6\chi_5\chi_2\chi_1}}{\sqrt{\chi_4\chi_3}},
% \frac{\sqrt{\chi_6\chi_5}}{\sqrt{\chi_4\chi_3\chi_2\chi_1}},
% \frac{\sqrt{\chi_6\chi_4\chi_3\chi_1}}{\sqrt{\chi_5\chi_2}}, \label{eq:E7-Theta-S0-1}\\
% &
% \frac{\sqrt{\chi_6\chi_4\chi_2\chi_1}}{\sqrt{\chi_5\chi_3}},
% \frac{\sqrt{\chi_6\chi_4}}{\sqrt{\chi_5\chi_3\chi_2\chi_1}},
% \frac{\sqrt{\chi_6\chi_3\chi_2\chi_1}}{\sqrt{\chi_5\chi_4}},
% \frac{\sqrt{\chi_5\chi_4\chi_3\chi_1}}{\sqrt{\chi_6\chi_2}},
% \frac{\sqrt{\chi_5\chi_4\chi_2\chi_1}}{\sqrt{\chi_6\chi_3}},\label{eq:E7-Theta-S0-2}\\ 
% &\sqrt{\chi_1\chi_2\chi_3\chi_4\chi_5\chi_6},
% \frac{\sqrt{\chi_3\chi_4\chi_5\chi_6}}{\sqrt{\chi_1\chi_2}},
% \frac{\sqrt{\chi_2\chi_4\chi_5\chi_6}}{\sqrt{\chi_1\chi_3}},	
% \frac{\sqrt{\chi_2\chi_3\chi_5\chi_6}}{\sqrt{\chi_1\chi_4}},
% \frac{\sqrt{\chi_2\chi_3\chi_4\chi_6}}{\sqrt{\chi_1\chi_5}},
% \frac{\sqrt{\chi_2\chi_3\chi_4\chi_5}}{\sqrt{\chi_1\chi_6}}\label{eq:E7-Theta-S0-3}\\
% % &\pm\varepsilon_{i}+\frac{1}{2}(\varepsilon_8-\varepsilon_7) \text{ for $1\leq i\leq 6$}.\label{eq:E7-Theta-S0-c}
% \end{align}
\end{prop}
\begin{proof}
By the computation of the colors, we know that $\Theta^+$ is the smallest subset of $\Theta$ satisfying the following 5 conditions:
\begin{enumerate}
\item $\frac{1}{2}(e_1+e_2+e_3-e_4-e_5+e_6),\;
\frac{1}{2}(-e_1-e_2-e_3+e_4-e_5+e_6),\;
\frac{1}{2}(e_1+e_2-e_3+e_4+e_5-e_6) \in \Theta^{+}$.
\item $\Theta^{+}-(\Theta^+\cap w_{\alpha_7}\Theta^+)=\{\frac{1}{2}(e_1+e_2+e_3-e_4-e_5+e_6),\;
\frac{1}{2}(-e_1-e_2-e_3+e_4-e_5+e_6)\}$.
\item $\Theta^{+}-(\Theta^+\cap w_{\alpha_5}\Theta^+)=\{\frac{1}{2}(e_1+e_2-e_3+e_4+e_5-e_6),\;
\frac{1}{2}(-e_1-e_2-e_3+e_4-e_5+e_6)\}$.
\item $\Theta^{+}-(\Theta^+\cap w_{\alpha_2}\Theta^+)=\{
\frac{1}{2}(e_1+e_2-e_3+e_4+e_5-e_6),\;
\frac{1}{2}(e_1+e_2+e_3-e_4-e_5+e_6)\}$.
\item $\Theta^{+}$ is stable under $w_{\alpha_6},\;w_{\alpha_4},\; w_{\alpha_3}$ and $w_{\alpha_1}$.
\end{enumerate}
It is clear that the set in the statement satisfies these conditions. So we just need to show that the set is the unique subset of $\Theta$ satisfying these conditions. The argument is exactly the same as the case $(\GSp_6\times \GSp_4, (\GSp_4\times \GSp_2)^0)$ in Proposition \ref{theta proposition}. We will skip it here.
\end{proof}

It is clear that $\Theta^+$ satisfies \eqref{theta plus}. The last thing remains is to prove Lemma \ref{lem constant nonreductive} for the current case. 
Denote by $\Theta_{1}^+$ the subset of $\Theta^+$ consisting of the 12 weights in \eqref{eq:E7-Theta-S0-c} and $\Theta_2^+$ the complement of $\Theta_1^+$ in $\Theta^+$, that is, consisting of the 16 weights in \eqref{eq:E7-Theta-S0-1} and \eqref{eq:E7-Theta-S0-3}.
Then $\Theta_{2}^{+}$ corresponds to the weights of the $\GSO_{12}$  model in Proposition \ref{prop:theta-GSO12}. 
We also decompose the set of positive roots $\Phi^+$ as $\Phi^+_1\cup \Phi_2^+$ where
 $\Phi_{2}^{+}=\{e_j\pm e_i\mid 1\leq i<j\leq 6\}$ is the set of  the roots contained in $\GSO_{12}$,
and  $\Phi_{1}^+$ consists of the remaining positive roots, that is,
\[
e_8-e_7,\; \frac{1}{2}(e_8-e_7+\sum_{i=1}^6(-1)^{a_i}e_i) \text{ with $\sum^6_{i=1}a_i$ odd.}
\]
Denote by $W(D_6)$ the Wely group of the Levi subgroup of type $D_6$, generated by the simple reflections $w_{\alpha_i}$ for $i\ne 1$. We embed $W(D_6)$ into the Weyl group $W$.

\begin{lem}
With the notation above, we have 
$$\sum_{w\in W}c_{WS}(w\theta)=\frac{1}{\Delta_{H_0}(1)}=\frac{1}{\zeta(2)}=(1-q^{-2}).$$
\end{lem}

\begin{proof}
By the identity for the $\GSO_{12}$ model case proved in Lemma \ref{lm:GSO12-constant-lem}, we have 
$$\sum_{w\in W}c_{WS}(w\theta)=\sum_{w\in W}\frac{\prod_{\gamma^\vee\in \Theta^+}1-q^{-\frac{1}{2}}e^{\gamma^\vee}}
{\prod_{\alpha\in \Phi^+}1-e^{\alpha^\vee}}(w\theta)$$
$$=(1-q^{-2})\cdot \sum_{w\in W/W(D_6)}\frac{\prod_{\gamma^\vee\in \Theta_{1}^+}1-q^{-\frac{1}{2}}e^{\gamma^\vee}}
{\prod_{\alpha\in \Phi_{1}^+}1-e^{\alpha^\vee}}(w\theta).$$
Hence it is enough to show that 
$$\sum_{w\in W/W(D_6)}\frac{\prod_{\gamma^\vee\in \Theta_{1}^+}1-q^{-\frac{1}{2}}e^{\gamma^\vee}}
{\prod_{\alpha\in \Phi_{1}^+}1-e^{\alpha^\vee}}(w\theta)=1.$$
It is easy to see that the constant coefficient of the above summation is equal to 1, so it is enough to show that all the $q^{-i/2}$-coefficients are equal to 0 for $1\leq i\leq 12$. 
We can replace the summation on $W/W(D_6)$ by the summation on $W$ and rewrite the function inside the summation ($\theta_i$ are arbitrary variables):
\begin{eqnarray*}
&&\frac{\prod_{\gamma^\vee\in \Theta_{1}^+}1-q^{-\frac{1}{2}}e^{\gamma^\vee}}
{\prod_{\alpha\in \Phi_{1}^+}1-e^{\alpha^\vee}}(w\theta)\\
&=&\frac{e^{-\rho^\vee}\prod_{\alpha\in \Phi^+_2}1-e^{\alpha^\vee}\cdot\prod_{\gamma^\vee\in \Theta_{1}^+}1-q^{-\frac{1}{2}}e^{\gamma^\vee}}
{e^{-\rho^\vee}\prod_{\alpha\in \Phi^+}1-e^{\alpha^\vee}}(w\theta)\\
&=&w\big(\frac{\theta_7^{-\frac{17}{2}}\prod_{i=2}^6\theta_i^{-i+1}\prod_{1\leq i<j\leq 6}(1-\theta_j\theta^{-1}_i)(1+\theta_j\theta_i)
}
{\displaystyle \theta_7^{-\frac{17}{2}}\prod_{i=2}^6\theta_i^{-i+1}\cdot (1-\theta_7)
\prod_{\sum a_i \text{ odd}}(1-\theta_7^{\frac{1}{2}}\prod_{i=1}^6\theta_i^{(-\frac{1}{2})^{a_i}})}\\
&&\cdot \frac{\prod_{i=1}^{6} (1-q^{-1/2}\cdot \theta_i\theta_7^{\frac{1}{2}})(1-q^{-1/2}\cdot \theta^{-1}_i\theta_7^{\frac{1}{2}})}{\prod_{1\leq i<j\leq 6}(1-\theta_j\theta^{-1}_i)(1+\theta_j\theta_i)} \big)
\end{eqnarray*}
where $e^{\rho^\vee}(\theta)=\theta_{7}^{\frac{17}{2}}\prod_{i=2}^6\theta_{i}^{i-1}$.

Then the denominator becomes $(W,\sgn)$-invariant, so it is enough to show that the $(W,\sgn)$-summation of the $q^{-i/2}$-coefficient of 
\begin{equation}\label{E7 constant 1}
\displaystyle \theta_7^{-\frac{17}{2}}\prod_{i=2}^6\theta_i^{-i+1}\prod_{1\leq i<j\leq 6}(1-\theta_j\theta^{-1}_i)(1+\theta_j\theta_i)
\end{equation}
$$\cdot \prod_{i=1}^{6} (1-q^{-1/2}\cdot  \theta_i\theta_7^{\frac{1}{2}} )(1-q^{-1/2}\cdot   \theta^{-1}_i\theta_7^{\frac{1}{2}})$$
is equal to 0 for $1\leq i\leq 12$. We need the following claim which follows from the Weyl Denominator formula of type $D_6$.

\begin{itemize} 
\item[\textbf{Claim}]: the product $$\prod_{i=2}^6\theta_i^{-i+1}(1-\theta_j\theta^{-1}_i)(1+\theta_j\theta_i)=\sum_{w\in W(D_6)}sgn(w)w(\prod_{i=2}^6\theta_i^{i-1})$$
is consisting of terms  of the form
$$\prod_{i=1}^{6} \theta_{i}^{a_i},\;\{|a_1|,|a_2|,|a_3|,|a_4|,|a_5|,|a_6|\}=\{ 5,  4, 3,  2, 1, 0\}.$$
\end{itemize}

Now we can study the coefficients of $q^{-i/2}$. For the $q^{-1/2}$-coefficient, the above claim implies that any term $\prod_{i=1}^{7} \theta_{i}^{b_i}$ appears in the $q^{-1/2}$-coefficient of \eqref{E7 constant 1} satisfies $b_7=-8$ and one of the following two conditions
\begin{itemize}
\item $b_i=\pm b_j$ for some $1\leq i\neq j\leq 6$;
\item $\{|b_1|,|b_2|,|b_3|,|b_4|,|b_5|,|b_6|\}=\{ 6,  4, 3,  2, 1, 0\}$.
\end{itemize}

In the first case, by using a simple reflection in the Weyl group of $D_6$, we know that the $(W,\sgn)$-summation of the term is equal to 0. In the second case, up to a Weyl element $w_0$ action, we may assume that the term is of the form 
$$\theta_7^{-8}\theta_6^6\prod_{i=2}^5\theta_i^{i-1}=e^{-8(e_8-e_7)+\sum^{5}_{i=2}(i-1)e_i+6e_6}(\theta).$$
However by changing variable $\theta_7^{-1}$ to $\theta_7$, 
 the weight 
 $$8(e_8-e_7)+\sum^{5}_{i=2}(i-1)e_i+6e_6$$ 
 is orthogonal to $\alpha_1^\vee$. 
This implies that the $(W,\sgn)$-summation of the $q^{-1/2}$-coefficient is equal to 0.

For the $q^{-1}$-coefficient, the above claim implies that any term $\prod_{i=1}^{7} \theta_{i}^{b_i}$ appears in the $q^{-1}$-coefficient of \eqref{E7 constant 1} satisfies $b_7=-\frac{15}{2}$ and one of the following two conditions
\begin{itemize}
\item $b_i=\pm b_j$ for some $1\leq i\neq j\leq 6$;
\item $\{|b_1|,|b_2|,|b_3|,|b_4|,|b_5|,|b_6|\}=\{ 6,  5, 3,  2, 1, 0\}$ or 
$\{ 5,  4, 3,  2, 1, 0\}$.
\end{itemize}

In the first case, by using a simple reflection in the Weyl group of $D_6$, we know that the $(W,\sgn)$-summation of the term is equal to 0. In the second case, up to  a Weyl element $w_0$ action, we may assume that the term is of the form 
$$e^{-\frac{15}{2}(e_8-e_7)+\sum^{6}_{i=2}(i-1)e_i}(\theta)$$
or
$$e^{-\frac{15}{2}(e_8-e_7)+\sum^{4}_{i=2}(i-1)e_i+5\theta_5+6e_6}(\theta).$$
By changing variable $\theta_7^{-1}\to\theta_7$, 
the weight 
$$\frac{15}{2}(e_8-e_7)+\sum^{6}_{i=2}(i-1)e_i$$ 
is orthogonal to $\alpha_1^\vee$. 
And the weight 
$$w_{\alpha_1}(\frac{15}{2}(e_8-e_7)+\sum^{4}_{i=2}(i-1)e_i+5\theta_5+6e_6)$$
$$=8(e_8-e_7)+\frac{1}{2}(e_1+e_2+3e_3+5e_4+9e_5+11e_6)$$ 
is orthogonal to $\alpha_3^\vee$.
This implies that the $(W,\sgn)$-summation of the $q^{-1}$-coefficient is equal to 0.

For the $q^{-3/2}$-coefficient, the above claim implies that any term $\prod_{i=1}^{7} \theta_{i}^{b_i}$ appears in the $q^{-1/2}$-coefficient of \eqref{E7 constant 1} satisfies $b_7=-7$ and one of the following two conditions
\begin{itemize}
\item $b_i=\pm b_j$ for some $1\leq i\neq j\leq 6$;
\item $\{|b_1|,|b_2|,|b_3|,|b_4|,|b_5|,|b_6|\}=\{ 6,  4, 3,  2, 1, 0\}$ or $\{ 6,  5, 4,  2, 1, 0\}$.
\end{itemize}

In the first case, by using a simple reflection in the Weyl group of $D_6$, we know that the $(W,\sgn)$-summation of the term is equal to 0. In the second case, up to  a Weyl element $w_0$ action, we may assume that the term is of the form 
$$e^{-7(e_8-e_7)+\sum^{5}_{i=2}(i-1)e_i+6e_6}(\theta)$$ 
or
$$e^{-7(e_8-e_7)+\sum^{3}_{i=2}(i-1)e_i+\sum_{i=4}^5ie_i}(\theta).$$
By changing variable $\theta_7^{-1}$ to $\theta_7$, 
the weight 
$$w_{\alpha_1}(7(e_8-e_7)+\sum^{5}_{i=2}(i-1)e_i+6e_6)$$
$$=\frac{15}{2}(e_8-e_7)+\frac{1}{2}(e_1+e_2+3e_3+5e_4+7e_5+11e_6) $$
is orthogonal to $e_2-e_1$; and the weight 
$$w_{\alpha_3}w_{\alpha_1}(7(e_8-e_7)+\sum^{3}_{i=2}(i-1)e_i+\sum_{i=4}^5ie_i)$$
$$=8(e_8-e_7)+ e_2+e_3+3e_4+4e_5+5e_6$$ 
is orthogonal to $e_3-e_2$.
This implies that the $(W,\sgn)$-summation of the $q^{-3/2}$-coefficient is equal to 0.

For the $q^{-2}$-coefficient, the above claim implies that any term $\prod_{i=1}^{7} \theta_{i}^{b_i}$ appears in the $q^{-1/2}$-coefficient of \eqref{E7 constant 1} satisfies $b_7=-\frac{13}{2}$ and one of the following two conditions
\begin{itemize}
\item $b_i=\pm b_j$ for some $1\leq i\neq j\leq 6$;
\item $\{|b_1|,|b_2|,|b_3|,|b_4|,|b_5|,|b_6|\}=\{ 5,  4, 3,  2, 1, 0\}$, $\{ 6,  5, 3,  2, 1, 0\}$ or $\{ 6,  5, 4,  3, 1, 0\}$.
\end{itemize}

In the first case, by using a simple reflection in the Weyl group of $D_6$, we know that the $(W,\sgn)$-summation of the term is equal to 0.
In the second case, up to  a Weyl element $w_0$ action, we may assume that the term is of the form 
$$e^{-\frac{13}{2}(e_8-e_7)+\sum^{6}_{i=2}(i-1)e_i}(\theta),$$
$$e^{-\frac{13}{2}(e_8-e_7)+\sum^{4}_{i=2}(i-1)e_i+\sum^{6}_{i=5}ie_i}(\theta),$$
or
$$e^{-\frac{13}{2}(e_8-e_7)+e_2+\sum^{6}_{i=3}ie_i}(\theta).$$
By changing variable $\theta_7^{-1}$  to $\theta_7$, 
the weight 
$$w_{\alpha_1}(-\frac{13}{2}(e_8-e_7)+\sum^{6}_{i=2}(i-1)e_i)$$
$$=7(e_8-e_7)+\frac{1}{2}(e_1+e_2+3e_3+5e_4+7e_5+9e_6)$$
is orthogonal to $e_2-e_1$; the weight 
$$w_{\alpha_3}w_{\alpha_1}(\frac{13}{2}(e_8-e_7)+\sum^{4}_{i=2}(i-1)e_i+\sum^{6}_{i=5}ie_i)$$
$$=\frac{15}{2}(e_8-e_7)+ e_2+e_3+2e_4+4e_5+5e_6$$ 
is orthogonal to $e_3-e_2$; the weight 
$$w_{\alpha_3}w_{\alpha_1}(\frac{13}{2}(e_8-e_7)+e_2+\sum^{6}_{i=3}ie_i)$$
$$=8(e_8-e_7)+\frac{1}{2}(-e_1+3 e_2+3e_3+5e_4+7e_5+9e_6$$ 
is orthogonal to $e_3-e_2$. This implies that the $(W,\sgn)$-summation of the $q^{-2}$-coefficient is equal to 0.

For the $q^{-5/2}$-coefficient, the above claim implies that any term $\prod_{i=1}^{7} \theta_{i}^{b_i}$ appears in the $q^{-1/2}$-coefficient of \eqref{E7 constant 1} satisfies $b_7=-6$ and one of the following two conditions
\begin{itemize}
\item $b_i=\pm b_j$ for some $1\leq i\neq j\leq 6$;
\item $\{|b_1|,|b_2|,|b_3|,|b_4|,|b_5|,|b_6|\}=\{ 6,  4, 3,  2, 1, 0\}$, $\{ 6,  5, 4,  2, 1, 0\}$, \\
or $\{ 6,  5, 4,  3, 2, 0\}$.
\end{itemize}

In the first case, by using a simple reflection in the Weyl group of $D_6$, we know that the $(W,\sgn)$-summation of the term is equal to 0.
In the second case, up to  a Weyl element $w_0$ action, we may assume that the term is of the form 
$$ e^{-6(e_8-e_7)+\sum^{5}_{i=2}(i-1)e_i+6e_6}(\theta),$$
$$e^{-6(e_8-e_7)+\sum^{3}_{i=2}(i-1)e_i+\sum_{i=4}^6 ie_i}(\theta)$$
or
$$e^{-6(e_8-e_7)+ \sum_{i=2}^6i e_i}(\theta).$$
By changing variable $\theta_7^{-1}$ to $\theta_7$, under the action of $w_{\alpha_1}$, 
we have
\begin{align*}
 &w_{\alpha_1}(0,1,2,3,4,6,-6,6)=(1,0,1,2,3,5,-7,7)\\
 &w_{\alpha_1}(0,1,2,4,5,6,-6,6)=(\frac{3}{2},-\frac{1}{2},\frac{1}{2},\frac{5}{2},\frac{7}{2},\frac{9}{2},-\frac{15}{2},\frac{15}{2})\\
 &w_{\alpha_1}(0,2,3,4,5,6,-6,6)=(2,0,1,2,3,4,-8,8).   
\end{align*}
Here $(b_1,b_2,\dots,b_8)$ corresponds the weight $\sum^{8}_{i=1}b_i e_i$.
In particular we have $b_i=\pm b_j$ for some $1\leq i\ne j\leq 6$ which is just the first case.
This implies that the $(W,\sgn)$-summation of the $q^{-5/2}$-coefficient is equal to 0.

For the $q^{-3}$-coefficient, the above claim implies that any term $\prod_{i=1}^{7} \theta_{i}^{b_i}$ appears in the $q^{-1/2}$-coefficient of \eqref{E7 constant 1} satisfies $b_7=-\frac{11}{2}$ and one of the following two conditions
\begin{itemize}
\item $b_i=\pm b_j$ for some $1\leq i\neq j\leq 6$.
\item $\{|b_1|,|b_2|,|b_3|,|b_4|,|b_5|,|b_6|\}$ is equal to 
\[\{ 5,  4, 3,  2, 1, 0\},~ \{ 6,  5, 3,  2, 1, 0\},~ \{ 6,  5, 4,  3, 1, 0\}, 
\text{ or }
\{ 6,  5, 4,  3, 2,  1\}.
\]
\end{itemize}

In the first case, by using a simple reflection in the Weyl group of $D_6$, we know that the $(W,\sgn)$-summation of the term is equal to 0.
In the second case, up to  a Weyl element $w_0$ action, we may assume that the term is of the form  
$$e^{-\frac{11}{2}(e_8-e_7)+\sum^{6}_{i=2}(i-1)e_i}(\theta),$$
$$e^{-\frac{11}{2}(e_8-e_7)+\sum^{4}_{i=2}(i-1) e_i+\sum_{i=5}^6ie_i}(\theta),$$
$$e^{-\frac{11}{2}(e_8-e_7)+ e_2+\sum_{i=3}^6i e_i}(\theta),$$
or
$$e^{-\frac{11}{2}(e_8-e_7)+\sum^{6}_{i=2}i e_i}(\theta).$$
By changing variable $\theta_7^{-1}$ to $\theta_7$, under the action of $w_{\alpha_1}$, 
we have
\begin{align*}
 &w_{\alpha_1}(0,1,2,3,4,5,-\frac{11}{2},\frac{11}{2})=(1,0,1,2,3,4,-\frac{13}{2},\frac{13}{2})\\
 &w_{\alpha_1}(0,1,2,3,5,6,-\frac{11}{2},\frac{11}{2})=(\frac{3}{2},-\frac{1}{2},\frac{1}{2},\frac{3}{2},\frac{7}{2},\frac{9}{2},-7,7)\\
 &w_{\alpha_1}(0,1,3,4,5,6,-\frac{11}{2},\frac{11}{2})=(2,-1,1,2,3,4,-\frac{15}{2},\frac{15}{2})\\
 &w_{\alpha_1}(1,2,3,4,5,6,-\frac{11}{2},\frac{11}{2})=(3,0,1,2,3,4,-\frac{15}{2},\frac{15}{2}).   
\end{align*} 
After the action of $w_{\alpha_1}$, we have $b_i=\pm b_j$ for some $1\leq i\ne j\leq 6$ which is just the first case.
This implies that the $(W,\sgn)$-summation of the $q^{-3}$-coefficient is equal to 0.

Due to symmetry, the remaining $q^{-i/2}$-coefficients for $7\leq i\leq 12$ are vanishing by similar arguments and we omit the details here.  
This finishes the proof of the lemma.
\end{proof}

To sum up, we have proved that the local relative character is equal to
\[
\zeta(6)\zeta(8)\zeta(10)\zeta(12)\zeta(14)\zeta(18)\frac{L(1/2,\pi,\omega_7)}{L(1,\pi,\Ad)}
\]
where $\pi$ is an unramified representation of $E_7(F)$.

\section{The remaining models}\label{sec:remaining}
In this section, we will compute the local relative characters for the remaining 4 models in Table \ref{fig:1}.
The computations are very similar to the cases in the previous sections.

\subsection{The model $(\GSp_{10},\GL_2\ltimes U)$}
In this subsection, we compute the local relative character for the model $(\GSp_{10},\GL_2\ltimes U)$. 
For simplicity, 
% Let $J_2'=\begin{pmatrix}0&-1\\1&0\end{pmatrix}$ and $J_{2n}'=\begin{pmatrix}0&J_{2n-2}'\\J_2&0 \end{pmatrix}$. 
define
$$\GSp_{2n}=\{g\in \GL_{2n} \mid {}^t gJ_{2n}'g =l(g)J_{2n}'\},
\text{ where }J_{2n}'=\begin{pmatrix}0&J_{2n-2}'\\J_2&0 \end{pmatrix}.
$$
Note that  the skew-symmetric matrix $J'_{2n}$ is different with $J_{2n}$ when $n>1$ in Section \ref{sec:GSp-model} and $J_{2}=J'_2$. We use $J_{2n}'$ here  to simplify the definition and computation.
Let $G=\GSp_{10}$, $H=H_0\ltimes U$ with 
$$H_0=\{diag(h,h,h,\det(h)h^{\ast},\det(h)h^{\ast})\mid h\in \GL_2,\;h^\ast=J_2'{}^th^{-1} (J_2')^{-1}\}$$
$$=\{diag(h,h,h,h,h)|\;h\in \GL_2\}$$
and $U$ be the unipotent radical of the standard parabolic subgroup $P=LU$ of $G$ where 
$$L=\{(h_1,h_2, h_3,\det(h_3)h_{2}^{\ast},\det(h_3)h_{1}^{\ast})\mid  h_i\in \GL_2\}.$$ 
We define a generic character $\xi$ on $U(F)$ to be $\xi(u)=\psi(\lambda(u))$ where 
$$\lambda(u)=\tr(X)+\tr(Y),\;u=\begin{pmatrix}I_2&X&\ast &\ast &\ast \\ 0&I_2&Y&\ast&\ast \\ 0&0&I_2& \ast & \ast \\ 0&0&0&I_2 &\ast \\ 0&0&0&0&I_2 \end{pmatrix}.$$ 
It is easy to see that $H_0$ is the stabilizer of this character and $(G,H)$ is the Whittaker induction of the trilinear $\GL_2$-model $(L,H_0,\xi)$.

We can also define the quaternion version of this model. Let $D/F$ be a quaternion algebra, and let $G_D(F)=\GSp_5(D)$ (the group $\GSp_n(D)$ has been defined in Section 3.1), $H_D=H_{0,D}\ltimes U_D$  with 
$$H_{0,D}(F)=\{diag(h,h,h,h,h) \mid h\in \GL_1(D),\;h^\ast=\bar{h}^{-1} \}$$
and $U_D$ is the unipotent radical of the standard parabolic subgroup $P_D=L_DU_D$ of $G_D$ where 
$$L_D(F)=\{(h_1,h_2, h_3,N_{D/F}(h_3)h_{2}^{\ast},N_{D/F}(h_3)h_{1}^{\ast}) \mid h_i\in \GL_1(D)\}.$$ 
Here $N_{D/F}:\GL_1(D)\rightarrow F^{\times}$ is the norm map  and $x\rightarrow \bar{x}$ is the conjugation map on the quaternion algebra. Like the split case, we can define the character $\xi_D$ on $U_D(F)$ by replacing the trace map of $Mat_{2\times 2}$ by the trace map of $D$.

Let $w_0=\begin{pmatrix}0&0&0&0&I_2\\ 0&0&0&I_2&0\\ 0&0&I_2&0&0\\ 0&I_2&0&0&0\\ I_2&0&0&0&0 \end{pmatrix}$ be the Weyl element that sends $U$ to its opposite. It is clear that the $w_0$-conjugation map stabilizes $L$ and fixes $H_0$. We define the map $a:\GL_1\rightarrow Z_L$ to be
$$a(t)=diag(t^2I_2,tI_2,I_2,t^{-1}I_2,t^{-2}I_2).$$
This clearly satisfies the equation \eqref{the map a}. For the open Borel orbit, let
$$\eta_0=diag(I_2,\begin{pmatrix}0&1\\1&0 \end{pmatrix}, \begin{pmatrix}0&1\\1&0 \end{pmatrix}\begin{pmatrix}1&1\\0&1 \end{pmatrix}, \begin{pmatrix}0&1\\1&0 \end{pmatrix},-I_2)$$ 
be the representative of the open Borel orbit for the model $(L,H_0)$ as in Section \ref{section triple product}, and $\eta=\eta_0 w_0$. The relation \eqref{eta0 relation} has already been verified in Section \ref{section triple product}. This finishes the first three steps in Section \ref{sec:6-steps}.

Now we compute the set of colors and also the set $\Theta^+$. Let $\Theta$ be the weights of the 32-dimensional representation $\Spin_{11}$ of $\GSpin_{11}(\BC)$. We can write it as 
$$\Theta=\{\frac{\pm e_1\pm e_2\pm e_3\pm e_4\pm e_5}{2}\}.$$

Let $\alpha_i=\varepsilon_i-\varepsilon_{i+1}$, $1\leq i\leq 4$ and $\alpha_5=2\varepsilon_5$ be the simple roots of $\GSp_{10}$. By the computation of the trilinear $\GL_2$-model in Section \ref{section triple product} and the discussion in Section \ref{sec non-red strategy 2} (in particular, Remark \ref{color in non-red}), we have
$$\beta_{\alpha_1}^{\vee}=\frac{e_1-e_2-e_3+e_4+e_5}{2},\;\alpha_{1}^{\vee}-\beta_{\alpha_1}^{\vee}=\frac{e_1-e_2+e_3-e_4-e_5}{2},$$ $$\beta_{\alpha_3}^{\vee}=\frac{-e_1+e_2+e_3-e_4+e_5}{2},\;\alpha_{3}^{\vee}-\beta_{\alpha_3}^{\vee}=\frac{e_1-e_2+e_3-e_4-e_5}{2},$$
$$\beta_{\alpha_5}^{\vee}=\frac{-e_1+e_2+e_3-e_4+e_5}{2},\;\alpha_{5}^{\vee}-\beta_{\alpha_3}^{\vee}=\frac{e_1-e_2-e_3+e_4+e_5}{2}.$$
By a similar argument as in the Ginzburg--Rallis model case in Section 5, we can also verify \eqref{nonred root} for the roots $\alpha_2$ and $\alpha_4$. Next, we compute the set $\Theta^+$. 

\begin{prop}
$\Theta^+$ is consisting of the following 16 elements:
$$\frac{e_1+e_2\pm e_3\pm e_4\pm e_5}{2},\frac{e_1-e_2+ e_3\pm e_4\pm e_5}{2},\frac{e_1-e_2-e_3+ e_4+ e_5}{2},$$
$$\frac{-e_1+e_2+ e_3+ e_4\pm e_5}{2},\frac{-e_1+e_2+ e_3- e_4+ e_5}{2}.$$
\end{prop}

\begin{proof}
By the computation of the colors, we know that $\Theta^+$ is the smallest subset of $\Theta$ satisfying the following 5 conditions:
\begin{enumerate}
\item $\frac{e_1-e_2-e_3+e_4+e_5}{2},\;\frac{e_1-e_2+e_3-e_4-e_5}{2},\; \frac{-e_1+e_2+e_3-e_4+e_5}{2}\in \Theta^{+}$.
\item $\Theta^{+}-(\Theta^+\cap w_{\alpha_1}\Theta^+)=\{\frac{e_1-e_2-e_3+e_4+e_5}{2},\;\frac{e_1-e_2+e_3-e_4-e_5}{2}\}$.
\item $\Theta^{+}-(\Theta^+\cap w_{\alpha_3}\Theta^+)=\{\frac{e_1-e_2+e_3-e_4-e_5}{2},\; \frac{-e_1+e_2+e_3-e_4+e_5}{2}\}$.
\item $\Theta^{+}-(\Theta^+\cap w_{\alpha_5}\Theta^+)=\{\frac{e_1-e_2-e_3+e_4+e_5}{2},\; \frac{-e_1+e_2+e_3-e_4+e_5}{2}\}$.
\item $\Theta^{+}$ is stable under $w_{\alpha_2}$ and $w_{\alpha_4}$.
\end{enumerate}
It is clear that the set in the statement satisfies these conditions. So we just need to show that the set is the unique subset of $\Theta$ satisfying these conditions. The argument is exactly the same as the case $(\GSp_6\times \GSp_4, (\GSp_4\times \GSp_2)^0)$ in Proposition \ref{theta proposition}. We will skip it here.
\end{proof}

It is clear that $\Theta^+$ satisfies \eqref{theta plus}. The last thing remains is to prove Lemma \ref{lem constant nonreductive} for the current case. For $i=1,2$, we decompose $\Theta^+$ as $\Theta_{1}^{+}\cup \Theta_{2}^{+}$ with $\Theta_{1}^{+}$ consisting of the following 10 elements:
$$\frac{e_1+e_2+e_3\pm e_4\pm e_5}{2},\frac{e_1+e_2+e_3+ e_4\pm e_5-2e_i}{2},\;1\leq i\leq 3$$
and $\Theta_{2}^+$ consisting of the remaining 6 elements.
Then $\Theta_{2}^{+}$ corresponds to the weights in Lemma \ref{constant for GL(4)xGL(2) non-red}
(here we view $\GL_4\times \GL_2\simeq \GL_4\times \GSp_2$ as a standard Levi subgroup of $\GSp_{10}$). 
Decompose the set of positive roots $\Phi^+$ as $\Phi_{1}^+\cup \Phi_{2}^+$ where $\Phi_{2}^{+}=\{e_i-e_j,2e_5 \mid 1\leq i<j\leq 4\}$ is the set of the positive roots contained in $\GL_4\times \GSp_2$ and $\Phi_{1}^+$ contains the remaining positive roots. 
We also embed the Weyl group $S_4\times S_2$ of $\GL_4\times \GL_2$ into  $W$.

\begin{lem}
With the notation above, we have 
$$\sum_{w\in W}c_{WS}(w\theta)=\frac{1}{\Delta_{H_0/Z_{G,H}}(1)}=\frac{1}{\zeta(2)}=(1-q^{-2}).$$
\end{lem}

\begin{proof}
By the identity in Lemma \ref{constant for GL(4)xGL(2) non-red}, we have 
$$\sum_{w\in W}c_{WS}(w\theta)=\sum_{w\in W}\frac{\prod_{\gamma^\vee\in \Theta^+}1-q^{-\frac{1}{2}}e^{\gamma^\vee}}
{\prod_{\alpha\in \Phi^+}1-e^{\alpha^\vee}}(w\theta)$$
$$=(1-q^{-2})\cdot \sum_{w\in W/S_4\times S_2}\frac{\prod_{\gamma^\vee\in \Theta_{1}^+}1-q^{-\frac{1}{2}}e^{\gamma^\vee}}
{\prod_{\alpha\in \Phi_{1}^+}1-e^{\alpha^\vee}}(w\theta).$$
Hence it is enough to show that 
$$\sum_{w\in W/S_4\times S_2}\frac{\prod_{\gamma^\vee\in \Theta_{1}^+}1-q^{-\frac{1}{2}}e^{\gamma^\vee}}
{\prod_{\alpha\in \Phi_{1}^+}1-e^{\alpha^\vee}}(w\theta)=1.$$
It is easy to see that the constant coefficient of the above summation is equal to 1, so it is enough to show that all the $q^{-i/2}$-coefficients are equal to 0 for $1\leq i\leq 10$. Like in the previous cases, we can replace the summation on $W/S_4\times S_2$ by the summation on $W$. We also need to rewrite the function inside the summation $\frac{\prod_{\gamma^\vee\in \Theta_{1}^+}1-q^{-\frac{1}{2}}e^{\gamma^\vee}}
{\prod_{\alpha\in \Phi_{1}^+}1-e^{\alpha^\vee}}(w\theta)$ as (here $\theta_i$ are arbitrary variables):
$$w\big(\frac{(1-q^{-1/2}\cdot \sqrt{\theta_1\theta_2\theta_3\theta_4\theta_5})\cdot \Pi_{i=1}^{5} (1-q^{-1/2}\cdot \frac{\sqrt{\theta_1\theta_2\theta_3\theta_4\theta_5}}{\theta_i}) }
{\Pi_{1\leq i<j\leq 5} (1-\theta_i\theta_j)\cdot \Pi_{i=1}^{4}(1-\theta_i) \Pi_{i=1}^{4}(1-\theta_i/\theta_5)}
$$
$$\cdot \Pi_{1\leq i\leq 4} (1-q^{-1/2}\cdot \frac{\sqrt{\theta_1\theta_2\theta_3\theta_4\theta_5}}{\theta_i\theta_5})\big)$$
$$=w\big(\frac{\ast \cdot (1-q^{-1/2}\cdot \sqrt{\theta_1\theta_2\theta_3\theta_4\theta_5})\cdot \Pi_{i=1}^{5} (1-q^{-1/2}\cdot \frac{\sqrt{\theta_1\theta_2\theta_3\theta_4\theta_5}}{\theta_i}) }
{\Pi_{1\leq i<j\leq 5} (\theta_{i}^{-1}-\theta_i-\theta_{j}^{-1}+\theta_j)\cdot \Pi_{i=1}^{5}(\theta_{i}^{-1/2}-\theta_{i}^{1/2})}$$
$$\cdot \Pi_{1\leq i\leq 4} (1-q^{-1/2}\cdot \frac{\sqrt{\theta_1\theta_2\theta_3\theta_4\theta_5}}{\theta_i\theta_5})\big)$$
where
$$\ast=\theta_{1}^{-3/2}\theta_{2}^{-3/2}\theta_{3}^{-3/2}\theta_{4}^{-3/2}(\theta_{5}^{-1/2}-\theta_{5}^{1/2})\cdot \Pi_{1\leq i<j\leq 4} (\theta_{i}^{-1}-\theta_{j}^{-1}).$$
Then the denominator becomes $(W,\sgn)$-invariant, so it is enough to show that the $(W,\sgn)$-summation of the $q^{-i/2}$-coefficient of 
\begin{equation}\label{Sp(10) constant 1}
\ast\cdot (1-q^{-1/2}\cdot \sqrt{\theta_1\theta_2\theta_3\theta_4\theta_5})\cdot \Pi_{i=1}^{5} (1-q^{-1/2}\cdot \frac{\sqrt{\theta_1\theta_2\theta_3\theta_4\theta_5}}{\theta_i}) 
\end{equation}
$$\cdot \Pi_{1\leq i\leq 4} (1-q^{-1/2}\cdot \frac{\sqrt{\theta_1\theta_2\theta_3\theta_4\theta_5}}{\theta_i\theta_5})$$
is equal to 0 for $1\leq i\leq 10$. The product $\ast$ consists of terms of the form
$$\Pi_{i=1}^{5} \theta_{i}^{a_i},\;\{a_1,a_2,a_3,a_4\}=\{-9/2,-7/2,-5/2,-3/2\},a_5=\pm1/2.$$
Then the $q^{-5}$-coefficients consisting of terms of the form 
$$\Pi_{i=1}^{5} \theta_{i}^{b_i},\;\{b_1,b_2,b_3,b_4\}=\{-3/2,-1/2,1/2,3/2\},\;b_5=\pm 1/2.$$
The $(W,\sgn)$-summation of these terms is equal to 0 since $b_i=\pm b_j$ for some $i\neq j$.

For the $q^{-1/2}$-coefficient, any term $\Pi_{i=1}^{5} \theta_{i}^{b_i}$ appearing in it 
must satisfy one of the following two conditions
\begin{itemize}
\item $b_i=b_j$ for some $1\leq i< j\leq 4$.
\item $\{b_1,b_2,b_3,b_4\}=\{-5,-3,-2,-1\}$ or $\{-4,-3,-2,-1\}$ and $b_5\in \{-1,1,0\}$.
\end{itemize}
In either case, we have $b_i=\pm b_j$ for some $i\neq j$ or $b_5=0$. This implies that the $(W,\sgn)$-summation of the $q^{-1/2}$-coefficient is equal to 0. 
Similarly, we can also show that  the $(W,\sgn)$-summation of the $q^{-9/2}$-coefficient is equal to 0.

For the $q^{-1}$-coefficient, any term $\Pi_{i=1}^{5} \theta_{i}^{b_i}$ appearing in it 
must satisfy one of the following two conditions
\begin{itemize}
\item $b_i=b_j$ for some $1\leq i< j\leq 4$.
\item $\{b_1,b_2,b_3,b_4\}$ is equal to $\{-11/2,-5/2,-3/2,-1/2\},$
\\
$\{-9/2,-7/2,-3/2,-1/2\},$ $\{-9/2,-5/2,$ $-3/2,-1/2\}$ \\
or $\{-7/2,-5/2,-3/2,-1/2\}$, and $b_5\in \{\pm 3/2,\pm 1/2\}$.
\end{itemize}
In either case, we have $b_i=\pm b_j$ for some $i\neq j$. This implies that the $(W,\sgn)$-summation of the $q^{-1}$-coefficient is equal to 0. Similarly, we can also show that  the $(W,\sgn)$-summation of the $q^{-4}$-coefficient is equal to 0.

For the $q^{-3/2}$-coefficient, any term $\Pi_{i=1}^{5} \theta_{i}^{b_i}$ appearing in it 
must satisfy one of the following two conditions
\begin{itemize}
\item $b_i=b_j$ for some $1\leq i< j\leq 4$.
\item $b_i=0$ for some $1\leq i\leq 4$.
\end{itemize}
This implies that the $(W,\sgn)$-summation of the $q^{-3/2}$-coefficient is equal to 0. Similarly, we can also show that  the $(W,\sgn)$-summation of the $q^{-7/2}$-coefficient is equal to 0.

For the $q^{-2}$-coefficient, any term $\Pi_{i=1}^{5} \theta_{i}^{b_i}$ appearing in it 
must satisfy one of the following two conditions
\begin{itemize}
\item $b_i=\pm b_j$ for some $1\leq i< j\leq 4$.
\item $\{b_1,b_2,b_3,b_4\}$ is equal to $\{-7/2,-5/2,-3/2,-1/2\},$
\\
$\{-9/2,-5/2,-3/2,1/2\},$ or $\{-7/2,-5/2,$ $-3/2,1/2\}$, and $b_5\in \{\pm 5/2,\pm 3/2,\pm 1/2\}$.
\end{itemize}
In either case, we have $b_i=\pm b_j$ for some $i\neq j$. This implies that the $(W,\sgn)$-summation of the $q^{-2}$-coefficient is equal to 0. 
Similarly, we can also show that  the $(W,\sgn)$-summation of the $q^{-3}$-coefficient is equal to 0.

For the $q^{-5/2}$-coefficient, any term $\Pi_{i=1}^{5} \theta_{i}^{b_i}$ appearing in it 
must satisfy one of the following two conditions
\begin{itemize}
\item $b_i=\pm b_j$ for some $1\leq i< j\leq 4$.
\item $b_i=0$ for some $1\leq i\leq 4$.
\end{itemize}
This implies that the $(W,\sgn)$-summation of the $q^{-5/2}$-coefficient is equal to 0. This finishes the proof of the lemma.
\end{proof}

To sum up, we have proved that the local relative character is equal to 
$$\zeta(1)\zeta(4)\zeta(6)\zeta(8)\zeta(10)\frac{L(1/2,\pi,\Spin_{11})}{L(1,\pi,\Ad)}$$
where $\pi$ is an unramified representation of $\GSp_{10}(F)$.

\subsection{The model $(\GSp_{6}\times \GL_2,\GL_2\ltimes U)$}
In this subsection, we compute the local relative character for the model $(\GSp_{6}\times \GL_2,\GL_2\ltimes U)$. Let $G=\GSp_{6}\times \GL_2$, $H=H_0\ltimes U$ with 
$$H_0=\{diag(h,h,h)\times h \mid h\in \GL_2\}$$
and $U$ be the unipotent radical of the standard parabolic subgroup $P=LU$ of $\GSp_6$ embedded into $G$ via the map $u\mapsto (u,I_2)$ where 
$$L=\{(h_1,h_2, \det(h_2)h_{1}^{\ast})|\;h_i\in \GL_2\}.$$ 
We define a generic character $\xi$ on $U(F)$ to be $\xi(u)=\psi(\lambda(u))$ where 
$$\lambda(u)=\tr(X),\;u=\begin{pmatrix}I_2&X&\ast  \\ 0&I_2&\ast \\ 0&0&I_2 \end{pmatrix}.$$ 
The model $(G,H)$ is the Whittaker induction of the trilinear $\GL_2$-model $(L\times \GL_2,H_0,\xi)$. As in the previous case, we can also define the quaternion version of this model.

Let $w_0=\begin{pmatrix}0&0&I_2\\ 0&I_2&0\\ I_2&0&0 \end{pmatrix}$ be the Weyl element that sends $U$ to its opposite. It is clear that the $w_0$-conjugation map stabilizes $L$ and fixes $H_0$. We define the map $a:\GL_1\rightarrow Z_L$ to be
$$a(t)=diag(tI_2,I_2,t^{-1}I_2)\times I_2.$$
This clearly satisfies the equation \eqref{the map a}. For the open Borel orbit, let
$$\eta_0=diag(I_2,\begin{pmatrix}0&1\\1&0 \end{pmatrix}, I_2)\times \begin{pmatrix}0&1\\1&0 \end{pmatrix}\begin{pmatrix}1&1\\0&1 \end{pmatrix}$$ 
be the representative of the open Borel orbit for the model $(L\times \GL_2,H_0)$ as in Section \ref{section triple product}, and $\eta=\eta_0 w_0$. The relation \eqref{eta0 relation} has already been verified in Section \ref{section triple product}. This finishes the first three steps in Section \ref{sec:6-steps}.

Let $\Theta$ be the weights of the 16-dimensional representation $\Spin_{7}\times {\rm std}_2$ of $\GSpin_{7}(\BC)\times \GL_2(\BC)$. We can write it as 
$$\Theta=\{\frac{\pm e_1\pm e_2\pm e_3}{2}+e_i' \mid 1\leq i\leq 2\}.$$

Let $\alpha_i=\varepsilon_i-\varepsilon_{i+1},\;1\leq i\leq 2$ and $\alpha_3=2\varepsilon_3$ be the simple roots of $\GSp_{6}$ and $\alpha'=\varepsilon_1'-\varepsilon_2'$ be the simple root of $\GL_2$. By the computation of the trilinear $\GL_2$-model in Section \ref{section triple product} and the discussion in Section \ref{sec non-red strategy 2} (in particular, Remark \ref{color in non-red}), we have
$$\beta_{\alpha_1}^{\vee}=\frac{e_1-e_2-e_3}{2}+e_1',\;\alpha_{1}^{\vee}-\beta_{\alpha_1}^{\vee}=\frac{e_1-e_2+e_3}{2}+e_2',$$
$$\beta_{\alpha_3}^{\vee}=\frac{-e_1+e_2+e_3}{2}+e_1',\;\alpha_{3}^{\vee}-\beta_{\alpha_3}^{\vee}=\frac{e_1-e_2+e_3}{2}+e_2',$$
$$\beta_{\alpha'}^{\vee}=\frac{-e_1+e_2+e_3}{2}+e_1',\;\alpha'^{\vee}-\beta_{\alpha'}^{\vee}=\frac{e_1-e_2-e_3}{2}+e_1'.$$
By a similar argument as in the Ginzburg--Rallis model case in Section 5, we can also verify \eqref{nonred root} for the root $\alpha_2$. The proof of the following proposition follows from a similar but easier argument as the model $(\GSp_{10},\GL_2\ltimes U)$ in the previous subsection. The only difference is to replace the identity in Lemma \ref{constant for GL(4)xGL(2) non-red} by the identity in Section  \ref{section triple product} for the trilinear $\GL_2$-model. We will skip it here.

\begin{prop}
$\Theta^+$ is consisting of the following 8 elements:
$$\frac{e_1+e_2\pm e_3}{2}+e_i', \frac{e_1-e_2+ e_3}{2}+e_i',\frac{\pm(e_1-e_2-e_3)}{2}+e_1',\;1\leq i\leq 2.$$
The set $\Theta^+$ satisfies \eqref{theta plus}. Moreover, we have
$$\sum_{w\in W}c_{WS}(w\theta)=\frac{1}{\Delta_{H_0/Z_{G,H}}(1)}=\frac{1}{\zeta(2)}=(1-q^{-2}).$$
\end{prop}

To sum up, we have proved that the local relative character is equal to 
$$\zeta(1)\zeta(2)\zeta(4)\zeta(6)\frac{L(1/2,\pi,\Spin_{7}\times std_2)}{L(1,\pi,\Ad)}$$
where $\pi$ is an unramified representation of $\GSp_{6}(F)\times \GL_2(F)$.

\subsection{The model $(\GSO_{12},\GL_2\ltimes U)$}\label{sec:GSO12}
In this subsection, we compute the local relative character for the model $(\GSO_{12},\GL_2\ltimes U)$. There are two models in this case (corresponding to the two Siegel parabolic subgroups) and they are differed by the outer automorphism of $\GSO_{12}$. Each of them corresponds to one of the Half-Spin $L$-function of $\GSpin_{12}(\BC)$. 
We will only compute the local relative character of one of the models, the other one can be computed just by applying the outer automorphism to the first one. Let $J_2'=\begin{pmatrix}0&-1\\1&0\end{pmatrix}$. Set $L_{4}=\begin{pmatrix}0&J_2'\\-J_2'&0 \end{pmatrix}$ and
$L_{4n}=\begin{pmatrix}0&0&J_2'\\0&L_{4n-4}&0\\-J_2'&0&0 \end{pmatrix}$. Define
$$\GSO_{4n}=\{g\in \GL_{4n} \mid g^tL_{4n}g =l(g)L_{4n}\}.$$
Let $G=\GSO_{12}$, $H=H_0\ltimes U$ with 
$$H_0=diag(h,h,h,h,h,h) \mid h\in \GL_2\}$$
and $U$ be the unipotent radical of the standard parabolic subgroup $P=LU$ of $G$ with ($h^\ast=J_2'{}^th^{-1} (J_2')^{-1}$)
$$L=\{diag(h_1,h_2, h_3,t h_{3}^{\ast},th_{2}^{\ast}, th_{1}^{\ast}) \mid h_i\in \GL_2,t\in \GL_1\}.$$ 
We define a generic character $\xi$ on $U(F)$ to be $\xi(u)=\psi(\lambda(u))$ where 
$$\lambda(u)=\tr(X)+\tr(Y)+\tr(Z),\;u=\begin{pmatrix}I_2&X&\ast &\ast &\ast&\ast \\ 0&I_2&Y&\ast&\ast &\ast\\ 0&0&I_2& Z & \ast&\ast \\ 0&0&0&I_2 &\ast&\ast \\ 0&0&0&0&I_2&\ast  \\ 0&0&0&0&0&I_2\end{pmatrix}.$$
It is easy to see that $H_0$ is the stabilizer of this character and $(G,H)$ is the Whittaker induction of the trilinear $\GL_2$-model $(L,H_0,\xi)$.

We can also define the quaternion version of this model. Let $D/F$ be a quaternion algebra, and let 
$$\GSO_{2n}(D)=\{g\in \GL_{2n}(D) \mid {}^t\bar{g} J_{2n'}g=l(g)J_{2n}'\}.$$
Let $G_D(F)=\GSO_6(D)$, $H_D=H_{0,D}\ltimes U_D$  with 
$$H_{0,D}(F)=\{diag(h,h,h,h,h,h) \mid h\in \GL_1(D) \}$$
and $U_D$ be the unipotent radical of the standard parabolic subgroup $P_D=L_DU_D$ of $G_D$ where ($h^\ast=\bar{h}^{-1}$)
$$L_D(F)=\{(h_1,h_2, h_3,t h_{3}^{\ast},t h_{2}^{\ast},t h_{1}^{\ast})|\;h_i\in \GL_1(D),t\in \GL_1(F)\}.$$ 
Here $x\rightarrow \bar{x}$ is the conjugation map on the quaternion algebra. Like the split case, we can define the character $\xi_D$ on $U_D(F)$ by replacing the trace map of $Mat_{2\times 2}$ by the trace map of $D$.

Let $w_0=\begin{pmatrix}0&0&0&0&0&I_2\\ 0&0&0&0&I_2&0\\ 0&0&0&I_2&0&0\\ 0&0&I_2&0&0&0\\ 0&I_2&0&0&0&0\\I_2&0&0&0&0&0 \end{pmatrix}$ be the Weyl element that sends $U$ to its opposite. It is clear that the $w_0$-conjugation map stabilizes $L$ and fixes $H_0$. We define the map $a:\GL_1\rightarrow Z_L$ to be
$$a(t)=diag(t^3I_2,t^2I_2,tI_2,I_2,t^{-1}I_2,t^{-2}I_2).$$
This clearly satisfies the second identity of the equation \eqref{the map a}. 
Although it does not satisfy the first equation of \eqref{the map a}, but the difference between $a(t)^{-1}$ and $w_{0}^{-1}a(t)w_0$ belongs to the center 
so all the arguments in Section \ref{sec non-red strategy 1} still work
(because all the characters are unramified). 
For the open Borel orbit, let
$$\eta_0=diag(I_2,\begin{pmatrix}0&1\\1&0 \end{pmatrix}, \begin{pmatrix}0&1\\1&0 \end{pmatrix}\begin{pmatrix}1&1\\0&1 \end{pmatrix}, \begin{pmatrix}0&-1\\-1&-1 \end{pmatrix},-\begin{pmatrix}0&1\\1&0 \end{pmatrix},I_2)$$ 
be the representative of the open Borel orbit for the model $(L,H_0)$ as in Section \ref{section triple product}, and $\eta=\eta_0 w_0$. The relation \eqref{eta0 relation} has already been verified in Section \ref{section triple product}. This finishes the first three steps in Section \ref{sec:6-steps}.

Now we compute the set of colors and also the set $\Theta^+$. Let $\Theta$ be the weights of the 32-dimensional Half-Spin representation $\HSpin_{12}$ of $\GSpin_{12}(\BC)$ given by
$$\Theta=\{\frac{\pm e_1\pm e_2\pm e_3\pm e_4\pm e_5\pm e_6}{6} \mid -\;\text{appears odd times}\}.$$

Let $\alpha_i=\varepsilon_i-\varepsilon_{i+1},\;1\leq i\leq 5$ and $\alpha_6=\varepsilon_5+\varepsilon_6$ be the simple roots of $\GSO_{12}$. By the computation of the trilinear $\GL_2$-model in Section \ref{section triple product} and the discussion in Section \ref{sec non-red strategy 2} (in particular, Remark \ref{color in non-red}), we have
$$\beta_{\alpha_1}^{\vee}=\frac{e_1-e_2-e_3+e_4+e_5-e_6}{2},$$
$$\alpha_{1}^{\vee}-\beta_{\alpha_1}^{\vee}=\frac{e_1-e_2+e_3-e_4-e_5+e_6}{2},$$ 
$$\beta_{\alpha_3}^{\vee}=\frac{-e_1+e_2+e_3-e_4+e_5-e_6}{2},$$
$$\alpha_{3}^{\vee}-\beta_{\alpha_3}^{\vee}=\frac{e_1-e_2+e_3-e_4-e_5+e_6}{2},$$
$$\beta_{\alpha_5}^{\vee}=\frac{-e_1+e_2+e_3-e_4+e_5-e_6}{2},$$
$$\alpha_{5}^{\vee}-\beta_{\alpha_5}^{\vee}=\frac{e_1-e_2-e_3+e_4+e_5-e_6}{2}.$$

By a similar argument as in the Ginzburg--Rallis model case in Section 5, we can also verify \eqref{nonred root} for the roots $\alpha_2$, $\alpha_4$ and $\alpha_6$. Next, we compute the set $\Theta^+$. 

\begin{prop}\label{prop:theta-GSO12}
$\Theta^+$ is consisting of the following 16 elements:
$$\frac{e_1+e_2+e_3+e_4+e_5+e_6-2e_l}{2},$$
$$\frac{-e_1-e_2-e_3-e_4-e_5-e_6+2e_i+2e_j+2e_k}{2}$$
with $1\leq l\leq 6$ and $(i,j,k)$ belongs to the set
$$\{(123),(124),(125),(126),(134),(135),(136),(145),(234),(235)\}.$$
\end{prop}

\begin{proof}
By the computation of the colors, we know that $\Theta^+$ is the smallest subset of $\Theta$ satisfying the following 5 conditions:
\begin{enumerate}
\item $\frac{e_1-e_2-e_3+e_4+e_5-e_6}{2},\frac{-e_1+e_2+e_3-e_4+e_5-e_6}{2},\frac{e_1-e_2+e_3-e_4-e_5+e_6}{2}\in \Theta^{+}$.
\item $\Theta^{+}-(\Theta^+\cap w_{\alpha_1}\Theta^+)=\{\frac{e_1-e_2-e_3+e_4+e_5-e_6}{2},\frac{e_1-e_2+e_3-e_4-e_5+e_6}{2}\}$.
\item $\Theta^{+}-(\Theta^+\cap w_{\alpha_3}\Theta^+)=\{\frac{e_1-e_2+e_3-e_4-e_5+e_6}{2},\frac{-e_1+e_2+e_3-e_4+e_5-e_6}{2}\}$.
\item $\Theta^{+}-(\Theta^+\cap w_{\alpha_5}\Theta^+)=\{\frac{e_1-e_2-e_3+e_4+e_5-e_6}{2},\frac{-e_1+e_2+e_3-e_4+e_5-e_6}{2}\}$.
\item $\Theta^{+}$ is stable under $w_{\alpha_2},\;w_{\alpha_4}$ and $w_{\alpha_6}$.
\end{enumerate}
It is clear that the set in the proposition satisfies these conditions. So we just need to show that the set is the unique subset of $\Theta$ satisfying these conditions. The argument is exactly the same as the case $(\GSp_6\times \GSp_4, (\GSp_4\times \GSp_2)^0)$ in Proposition \ref{theta proposition}. We will skip it here.
\end{proof}

It is clear that $\Theta^+$ satisfies \eqref{theta plus}. The last thing remains  to  prove Lemma \ref{lem constant nonreductive} for the current case. Let $\Theta_{1}^{+}$ (resp. $\Theta_{2}^{+}$) be the subset of $\Theta^+$ consisting of elements of the form 
$$\frac{e_1+e_2+e_3+e_4+e_5+e_6-2e_l}{2}$$
and let $\Theta_{2}^{+}$ be the subset of $\Theta^+$ consisting of elements of the form
$$\frac{e_1+e_2+e_3+e_4+e_5+e_6-2e_i-2e_j-2e_k}{2}.$$
Then $\Theta_{2}^{+}$ corresponds to the weights of the Ginzburg--Rallis model discussed in Section \ref{sec:GL6}. 
We also decompose the set of positive roots $\Phi^+$ as $\Phi_{1}^+\cup \Phi_{2}^+$ where 
$$\Phi_{2}^{+}=\{e_i-e_j|\;1\leq i<j\leq 6\}$$ 
is the set of the roots contained in $\GL_6$ and $\Phi_{1}^+$ contains the remaining positive roots. We also embed the Weyl group $S_6$ of $\GL_6$ into the Weyl group $W$.

\begin{lem}\label{lm:GSO12-constant-lem}
With the notation above, we have 
$$\sum_{w\in W}c_{WS}(w\theta)=\frac{1}{\Delta_{H_0/Z_{G,H}}(1)}=\frac{1}{\zeta(2)}=(1-q^{-2}).$$
\end{lem}

\begin{proof}
By the identity for the Ginzburg--Rallis model case proved in Lemma \ref{GR case constant lem}, we have 
$$\sum_{w\in W}c_{WS}(w\theta)=\sum_{w\in W}\frac{\prod_{\gamma^\vee\in \Theta^+}1-q^{-\frac{1}{2}}e^{\gamma^\vee}}
{\prod_{\alpha\in \Phi^+}1-e^{\alpha^\vee}}(w\theta)$$
$$=(1-q^{-2})\cdot \sum_{w\in W/S_6}\frac{\prod_{\gamma^\vee\in \Theta_{1}^+}1-q^{-\frac{1}{2}}e^{\gamma^\vee}}
{\prod_{\alpha\in \Phi_{1}^+}1-e^{\alpha^\vee}}(w\theta).$$
Hence it is enough to show that 
$$\sum_{w\in W/S_6}\frac{\prod_{\gamma^\vee\in \Theta_{1}^+}1-q^{-\frac{1}{2}}e^{\gamma^\vee}}
{\prod_{\alpha\in \Phi_{1}^+}1-e^{\alpha^\vee}}(w\theta)=1.$$
It is easy to see that the constant coefficient of the above summation is equal to 1, so it is enough to show that all the $q^{-i/2}$-coefficients are equal to 0 for $1\leq i\leq 6$. Like in the previous cases, we can replace the summation on $W/S_6$ by the summation on $W$. We also need to rewrite the function inside the summation ($\theta_i$ are arbitrary variables):
\begin{align*}
 &\frac{\prod_{\gamma^\vee\in \Theta_{1}^+}1-q^{-\frac{1}{2}}e^{\gamma^\vee}}
{\prod_{\alpha\in \Phi_{1}^+}1-e^{\alpha^\vee}}(w\theta)=\frac{\Pi_{i=1}^{6} (1-q^{-1/2}\cdot w(\frac{\sqrt{\theta_1\theta_2\theta_3\theta_4\theta_5\theta_6}}{\theta_i}))}{\Pi_{1\leq i<j\leq 6} (1-w(\theta_i\theta_j))}\\
=&\frac{\Pi_{1\leq i<j\leq 6} w(\theta_{i}^{-1}-\theta_{j}^{-1})\cdot \Pi_{i=1}^{6} (1-q^{-1/2}\cdot w(\frac{\sqrt{\theta_1\theta_2\theta_3\theta_4\theta_5\theta_6}}{\theta_i}))}{\Pi_{1\leq i<j\leq 6} w(\theta_{i}^{-1}-\theta_i-\theta_{j}^{-1}+\theta_j) }. 
\end{align*}
Then the denominator becomes $(W,\sgn)$-invariant, so it is enough to show that the $(W,\sgn)$-summation of the $q^{-i/2}$-coefficient of 
\begin{equation}\label{SO(12) constant 1}
\Pi_{1\leq i<j\leq 6} (\theta_{i}^{-1}-\theta_{j}^{-1})\cdot \Pi_{i=1}^{6} (1-q^{-1/2}\cdot \frac{\sqrt{\theta_1\theta_2\theta_3\theta_4\theta_5\theta_6}}{\theta_i})
\end{equation}
is equal to 0 for $1\leq i\leq 6$. The product $\Pi_{1\leq i<j\leq 6} (\theta_{i}^{-1}-\theta_{j}^{-1})$ consists of terms of the form
$$\Pi_{i=1}^{6} \theta_{i}^{a_i},\;\{a_1,a_2,a_3,a_4,a_5,a_6\}=\{-5,-4,-3,-2,-1,0\}.$$

Then any term $\Pi_{i=1}^{6} \theta_{i}^{b_i}$ appearing in the $q^{-1/2}$-coefficient of \eqref{SO(12) constant 1} must satisfy one of the following two conditions
\begin{itemize}
\item $b_i=b_j$ for some $i\neq j$.
\item $\{b_1,b_2,b_3,b_4,b_5,b_6\}=\{-11/2,-7/2,-5/2,-3/2,-1/2,1/2\}$.
\end{itemize}
In either case, we have $b_i=\pm b_j$ for some $i\neq j$. This implies that the $(W,\sgn)$-summation of the $q^{-1/2}$-coefficient is equal to 0.

For the $q^{-1}$-coefficient, any term $\Pi_{i=1}^{6} \theta_{i}^{b_i}$ appearing in it 
must satisfy one of the following two conditions
\begin{itemize}
\item $b_i=b_j$ for some $i\neq j$.
\item $\{b_1,b_2,b_3,b_4,b_5,b_6\}=\{-5,-4,-2,-1,0,1\}$.
\end{itemize}
In either case, we have $b_i=\pm b_j$ for some $i\neq j$. This implies that the $(W,\sgn)$-summation of the $q^{-1}$-coefficient is equal to 0.

For the $q^{-3/2}$-coefficient, any term $\Pi_{i=1}^{6} \theta_{i}^{b_i}$ appearing in it 
must satisfy one of the following two conditions
\begin{itemize}
\item $b_i=b_j$ for some $i\neq j$.
\item $\{b_1,b_2,b_3,b_4,b_5,b_6\}=\{-9/2,-7/2,-5/2,-1/2,1/2,3/2\}$.
\end{itemize}
In either case, we have $b_i=\pm b_j$ for some $i\neq j$. This implies that the $(W,\sgn)$-summation of the $q^{-3/2}$-coefficient is equal to 0.

For the $q^{-2}$-coefficient, any term $\Pi_{i=1}^{6} \theta_{i}^{b_i}$ appearing in it 
must satisfy one of the following two conditions
\begin{itemize}
\item $b_i=b_j$ for some $i\neq j$.
\item $\{b_1,b_2,b_3,b_4,b_5,b_6\}=\{-4,-3,-2,-1,1,2\}$.
\end{itemize}
In either case, we have $b_i=\pm b_j$ for some $i\neq j$. This implies that the $(W,\sgn)$-summation of the $q^{-2}$-coefficient is equal to 0.

For the $q^{-5/2}$-coefficient, any term $\Pi_{i=1}^{6} \theta_{i}^{b_i}$ appearing in it must satisfy one of the following two conditions
\begin{itemize}
\item $b_i=b_j$ for some $i\neq j$.
\item $\{b_1,b_2,b_3,b_4,b_5,b_6\}=\{-7/2,-5/2,-3/2,-1/2,1/2,5/2\}$.
\end{itemize}
In either case, we have $b_i=\pm b_j$ for some $i\neq j$. This implies that the $(W,\sgn)$-summation of the $q^{-3/2}$-coefficient is equal to 0.

Finally, the $q^{-3}$-coefficients consisting of terms of the form 
$$\Pi_{i=1}^{6} \theta_{i}^{b_i},\;\{b_1,b_2,b_3,b_4,b_5,b_6\}=\{-3,-2,-1,0,1,2\}.$$
The $(W,\sgn)$-summation of these terms is equal to 0. This finishes the proof of the lemma.
\end{proof}

To sum up, we have proved that the local relative character is equal to 
$$\zeta(1)\zeta(4)\zeta(6)^2\zeta(8)\zeta(10)\frac{L(1/2,\pi,\HSpin_{12})}{L(1,\pi,\Ad)}$$
where $\pi$ is an unramified representation of $\GSO_{12}(F)$.

\subsection{The model $(\GSO_8\times \GL_2,\GL_2\ltimes U)$}\label{sec:GSO8}
In this subsection, we compute the local relative character for the model $(\GSO_8\times \GL_2,\GL_2\ltimes U)$. Like the previous case, there are two models in this case and they are differed by the outer automorphism of $\GSO_{8}$. Each of them corresponds to one of the Half-Spin L-function of $\GSpin_{8}(\BC)$. We will only compute the local relative character of one of the models, the other one can be computed just by applying the outer automorphism to the first one.

Let $G=\GSO_{8}\times \GL_2$, $H=H_0\ltimes U$ with 
$$H_0=\{diag(h,h,h,h)\times h \mid h\in \GL_2\}$$
and $U$ be the unipotent radical of the standard parabolic subgroup $P=LU$ of $\GSO_8$ (we embed $U$ into $G$ via the map $u\mapsto u\times I_2$) where 
$$L=\{diag(h_1,h_2, th_{2}^{\ast}, th_{1}^{\ast})|\;h_i\in \GL_2,t\in \GL_1\}.$$ 
We define a generic character $\xi$ on $U(F)$ to be $\xi(u)=\psi(\lambda(u))$ where 
$$\lambda(u)=\tr(X)+\tr(Y),\;u=\begin{pmatrix}I_2&X&\ast &\ast  \\ 0&I_2&Y&\ast\\ 0&0&I_2& \ast \\ 0&0&0&I_2\end{pmatrix}.$$ 
The model $(G,H)$ is the Whittaker induction of the trilinear $\GL_2$-model $(L\times \GL_2,H_0,\xi)$. Similarly we can also define the quaternion algebra version of this model.

Let $w_0=\begin{pmatrix}0&0&0&I_2\\ 0&0&I_2&0\\ 0&I_2&0&0\\ I_2&0&0&0\end{pmatrix}\times I_2$ be the Weyl element that sends $U$ to its opposite. It is clear that the $w_0$-conjugation map stabilizes $L$ and fixes $H_0$. We define the map $a:\GL_1\rightarrow Z_L$ as
$$a(t)=diag(t^2I_2,tI_2,I_2,t^{-1}I_2)\times I_2.$$
This clearly satisfies the second identity of the equation \eqref{the map a}. 
Although it does not satisfy the first equation of \eqref{the map a}, but the difference between $a(t)^{-1}$ and $w_{0}^{-1}a(t)w_0$ belongs   to the center so all the arguments in Section \ref{sec non-red strategy 1} still work. For the open Borel orbit, let
$$\eta_0=diag(I_2,\begin{pmatrix}0&1\\1&0 \end{pmatrix}, -\begin{pmatrix}0&1\\1&0 \end{pmatrix},I_2)\times \begin{pmatrix}0&1\\1&0 \end{pmatrix}\begin{pmatrix}1&1\\0&1 \end{pmatrix}$$ 
be the representative of the open Borel orbit for the model $(L,H_0)$ as in Section \ref{section triple product}, and $\eta=\eta_0 w_0$. The relation \eqref{eta0 relation} has already been verified in Section \ref{section triple product}. This finishes the first three steps in Section \ref{sec:6-steps}.

Now we compute the set of colors and also the set $\Theta^+$. Let $\Theta$ be the weights of the 16-dimensional Half-Spin representation $\HSpin_{12}$ of $\GSpin_{12}(\BC)$ given by
$$\Theta=\{\frac{\pm e_1\pm e_2\pm e_3\pm e_4}{2}+e_i' \mid -\;\text{appears even times},i\in \{1,2\}\}.$$

Let $\alpha_i=\varepsilon_i-\varepsilon_{i+1},\;1\leq i\leq 3$ and $\alpha_4=\varepsilon_3+\varepsilon_4$ be the simple roots of $\GSO_{8}$ and $\alpha'=\varepsilon_1'-\varepsilon_2'$ be the simple root of $\GL_2$. By the computation of the trilinear $\GL_2$-model in Section \ref{section triple product} and the discussion in Section \ref{sec non-red strategy 2} (in particular, Remark \ref{color in non-red}), we have
$$\beta_{\alpha_1}^{\vee}=\frac{e_1-e_2-e_3+e_4}{2}+e_1',\;\alpha_{1}^{\vee}-\beta_{\alpha_1}^{\vee}=\frac{e_1-e_2+e_3-e_4}{2}+e_2',$$ $$\beta_{\alpha_3}^{\vee}=\frac{-e_1+e_2+e_3-e_4}{2}+e_1',\;
\alpha_{3}^{\vee}-\beta_{\alpha_3}^{\vee}=\frac{e_1-e_2+e_3-e_4}{2}+e_2',$$
$$\beta_{\alpha'}^{\vee}=\frac{-e_1+e_2+e_3-e_4}{2}+e_1',\;\alpha'^{\vee}-\beta_{\alpha'}^{\vee}=\frac{e_1-e_2-e_3+e_4}{2}+e_1'.$$
By a similar argument as in the Ginzburg--Rallis model case in Section 5, we can also verify \eqref{nonred root} for the roots $\alpha_2$ and $\alpha_4$. The next proposition computes the set $\Theta^+$ and proves Lemma \ref{lem constant nonreductive} for the current case.

\begin{prop}
$\Theta^+$ is consisting of the following 8 elements:
$$\frac{e_1+e_2\pm(e_3+e_4)}{2}+e_i',\frac{e_1-e_2+e_3-e_4}{2}+e_i',$$
$$\frac{\pm(e_1-e_2-e_3+e_4)}{2}+e_1',\;1\leq i\leq 2.$$
Moreover, we have
$$\sum_{w\in W}c_{WS}(w\theta)=\frac{1}{\Delta_{H_0/Z_{G,H}}(1)}=\frac{1}{\zeta(2)}=(1-q^{-2}).$$
\end{prop}

\begin{proof}
The proof follows from a similar but  easier argument as the $(\GSO_{12},\GL_2\ltimes U)$ model case in the previous subsection. The only difference is that we need to use  Lemma \ref{constant for GL(4)xGL(2) non-red}  instead of Lemma \ref{GR case constant lem}. We will skip the details here.
\end{proof}

To sum up, we have proved that the local relative character is equal to 
$$\zeta(1)^2\zeta(2)\zeta(4)^2\zeta(6)\frac{L(1/2,\pi,\HSpin_{8}\times std_2)}{L(1,\pi,Ad)}$$
where $\pi$ is an unramified representation of $\GSO_{8}(F)\times \GL_2(F)$.

\section{Local multiplicity}\label{sec multiplicity}
In this section we will study the multiplicity for the models in Table \ref{fig:1}. Let $F$ be a local field of characteristic 0, $(G,H)$ be one of the models in Table \ref{fig:1}, and $\xi$ be the character of $H(F)$ defined in the previous sections (note that $\xi$ is trivial in the reductive case). Let $\pi$ be an irreducible admissible representation of $G(F)$ whose central character is trivial on $Z_{G,H}(F)$. Recall that the multiplicity is defined by
$$m(\pi)=\dim \Hom_{H(F)}(\pi,\xi).$$

Similarly, if $F\neq \BC$, let $D/F$ be the unique quaternion algebra (or $D\in H^1(F,H/Z_{G,H})$  if we are in the case of Model 2 of Table \ref{fig:1}), and let $(G_D,H_D,\xi_D)$ be the pure inner form of the model $(G,H,\xi)$ defined in the previous sections. Let $\pi_D$ be an irreducible representation of $G_D(F)$ whose central character is trivial on $Z_{G_D,H_D}(F)$. We can also define the multiplicity $m(\pi_D)=\dim(\Hom_{H_D(F)}(\pi_D,\xi_D)).$

In this section, we will prove a geometric multiplicity formula of $m(\pi)$ and $m(\pi_D)$ in terms of the Harish-Chandra character. 
Then by using the geometric multiplicity formula, together with the character identity in the local Langlands correspondence, we will show that for all the tempered $L$-packets, the summation of the multiplicities is equal to 1 and the unique distinguished element in the packet corresponds to a character of the component group.  The proof of all the results in this section is very similar to the Gan--Gross--Prasad model case (\cite{W10}, \cite{W12}, \cite{B15}) and the Ginzburg--Rallis model case (\cite{Wan15}, \cite{Wan16}, \cite{Wan17}, \cite{WZ}) 
since similar to the Gan--Gross--Prasad model and the Ginzburg--Rallis model, all the models in Table \ref{fig:1} are strongly tempered without Type $N$ root and has a unique open Borel orbit.

In Section \ref{sec:character-identity}, we will recall the local Langlands conjecture. In Section \ref{sec:reductive} we will study the reductive models in Table \ref{fig:1} and in Section \ref{sec:non-reductive} we will study the non-reductive models.

\subsection{The local Langlands conjecture}\label{sec:character-identity}
In this subsection we recall the local Langlands conjecture in Conjecture E of \cite{K}. Let $G$ be a quasi-split reductive group defined over $F$ and let $\{G_\alpha|\;\alpha\in H^1(F,G)\}$ be the set of pure inner forms of $G$. Let $\Pi_{irr,temp}(G_{\alpha})$ be the set of irreducible tempered representations of $G_{\alpha}(F)$. The local Langlands conjecture states that $\cup_{\alpha\in H^1(F,G)}\Pi_{irr,temp}(G_\alpha)$ is a disjoint union of finite sets (i.e. the local tempered Vogan $L$-packets)
$$\cup_{\phi} \Pi_{\phi}$$
where $\phi$ runs over all the tempered $L$-parameters of $G$ and $\Pi_{\phi}=\cup_{\alpha\in H^1(F,G)} \Pi_{\phi}(G_\alpha)$ consists of a finite number of tempered representations with $\Pi_{\phi}(G_\alpha)\subset \Pi_{irr,temp}(G_{\alpha})$ such that the following conditions hold.

\begin{itemize}
\item There is a unique generic element in $\Pi_\phi(G)$ with respect to any Whittaker datum of $G$.
\item For given Whittaker datum, there is a bijection between $\hat{S_\phi}$, the set of irreducible representations of the component group $S_\phi=Z_\phi/Z_{\phi}^{\circ}$ ($Z_\phi$ is the centralizer of $Im(\phi)$ in $\hat{G}$) of the Langlands parameter $\phi$, and $\Pi_\phi$ (denoted by $\pi\leftrightarrow \chi_\pi$) such that 
\begin{itemize}
\item the trivial character of $S_\phi$ corresponds to the unique generic element of $\Pi_\phi(G)$ with respect to the given Whittaker datum.
\item for $\alpha\in H^1(F,G)$, the distribution character $\theta_{\Pi_{\phi}(G_\alpha)}=\sum_{\pi\in \Pi_\phi(G_{\alpha})}\dim(\chi_\pi)\theta_\pi$ is stable. Moreover,  $\iota(G_{\alpha})\theta_{\Pi_\phi(G_\alpha)}$ is the transfer of $\theta_{\Pi_\phi(G)}$ where $\iota(G_\alpha)$ is the Kottwitz sign.
\item endoscopic identity.
\end{itemize}
\end{itemize}

We will not discuss the endoscopic identity of the local Langlands conjecture here since we don't need to use it in this paper, we refer the reader to \cite{K} for more details. In order to prove the multiplicity one of the L-packet for the models in Table \ref{fig:1}, we need to assume that the local Langlands conjecture holds for the groups associated to the models. Note that for the group $G$ in Model 3 and Model 6-10 of Table \ref{fig:1}, the component group $S_\phi$ is not necessarily abelian.

\subsection{The reductive case}\label{sec:reductive}
In this subsection we assume that $H$ is reductive. The model $(\GL_4\times \GL_2,\GL_2\times \GL_2)$ has already been considered in the previous paper \cite{PWZ19}, so we will focus on the models $(G,H)=(\GSp_6\times \GSp_4,(\GSp_4\times \GSp_2)^0)$ and $(G,H)=(\GU_4\times \GU_2,(\GU_2\times \GU_2)^0)$. 

Let $\pi$ be an irreducible representation of $G(F)$ with trivial central character and $\theta_\pi$ be its Harish-Chandra character. For a semisimple element $x\in G(F)$, we let $c_\pi(x)$ be the average of the regular germs of $\theta_\pi$ at $x$. We refer the reader to Section 4.5 of \cite{B15} for the definition of regular germs. We want to emphasize that $c_\pi(x)$ is zero if the centralizer $G_x$ is not quasi-split. We also let $\CT_{ell}(G)$ be a set of representatives of maximal elliptic tori of $G(F)$.

\subsubsection{The model $(\GSp_6\times \GSp_4,(\GSp_4\times \GSp_2)^0)$}
We first consider the case $(G,H)=(\GSp_6\times \GSp_4,(\GSp_4\times \GSp_2)^0)$. For $T\in \CT_{ell}(\GSp_2)$, let
$$T^{n,0}=\{(t_1,\cdots,t_n)\in T^n|\;\det(t_i)=\det(t_j)\;\text{for all}\;1\leq i,j\leq n\}.$$
We use $\iota_n$ to denote the diagonal embedding from $T$ to $T^{n,0}$. We can view $T^{n,0}$ as a maximal elliptic torus of $\GSp_{2n}$. Moreover, up to $\GSp_{2n}$-conjugation, there are $2^{n-1}$-many different embeddings from $T^{n,0}$ to $\GSp_{2n}$.

When $n=2$, there are two embeddings $\nu_2,\nu_2'$ from $T^{2,0}$ to $\GSp_{4}$ and the centralizer of the image of $ \nu_2\circ\iota_2$ (resp. $\nu_2'\circ\iota_2$) in $\GSp_4$ is the quasi-split (resp. non quasi-split) unitary similitude group of 3 variables. Meanwhile, there are four embeddings from $T^{3,0}$ to $(\GSp_{4}\times \GSp_2)^0$ and there are two of them whose projection to $\GSp_4$ coincide with $\nu_2$. Compose with the embedding from $(\GSp_{4}\times \GSp_2)^0$ to $\GSp_6$, we get  two embeddings $\nu_{31},\nu_{32}$ from $T^{3,0}$ to $\GSp_6$. The centralizers of the image of $ \nu_{3i}\circ \iota_3$ ($i=1,2$) in $\GSp_6$ are the two  unitary similitude groups of 3 variables (both of them are quasi-split). We use $\nu_{T,i}=(\nu_{3i} \circ \iota_3)\times ( \nu_2\circ \iota_2)$ to denote the two embeddings from $T$ to $G$ (both factor through $H$). It is easy to see that these two embeddings are conjugated to each other in $H$ and we will use $\nu_T$ to denote one of it.

Meanwhile, let $\iota_{1,2}$ be the embedding from $T^{2,0}$ to $T^{3,0}$ given by 
$$(t_1,t_2)\mapsto (t_1,t_2,t_2).$$ 
Among the four embeddings from $T^{3,0}$ to $\GSp_{6}$, there are two of them (denoted by $\nu_3,\nu_3'$) such that the centralizers in $\GSp_{6}$ of the image of $\nu_3\circ \iota_{1,2}$ and $\nu_3'\circ \iota_{1,2}$ are quasi-split (the centralizer is the quasi-split unitary similitude group of 2 variables times an abelian group). Up to conjugation we may assume that $\nu_3,\nu_3'$ factor  through $(\GSp_{4}\times \GSp_2)^0$ and the projection to $\GSp_4$ of $\nu_3\circ \iota_{1,2}$ (resp. $\nu_3'\circ \iota_{1,2}$) is equal to $\nu_2$ (resp. $\nu_2'$). We use 
$$\nu_{T^{2,0},1}=(\nu_{3}\circ \iota_{1,2})\times \nu_2,\; \nu_{T^{2,0},2}=(\nu_{3}'\circ \iota_{1,2})\times \nu_2'$$ 
to denote the two embeddings from $T^{2,0}$ to $G$. Both of them factor through $H$.

Finally, for $T_1,T_2\in \CT_{ell}(\GSp_2)$ with $T_1\neq T_2$ (this will not happen in the archimedean case), let
$$(T_1\times T_2)^0=\{(t_1,t_2)\in T_1\times T_2|\;\det(t_1)=\det(t_2)\}.$$
Similarly, we can define $(T_1\times T_2\times T_2)^0$. Up to conjugation, there is only one embedding from $(T_1\times T_2)^0$ to $\GSp_4$ and there are two embeddings from $(T_1\times T_2\times T_2)^0$ to $\GSp_6$. The two embeddings induce  two embeddings from $(T_1\times T_2)^0$ to $\GSp_6$ (we first map $T_2$ diagonally into $(T_2\times T_2)^{0}$). We let $\nu$ be the embedding such that the centralizer of its image is quasi-split (the centralizer of the other embedding is not quasi-split). Up to conjugation we may assume that $\nu$ factors through $(\GSp_4\times \GSp_2)^0$ and its projection to $\GSp_4$ is equal to the embedding from $(T_1\times T_2)^0$ to $\GSp_4$. This gives us an embedding $\nu_{T_1,T_2}$ from $(T_1\times T_2)^{0}$ to $G$ that factors through $H$. 

Define the geometric multiplicity to be ($D^H(\cdot)$ is the Weyl determinant)
\begin{eqnarray*}
m_{geom}(\pi)&=&c_{\pi}(1)+\sum_{T\in \CT_{ell}(H)}|W(H,T)|^{-1}\int_{T(F)/Z_{G,H}(F)}^{\ast} D^H(t)\theta_{\pi}(t)\ud t\\
&&+\frac{1}{2} \sum_{T\in \CT_{ell}(\GSp_2)}\big(\int_{T(F)/Z_{\GL_2}(F)}^{\ast} D^H(\nu_{T}(t))c_{\pi}(\nu_{T}(t))\ud t\\
&&+\sum_{i\in \{1,2\}}\int_{T^{2,0}(F)/Z_{\GL_2}(F)}^{\ast} D^H(\nu_{T^{2,0},i}(t))c_{\pi}(\nu_{T^{2,0},i}(t))\ud t \big)\\
&&+\frac{1}{4} \sum_{T_1,T_2\in \CT_{ell}(\GSp_2),T_1\neq T_2}\int_{(T_1\times T_2)^0(F)/Z_{\GL_2}(F)^{diag}}^{\ast}\\ &&D^H(\nu_{T_1,T_2}(t)) c_{\pi}(\nu_{T_1,T_2}(t))\ud t.
\end{eqnarray*}
Here $1$ always stands for the identity element of $G(F)$, $W(H,T)$ is the Weyl group, all the Haar measures are chosen so that the total volume is equal to 1 (note that all the integral domains are compact), and the factors $\frac{1}{2},\frac{1}{4}$ come from the cardinality of the Weyl groups $W(\GSp_2,T),W(\GSp_2,T_i)$. Note that if $F=\BC$, then $\CT_{ell}(H)$ and $\CT_{ell}(\GL_2)$ are empty. Hence we have $m_{geom}(\pi)=c_{\pi}(1)$. When $F=\BR$, $\CT_{ell}(\GSp_2)$ only contains one element and the term associated to 
$$T_1,T_2\in \CT_{ell}(\GSp_2),T_1\neq T_2$$ 
will not appear. We leave it as an excise for the reader to check that our definition of $m_{geom}(\pi)$ matches the definition in \cite{Wan} for general spherical varieties. We refer the reader to \cite{Wan} for a detailed discussion of the geometric multiplicity for general spherical varieties.

\begin{rmk}\label{local convergence}
Like the unitary Gan--Gross--Prasad model case (Proposition 11.2.1 of \cite{B15}), the integrals defining the geometric multiplicity are not necessarily absolutely convergent and they need to be regularized (this is why we write the integral as $\int^\ast$). The regularization is the same as the unitary Gan--Gross--Prasad model case. To be specific, one replace the Weyl determinant $D^{H}$ in the integrand by $(D^G)^{1/2}\cdot (\frac{(D^H)^2}{D^G})^{s-1/2}$. By a very similar argument as in Proposition 11.2.1 of \cite{B15}, we know that the integral is absolutely convergent when $s>0$ and has a limit as $s\rightarrow 0^+$. Then we can define the regularized integral to be this limit. This remark also applies to the model $(\GU_4\times \GU_2,(\GU_2\times \GU)^0)$ in the next subsection.
\end{rmk}

Similarly, if $F\neq \BC$, for the quaternion version of the model, we can also define the embeddings $\nu_{T_D,i},\nu_{T_{D}^{2,0},i},\nu_{T_{1,D},T_{2,D}}$ for $T_D,T_{1,D},T_{2,D}\in \CT_{ell}(\GSp_1(D))=\CT_{ell}(\GSp_2)$ with $T_{1,D}\neq T_{2,D}$. We can define the geometric multiplicity $m_{geom}(\pi_D)$ to be
\begin{eqnarray*}
&&\sum_{T_D\in \CT_{ell}(H_D)}|W(H_D,T_D)|^{-1}\int_{T_D(F)/Z_{G_D,H_D}(F)}^{\ast} D^{H_D}(t)\theta_{\pi_D}(t)\ud t\\
&&+\frac{1}{2}\sum_{T_D\in \CT_{ell}(\GSp_1(D))}\big(\int_{T_D(F)/Z_{\GL_1(D)}(F)}^{\ast} D^{H_D}(\nu_{T_D}(t))c_{\pi_D}(\nu_{T_D}(t))\ud t\\
&&+\sum_{i\in \{1,2\}}\int_{T_{D}^{2,0}(F)/Z_{\GL_1(D)}(F)}^{\ast} D^{H_D}(\nu_{T_{D}^{2,0},i}(t))c_{\pi_D}(\nu_{T_{D}^{2,0},i}(t))\ud t \big)\\
&&+\frac{1}{4}\sum_{T_{1,D},T_{2,D}\in \CT_{ell}(\GSp_1(D)),T_{1,D}\neq T_{2,D}}\int_{(T_{1,D}\times T_{2,D})^0(F)/Z_{\GL_1(D)}(F)^{diag}}^{\ast} \\
&&D^{H_D}(\nu_{T_{1,D},T_{2,D}}(t)) c_{\pi_D}(\nu_{T_{1,D},T_{2,D}}(t))\ud t.
\end{eqnarray*}
The only difference between $m_{geom}(\pi)$ and $m_{geom}(\pi_D)$ is that $m_{geom}(\pi)$ contains the germ at 1 (since $G(F)$ is quasi-split) while $m_{geom}(\pi_D)$ dose not. The following theorem gives a geometric multiplicity formula for the model.

\begin{thm}
For all tempered representations $\pi$ of $G(F)$ (resp. $\pi_D$ of $G_D(F)$) whose central character is trivial on $Z_{G,H}(F)$ (resp. $Z_{G_D,H_D}(F)$), we have
$$m(\pi)=m_{geom}(\pi),\;m(\pi_D)=m_{geom}(\pi_D).$$
\end{thm}

\begin{proof}
This follows from a similar but easier argument as in the Gan--Gross--Prasad model case (\cite{W10}, \cite{W12},  \cite{B15}) and the Ginzburg--Rallis model case (\cite{Wan15}, \cite{Wan16}, \cite{Wan17}). The argument is easier for this model because it is reductive and hence there is no need to regularize the integral over $H$. The only difference is that the proofs in the above papers used the Gelfand pair condition (i.e. $m(\pi)\leq 1$ for all irreducible representation $\pi$ of $G(F)$) which is not known for this model. But this can be solved by the same argument as the unitary Ginzburg--Rallis model case in our previous paper (Section 6 and Appendix A of \cite{WZ}). We will skip the proof.
\end{proof}

If $F=\BC$, then any tempered representation of $G(F)$ is generic and we have $m_{geom}(\pi)=c_{\pi}(1)=1$ by the result for Whittaker model in \cite{Mat}. Hence the above theorem implies that $m(\pi)=1$ for all tempered representations $\pi$ of $G(F)$ with trivial central character (note that the $L$-packet only contains one element in the complex case).

If $F\neq \BC$, let $\Pi_{\phi}=\Pi_{\phi}(G)\cup \Pi_{\phi}(G_D)$ be a tempered local $L$-packet of $G$ whose central character is trivial on $Z_{G,H}(F)$. Assume that the local Langlands conjecture holds for $G(F)$. Let 
$$\theta_{\Pi_{\phi}(G)}=\sum_{\pi\in \Pi_{\phi}(G)}\dim(\chi_\pi)\theta_{\pi},\;\theta_{\Pi_{\phi}(G_D)}=\sum_{\pi_D\in \Pi_{\phi}(G_D)}\dim(\chi_{\pi_D})\theta_{\pi_D}$$
be the corresponding stable characters (note that $G(F)$ has a unique Whittaker datum). The Kottwitz sign between $G$ and $G_D$ is $-1$. Hence we have 
$$\theta_{\Pi_{\phi}(G)}(g)=-\theta_{\Pi_{\phi}(G_D)}(g_D),\;\forall g\in G_{reg}(F),g_D\in G_D(F),g\leftrightarrow g_D.$$
Combining with Proposition 4.5.1 of \cite{B15}, we have
$$c_{\theta_{\Pi_{\phi}(G)}}(\nu_{T}(t))=-c_{\theta_{\Pi_{\phi}(G_D)}}(\nu_{T_D}(t_D)),\;\forall t\in T(F)\leftrightarrow t_D\in T_D(F);$$
$$c_{\theta_{\Pi_{\phi}(G)}}(\nu_{T^{2,0},i}(t))=-c_{\theta_{\Pi_{\phi}(G_D)}}(\nu_{T_{D}^{2,0},i}(t_D)),\;\forall t\in T^{2,0}(F)\leftrightarrow t_D\in T_{D}^{2,0}(F);$$
$$c_{\theta_{\Pi_{\phi}(G)}}(\nu_{T_1,T_2}(t))=-c_{\theta_{\Pi_{\phi}(G_D)}}(\nu_{T_{1,D},T_{2,D}}(t_D)),$$
$$\forall t\in (T_1\times T_2)^0(F)\leftrightarrow t_D\in (T_{1,D}\times T_{2,D})^0(F).$$
Here $g\leftrightarrow g_D$ (resp. $t\leftrightarrow t_D$) means that they have the same characteristic polynomial. Together with the multiplicity formula, we have
$$\sum_{\pi\in \Pi_{\phi}(G)} \dim(\chi_\pi)m(\pi)+\sum_{\pi_D\in \Pi_{\phi}(G_D)}\dim(\chi_{\pi_D})m(\pi_D)$$
$$= m_{geom}(\theta_{\Pi_\phi(G)})+m_{geom}(\theta_{\Pi_\phi(G_D)})=c_{\theta_{\Pi_{\phi}(G)}}(1)=1$$
where the last equality follows from the results for Whittaker model in \cite{Rod81}, \cite{Mat} and the fact that there is a unique generic element in the $L$-packet. In particular, we have proved that the summation of the multiplicities is equal to 1 over every tempered local Vogan $L$-packet and the unique distinguished element corresponds to a character of the component group.

\subsubsection{The model $(\GU_{4}\times \GU_{2},(\GU_2\times \GU_2)^0)$}
Let $(G,H)=(\GU_{2,2}\times \GU_{1,1},(\GU_{1,1}\times \GU_{1,1})^0)$, and 
$$(G_1,H_1)=(\GU_{2,2}\times \GU_{2,0},(\GU_{2,0}\times \GU_{0,2})^0),$$
$$(G_2,H_2)=(\GU_{3,1}\times \GU_{1,1},(\GU_{1,1}\times \GU_{2,0})^0),$$
$$(G_3,H_3)=(\GU_{3,1}\times \GU_{2,0},(\GU_{2,0}\times \GU_{1,1})^0),$$
$$(G_4,H_4)=(\GU_{4,0}\times \GU_{2,0},(\GU_{2,0}\times \GU_{2,0})^0)$$
be the pure inner forms (the pair $(G_4,H_4)$ only appears in the archimedean case).

Let $T_0$ be the unique element in $\CT_{ell}(\GU_{1,1})=\CT_{ell}(\GU_{2,0})$ that is isomorphic to 
$$E^{2,0}:=\{(a,b)\in E^\times \times E^{\times}|\;a\bar{a}=b\bar{b}\}\subset E^\times \times E^{\times}.$$ 
For $T\in \CT_{ell}(\GU_{1,1})=\CT_{ell}(\GU_{2,0})$ with $T\neq T_0$, let $$(T\times T)^{0}=\{(t_1,t_2)\in T\times T|\;\lambda(t_1)=\lambda(t_2)\}.$$ Up to conjugation, there is a unique embedding from $(T\times T)^{0}$ to $(\GU_{1,1}\times \GU_{1,1})^0$ (resp. $(\GU_{2,0}\times \GU_{0,2})^0$). Combining with the diagonal embedding from $T$ to $(T\times T)^{0}$, we get an embedding (denoted by $\nu_T$) from $T$ to $G$ (resp. $G_1(F)$) that factors through $H$ (resp. $H_1$), and we will denote this embedding by $\nu_T$ (resp. $\nu_{1,T}$).

For $T_0$, in the p-adic case up to conjugation there are two embeddings from $(T_0\times T_0)^{0}$ to $(\GU_{1,1}\times \GU_{1,1})^0$ (resp. $(\GU_{2,0}\times \GU_{0,2})^0$). Combining with the diagonal embedding from $T_0$ to $(T_0\times T_0)^{0}$, we get two embeddings  from $T_0$ to $G$ (resp. $G_1$). The centralizer of the image of one of the embedding is quasi-split (it is isomorphic to $(\GU_{1,1}\times \GU_{1,1})^0$ times a torus), we will denote this embedding by $\nu_{T_0}$ (resp. $\nu_{1,T_0}$), while the centralizer of the image of  the other embedding is not quasi-split. In the archimedean case, we can define the embedding $\nu_{T_0}$ in the same way as in the p-adic case. On the other hand, up to conjugation there is only one embedding from $(T_0\times T_0)^{0}$ to $(\GU_{2,0}\times \GU_{0,2})^0$ and this defines the embedding $\nu_{1,T_0}$. Note that in this case the centralizer of the image of $\nu_{1,T_0}$ is still quasi-split.

\begin{rmk}
For $T\in \CT_{ell}(\GU_{1,1})$, we can also define the embeddings to $G_2$ and $G_3$ (also $G_4$ in the archimedean case), but the centralizer of the images will not be quasi-split.
\end{rmk}

Meanwhile, consider the following two subgroups of  $(T_0\times T_0)^0$ (we identify $T_0$ with $E^{2,0}=\{(a,b)\in E^\times \times E^{\times}|\;a\bar{a}=b\bar{b}\}$):
$$T_0'=\{(1,1)\times (1,a)\in (T_0\times T_0)^{0}|\;a\in E^1\},$$
$$T_0''=\{(1,a)\times (1,b)\in (T_0\times T_0)^{0}|\;a,b\in E^1\}.$$
The two embeddings from $(T_0\times T_0)^{0}$ to $(\GU_{1,1}\times \GU_{1,1})^0$ (resp. $(\GU_{1,1}\times \GU_{2,0})^0$) induce two embeddings from $T_0'$ to $G$ (resp. $G_2$) that are conjugated to each other. Let $\nu_{T_0'}$ (resp. $\nu_{2,T_0'}$) be one of the embedding. Note that the projection of these embeddings to the $\GU_{1,1}$-factor is the trivial map. The centralizers of the image of these embeddings are quasi-split (they are isomorphic to $\GU_3\times \GU_{1,1}\times U_1$).

\begin{rmk}
We can also define embeddings from $T_0'$ to $G_1$ and $G_3$ (also $G_4$ in the archimedean case), but the centralizer of the images will not be quasi-split.
\end{rmk}

On the other hand, the two embeddings from $(T_0\times T_0)^{0}$ to $(\GU_{1,1}\times \GU_{1,1})^0$ induce two embeddings from $T_0''$ to $G$. The centralizer of the image of one of the embedding is quasi-split (isomorphic to $\GU_{1,1}$ times some torus, we will denote this embedding by $\nu_{T_0''}$) and the centralizer of the image of the other embedding is not quasi-split. Similarly, we can also define the embeddings $\nu_{i,T_0''}$ from $T_0''$ to $G_i$ for $1\leq i\leq 3$. 

\begin{rmk}
We can also define the embedding from $T_0''$ to $G_4$ in the archimedean case but the centralizer of the images will not be quasi-split.
\end{rmk}

Now we are ready to define the geometric multiplicity. Let $\pi$ (resp. $\pi_i$) be an irreducible representation of $G(F)$ (resp. $G_i(F)$) with trivial central character. For $T\in \CT_{ell}(\GU_{1,1})=\CT_{ell}(\GU_{2,0})$, we use $T^{\ast}(F)$ to denote $T(F)/Z_{\GU_{1,1}}(F)=T(F)/Z_{\GU_{2,0}}(F)$. Define
\begin{eqnarray*}
m_{geom}(\pi)&=&c_\pi(1)+\sum_{T\in \CT_{ell}(H)} |W(H,T)|^{-1} \int_{T(F)/Z_{G,H}(F)}^{\ast} D^H(t)\theta_\pi(t) \ud t\\
&&+\frac{1}{2}\sum_{T\in \CT_{ell}(\GU_{1,1})}\int_{T^\ast(F)}^{\ast}  D^H(\nu_T(t))c_\pi(\nu_T(t)) \ud t\\
&&+\int_{T_{0}'(F)}^{\ast} D^H(\nu_{T_0'}(t))c_\pi(\nu_{T_0'}(t)) \ud t\\
&&+\int_{T_{0}''(F)}^{\ast} D^H(\nu_{T_0''}(t))c_\pi(\nu_{T_0''}(t)) \ud t,
\end{eqnarray*}
\begin{eqnarray*}
m_{geom}(\pi_1)&=&\sum_{T\in \CT_{ell}(H_1)} |W(H_1,T)|^{-1} \int_{T(F)/Z_{G_1,H_1}(F)}^{\ast} D^{H_1}(t)\theta_{\pi_1}(t) \ud t\\
&&+\frac{1}{2}\sum_{T\in \CT_{ell}(\GU_{2,0})} \int_{T^\ast(F)}^{\ast} D^{H_1}(\nu_{1,T}(t))c_{\pi_1}(\nu_{1,T}(t)) \ud t\\
&&+\int_{T_{0}''(F)}^{\ast}D^{H_1}(\nu_{1,T_0''}(t))c_{\pi_1}(\nu_{1,T_0''}(t)) \ud t,
\end{eqnarray*}
\begin{eqnarray*}
m_{geom}(\pi_2)&=&\sum_{T\in \CT_{ell}(H_2)} |W(H_2,T)|^{-1} \int_{T(F)/Z_{G_2,H_2}(F)}^{\ast} D^{H_2}(t)\theta_{\pi_2}(t) \ud t\\
&&+\int_{T_{0}'(F)}^{\ast} D^{H_2}(\nu_{2,T_0'}(t))c_{\pi_2}(\nu_{2,T_0'}(t)) \ud t\\
&&+\int_{T_{0}''(F)}^{\ast} D^{H_2}(\nu_{2,T_0''}(t))c_{\pi_2}(\nu_{2,T_0''}(t)) \ud t,
\end{eqnarray*}
$$m_{geom}(\pi_3)=\sum_{T\in \CT_{ell}(H_3)} |W(H_3,T)|^{-1} \int_{T(F)/Z_{G_3,H_3}(F)}^{\ast} D^{H_3}(t)\theta_{\pi_3}(t) \ud t$$
$$+\int_{T_{0}''(F)}^{\ast} D^{H_3}(\nu_{3,T_0''}(t))c_{\pi_3}(\nu_{3,T_0''}(t)) \ud t.$$
If we are in the archimedean case, we also define
$$m_{geom}(\pi_4)=\sum_{T\in \CT_{ell}(H_4)} |W(H_4,T)|^{-1} \int_{T(F)/Z_{G_4,H_4}(F)}^{\ast} D^{H_4}(t)\theta_{\pi_4}(t) \ud t.$$

Like in the previous case, we always choose the Haar measure so that the total volume is equal to 1 and the extra $\frac{1}{2}$ factor comes from the cardinality of the Weyl group of $\GU_2$. Also the integrals in the geometric multiplicity may not be absolutely convergent and they need to be regularized (see Remark \ref{local convergence}). We leave it as an excise for the reader to check that our definition of $m_{geom}(\pi)$ matches the definition in \cite{Wan} for general spherical varieties. Like the previous case, by a similar but easier argument as in the Gan--Gross--Prasad model case (\cite{W10}, \cite{W12},  \cite{B15}) and the Ginzburg--Rallis model case (\cite{Wan15}, \cite{Wan16}, \cite{Wan17}, \cite{WZ}), we can prove the following theorem.

\begin{thm}
For all tempered representations $\pi$ of $G(F)$ (resp. $\pi_i$ of $G_i(F)$) whose central character is trivial on $Z_{G,H}(F)$ (resp. $Z_{G_i,H_i}(F)$), we have
$$m(\pi)=m_{geom}(\pi),\;m(\pi_i)=m_{geom}(\pi_i).$$
\end{thm}

Now let $\Pi_{\phi}=\Pi_{\phi}(G)\cup \Pi_{\phi}(G_i)$ be a tempered local $L$-packet whose central character is trivial on $Z_{G,H}(F)$ ($1\leq i\leq 3$ in the $p$-adic case and $1\leq i\leq 4$ in the archimedean case). We can also define the character $\theta_{\Pi_\phi(G)}$ and $\theta_{\Pi_\phi(G_i)}$ as before (note that the component group is always abelian in this case). The summation $\sum_{\pi\in \Pi_{\phi}(G)} m(\pi)+\sum_{1\leq i\leq k,\pi_i\in \Pi_{\phi}(G_i)}m(\pi_i)$ is equal to
$$m_{geom}(\theta_{\Pi_\phi(G)})+\sum_{i=1}^{k}m_{geom}(\theta_{\Pi_\phi(G_i)})=c_{\theta_{\Pi_{\phi}(G)}}(1)=1$$
where $k=3$ in the $p$-adic case and $k=4$ in the archimedean case. Here the last equality follows from the results for Whittaker model in \cite{Mat}, \cite{MW} and the fact that there is a unique generic element in the $L$-packet. For the identity
$$m_{geom}(\theta_{\Pi_\phi(G)})+\sum_{i=1}^{k}m_{geom}(\theta_{\Pi_\phi(G_i)})=c_{\theta_{\Pi_{\phi}(G)}}(1),$$
we just need to apply the following cancellations (the Kottwitz sign between $G$ and $G_3$ is equal to 1, the Kottwitz sign between $G$ and $G_i$ is equal to -1 for $i=1,2,4$)
\begin{itemize}
\item The term $\sum_{T\in \CT_{ell}(H)}$ in $m_{geom}(\theta_{\Pi_\phi(G)})$ plus the term $\sum_{T\in \CT_{ell}(H_3)}$ in $m_{geom}(\theta_{\Pi_\phi(G_3)})$ can be cancelled with the term $\sum_{T\in \CT_{ell}(H_1)}$ in $m_{geom}(\theta_{\Pi_\phi(G_1)})$ plus the term  $\sum_{T\in \CT_{ell}(H_2)}$ in $m_{geom}(\theta_{\Pi_\phi(G_2)})$ (and also plus the term $\sum_{T\in \CT_{ell}(H_4)}$ in $m_{geom}(\theta_{\Pi_\phi(G_4)})$ if we are in the archimedean case).
\item The term $\frac{1}{2}\sum_{T\in \CT_{ell}(\GU_{1,1})}$ in $m_{geom}(\theta_{\Pi_\phi(G)})$ can be cancelled with the term $\frac{1}{2}\sum_{T\in \CT_{ell}(\GU_{2,0})}$ in $m_{geom}(\theta_{\Pi_\phi(G_1)})$.
\item The term associated to $T_{0}'$ in $m_{geom}(\theta_{\Pi_\phi(G)})$ can be cancelled with  the term associated to $T_{0}'$ in $m_{geom}(\theta_{\Pi_\phi(G_2)})$. 
\item The terms associated to $T_{0}''$ in $m_{geom}(\theta_{\Pi_\phi(G)})$ and $m_{geom}(\theta_{\Pi_\phi(G_3)})$ can be cancelled with  the terms associated to $T_{0}''$ in $m_{geom}(\theta_{\Pi_\phi(G_1)})$ and $m_{geom}(\theta_{\Pi_\phi(G_2)})$. 
\end{itemize}

In particular, we have proved that the summation of the multiplicities is equal to 1 over every tempered local Vogan $L$-packet.

\subsection{The non-reductive case}\label{sec:non-reductive}
In this subsection we consider the non-reductive cases. Let $(G,H)=(G,H_0\ltimes U)$ be one of the non-reductive models in Table \ref{fig:1}. For all the cases, $H_0(F)$ is essentially $\GL_2(F)$ (up to the center). If $F\neq \BC$, we let $(G_D,H_{0,D}\ltimes U_D)$ be the quaternion version of the model. 

Let $\CT_{ell}(H_0)$ (resp. $\CT_{ell}(H_{0,D})$) be a set of representatives of maximal elliptic tori of $H_0(F)$ (resp. $H_{0,D}(F)$). Define
$$m_{geom}(\pi)=c_\pi(1)+\sum_{T\in \CT_{ell}(H_0)}|W(H_0,T)|^{-1}\int_{T(F)/Z_{G,H}(F)} D^{H}(t) c_{\pi}(t)\ud t,$$ 
\begin{eqnarray*}
m_{geom}(\pi_D)&=&\sum_{T_D\in \CT_{ell}(H_{0,D})}|W(H_{0,D},T_D)|^{-1}\\
&&\cdot \int_{T_D(F)/Z_{G_D,H_D}(F)} D^{H_D}(t)c_{\pi_D}(t) \ud t
\end{eqnarray*}
where $\pi$ (resp. $\pi_D$) is an  irreducible admissible representation of $G(F)$ (resp. $G_D(F)$) with trivial central character, $W(H_0,T)$ (resp. $W(H_{0,D},T_D)$ is the Weyl group,
and all the Haar measure are chosen so that the total volume is equal to 1. Again we leave it as an excise for the reader to check that our definition of $m_{geom}(\pi)$ matches the definition in \cite{Wan} for general spherical varieties.

\begin{thm}
Assume that $F\neq \BR$, and $(G,H)$ is not the last model $(E_7,\PGL_2\ltimes U)$ in Table \ref{fig:1}. For all tempered representations $\pi$ of $G(F)$ (resp. $\pi_D$ of $G_D(F)$) whose central character is trivial on $Z_{G,H}(F)$ (resp. $Z_{G_D,H_D}(F)$), we have
$$m(\pi)=m_{geom}(\pi),\;m(\pi_D)=m_{geom}(\pi_D).$$
\end{thm}

\begin{proof}
The multiplicity formula for Model 4 in Table \ref{fig:1} has been proved in the previous papers of the first author (\cite{Wan15}, \cite{Wan16}, \cite{Wan17}), and the multiplicity formula for the Model 5 has been proved in our previous paper \cite{WZ}. The argument for the remaining 4 models is very similar to the Ginzburg--Rallis model case (\cite{Wan15}, \cite{Wan16}, \cite{Wan17}), we will skip it here. Like the reductive case,  the Gelfand pair condition is not known for these models, but it can be solved by the same argument as the unitary Ginzburg-–Rallis model case in our previous paper (Section 6 and Appendix A of \cite{WZ}). 

The reason we need to assume that $F\neq \BR$ is that in the case when $F=\BR$, we don't know how to prove the nonvanishing property of certain explicit intertwining operator is invariant under the parabolic induction because the operator is defined by a normalized integral in the non-reductive case and it is not clear how to study it under the parabolic induction in the real case. In Gan--Gross--Prasad case (Section 7.4 of \cite{B15}), this can be solved by passing to a reductive model of a larger group (e.g. instead of studying  $(U_{n+2k+1}\times U_n,U_n\ltimes N)$ one can just study $(U_{n+2k+1},U_{n+2k})$). But for all the cases in Table \ref{fig:1}, we cannot pass it to a reductive model of a larger group simply because such a module does not exist. For Model 4 in Table \ref{fig:1}, we solved this issue by using a special property that all the tempered representations of $\GL_6(\BR)$ are the parabolic induction of some tempered representations of $\GL_2(\BR)\times \GL_2(\BR)\times \GL_2(\BR)$, see Section 5.4 of \cite{Wan16} for details. But this is not true for Models 5--10 of Table \ref{fig:1} (although it is still true in the complex case which is why we can prove the multiplicity formula in the complex case). In general if one can prove that the nonvanishing property of the explicit intertwining operator is invariant under the parabolic induction, then we can also prove the multiplicity formula in the real case.

On the other hand, the reason we exclude the  model $(E_7,\PGL_2\ltimes U)$ is that in the proof of the geometric side of the trace formula, we need to study the slice representation, i.e. the conjugation action of $H(F)$ on the tangent space. We need to show that the regular orbits coincide with the stable conjugacy classes of $G(F)$. For all the other cases, this can be down by computing the characteristic polynomials as in the Gan--Gross--Prasad model case (Section 9 of \cite{W10} and Section 10 of \cite{B15}) and the Ginzburg--Rallis model case (Section 8 of \cite{Wan15}). But this is not possible for the $E_7$ case since the matrix presentation of $E_7$ is very complicated. If one can prove this result for the model $(E_7,\PGL_2\ltimes U)$, then we can also prove the multiplicity formula in this case.

\end{proof}

As in the reductive cases, combining the multiplicity formulas and the local Langlands conjecture, we can show that for any tempered L-packet $\Pi_\phi=\Pi_\phi(G)\cup \Pi_\phi(G_D)$ of $G(F)$ whose central character is trivial on $Z_{G,H}(F)$, the summation
$$\sum_{\pi\in \Pi_{\phi}(G)} \dim(\chi_\pi)m(\pi)+\sum_{\pi_D\in \Pi_{\phi}(G_D)}\dim(\chi_{\pi_D})m(\pi_D)$$ 
is equal to
$$m_{geom}(\theta_{\Pi_\phi(G)})+m_{geom}(\theta_{\Pi_\phi(G_D)})=c_{\theta_{\Pi_{\phi}(G)}}(1)=1.$$
In other words, the summation of the multiplicities is equal to 1 over every tempered local Vogan $L$-packet and the unique distinguished element corresponds to a character of the component group.


\begin{thebibliography}{cw}



\bibitem[B15]{B15}
Rapha{\"e}l {Beuzart-Plessis}, \emph{{A local trace formula for the
  Gan--Gross--Prasad conjecture for unitary groups: the archimedean case}}, Ast\'{e}risque  no.~418.

\bibitem[B02]{B02}
N. Bourbaki, 
\emph{Lie groups and Lie algebras} (Springer, Berlin, 2002), ch. 4--6.

\bibitem[BP]{BP}
P. Bravi, G. Pezzini, 
{\it The spherical systems of the wonderful reductive subgroups,} J. Lie Theory 25 (2015), 105-123,

\bibitem[C80]{C80}
W. Casselman, 
{\it The unramified principal series of $\Fp$-adic groups. I. The spherical function.}
Compositio Math. 40 (1980), no. 3, 387–-406. 





\bibitem[DG]{DG}
P. Deligne, B. Gross, 
{\it On the exceptional series, and its descendants.} C. R. Math. Acad. Sci. Paris 335 (2002), no. 11, 877–-881.


\bibitem[Gi]{Gi}
D. Ginzburg, 
{\it On spin L-functions for orthogonal groups.} Duke Math. J. 77 (1995), no. 3, 753-798. 

\bibitem[GR]{GR}
D. Ginzburg, S. Rallis,
{\it The exterior cube L-function for GL(6).} Compositio Math. 123(2000), no. 3, 243--272

\bibitem[GZ]{GZ}
R. Gomez, C. Zhu, 
{\it Local theta lifting of generalized Whittaker models associated to nilpotent orbits.} Geom. Funct. Anal. 24 (2014), no. 3, 796--853. 

\bibitem[G]{G}
B. Gross,
{\it On the motive of a reductive group,} Invent.Math.130(2), 287–313.

\bibitem[GGP]{GGP}
W. Gan, B. Gross, D. Prasad, 
{\it Symplectic local root numbers, central critical L values, and restriction problems in the representation theory of classical groups.} 
Sur les conjectures de Gross et Prasad. I. Astérisque No. 346 (2012), 1–109. ISBN: 978-2-85629-348-5


\bibitem[GP1]{GP1}
B. Gross, D. Prasad,
{\it On the decomposition of a representation of $\SO_n$ when restricted to $\SO_{n-1}$.} Canad. J. Math. 44 (1992), no. 5, 974-1002.

\bibitem[GP2]{GP2}
B. Gross, D. Prasad,
{\it On irreducible representations of $SO_{2n+1} \times \SO_{2m}$}. Canad. J. Math. 46 (1994), no. 5, 930-950.


\bibitem[H]{H}
R. Neal Harris,
{\it The refined Gross–Prasad conjecture for unitary groups}, IMRN Volume 2014, Issue 2, 2014, Pages 303-389.






\bibitem[I00]{I00}
J.-I. Igusa, 
{\it An introduction to the theory of local zeta functions.} 
AMS/IP Studies in Advanced Mathematics, 14. American Mathematical Society, Providence, RI; International Press, Cambridge, MA, 2000. xii+232 pp. ISBN: 0-8218-2015-X

\bibitem[II]{II}
A. Ichino, T. Ikeda,
{\it On the periods of automorphic forms on special orthogonal groups and the Gross--Prasad conjecture.} Geometric and Functional Analysis 19 (2010), no. 5, 1378-1425.

\bibitem[Iwa66]{Iwa66}
N. Iwahori,
{\it Generalized Tits system (Bruhat decompostition) on p-adic semisimple groups. }
1966 Algebraic Groups and Discontinuous Subgroups (Proc. Sympos. Pure Math., Boulder, Colo., 1965) pp. 71–83 Amer. Math. Soc., Providence, R.I. 

\bibitem[K]{K}
T. Kaletha,
{\it The local Langlands conjectures for non-quasi-split groups.} Families of Automorphic Forms and the Trace Formula, Simons Symposia, Springer 2016, 217-257.

\bibitem[KMS03]{KMS03}
S. Kato,  A. Murase, and T. Sugano  
{\it Whittaker-Shintani functions for orthogonal groups.} 
Tohoku Math. J. (2) 55 (2003), no. 1, 1--64.









\bibitem[L]{L}
Y. Liu,
{\it Refined Gan--Gross--Prasad conjecture for Bessel periods}, Journal für die reine und angewandte Mathematik, 717 (2016) 133-194.


\bibitem[Lang]{Lang}
R. P. Langlands, 
{\it On the classification of irreducible representations of real
algebraic groups,} Representation theory and harmonic analysis on semisimple
Lie groups, Math. Surveys Monogr., vol. 31, Amer. Math. Soc., Providence, RI, 1989, pp. 101-170.


\bibitem[LM]{LM}
E. Lapid, Z. Mao,
{\it A conjecture on Whittaker–Fourier coefficients of cusp forms}, Journal of Number Theory 146, 448-505.

\bibitem[Lu]{LU}
D. Luna, 
{\it Variete spheriques de type A}, Publ. Math. Inst. Hautes \'{E}tudes Sci. 94 (2001)
161–226.


\bibitem[Mac]{Mac}
Macdonald, I. G. 
{\it Spherical functions on a group of p-adic type,} 
Publications of the Ramanujan Institute, No. 2. University of Madras, Centre for Advanced Study in Mathematics, Ramanujan Institute, Madras, 1971. vii+79 pp.

\bibitem[Mat]{Mat}
H. Matumoto, 
{$C^{-\infty}$-Whittaker vectors corresponding to a principal nilpotent orbit of a real reductive linear Lie group and wave front sets}, Compositio Math. 82, 189-244 (1992)


\bibitem[MW]{MW}
C.~M\oe{}glin, J.-L.~Waldspurger, \textit{Mod\`eles de Whittaker d\'eg\'en\'er\'es pour des groupes $p$-adiques}, Math. Z. 196 (1987), no. 3, 427-452


\bibitem[P20]{P20}
Pollack, Aaron  
{\it The Fourier expansion of modular forms on quaternionic exceptional groups.} 
Duke Math. J. 169 (2020), no. 7, 1209–1280.

\bibitem[PWZ18]{PWZ18}
Aaron Pollack, Chen Wan, and Micha\l{} Zydor, \emph{{A $\mathrm{G}_2$-period of
  a Fourier coefficient of an Eisenstein series on $\mathrm{E}_6$}}, Israel Journal of Mathematics 234 (2019), 229--279.


\bibitem[PWZ19]{PWZ19}
Aaron Pollack, Chen Wan, and Micha\l{} Zydor, \emph{{On the residue method for period integrals}}. Accepted by Duke Math. Journal.

\bibitem[Rod81]{Rod81}
F. Rodier,
{\it Mod\`ele de Whittaker et caract\`eres de repr\'esentations.} Noncommutative harmonic analysis, Lecture Notes in Mathematics, vol.466, eds J. Carmona, J. Dixmier and M. Vergne(Springer, Berlin, 1981), 151--171.

\bibitem[R97]{R97}
Rumelhart, Karl E. 
 {\it Minimal representations of exceptional p-adic groups.} 
 Represent. Theory 1 (1997), 133–181.





\bibitem[Sa]{Sa}
Y. Sakellaridis,
{\it Spherical functions on spherical varieties.} Amer. J. Math., 135(5):1291-1381, 2013.

\bibitem[Sa08]{Sa08}
Y. Sakellaridis,
{\it On the unramified spectrum of spherical varieties over p-adic fields.} Compositio Mathematica 144 (2008), no. 4, 978–1016. 

\bibitem[Sa12]{Sa12}
Y. Sakellaridis, \emph{Spherical varieties and integral representations of L-functions}, Algebra \& Number Theory (2012), no.~4, 611--667.

\bibitem[SV17]{SV}
Y. Sakellaridis, A. Venkatesh, \emph{Periods and harmonic analysis
  on spherical varieties}, Ast\'{e}risque (2017), no.~396, viii+360.
  \MR{3764130}
  
% \bibitem[SV00]{SV00}
% Tonny Springer and Ferdinand Veldkamp, \emph{Octonions, Jordan algebras, and exceptional groups}, Springer Monographs in Mathematics (2000).
 
\bibitem[SW]{SW}
Y. Sakellaridis, J. Wang,
{\it Local L-factors and geometric asymptotics for spherical varieties}. Arxiv 2009.03943

% \bibitem[Sp09]{Sp09}
% T. A. Springer, 
% {\it Linear algebraic groups.} 
% Modern Birkh\"auser Classics. Birkhäuser Boston, Inc., Boston, MA, 2009. xvi+334 pp
 
\bibitem[W03]{W03}
J.-L. Waldspurger,
{\it La formule de Plancherel pour les groupes p-adiques (d’apr\'es Harish-Chandra)}, J. Inst. Math. Jussieu 2 (2003), no. 2, 235-333


\bibitem[W10]{W10}
J.-L. Waldspurger,
{\it Une formule int\'egrale reli\'ee \`a la conjecture locale de Gross--Prasad.} Compos. Math. 146(2010), no.5, 1180-1290.

\bibitem[W12]{W12}
J.-L. Waldspurger,
{\it Une formule int\'egrale reli\'ee \`a la conjecture locale de Gross--Prasad, 2e partie : extension aux repr\'esentations temp\'er\'ees.}  in "Sur les conjectures de Gross et Prasad. I" Ast\'erisque No. 346 (2012), 171-312


\bibitem[Wan15]{Wan15}
C. Wan, \emph{{A local relative trace formula for the Ginzburg-Rallis model: the geometric side}}, Memoirs of the American Mathematical Society Volume 261, Number 1263.

\bibitem[Wan16]{Wan16}
C. Wan, \emph{{Multiplicity One Theorem for the Ginzburg-Rallis Model: the
  tempered case}}, Trans. Amer. Math. Soc. 371 (2019), 7949-7994.

\bibitem[Wan17]{Wan17}
C. Wan, \emph{The local {G}inzburg-{R}allis model over the complex field},
  Pacific J. Math. \textbf{291} (2017), no.~1, 241--256. \MR{3693584}


\bibitem[Wan17b]{Wan17b}
C. Wan, \emph{A Local Trace Formula and the Multiplicity One Theorem for the Ginzburg-Rallis Model},  PhD Thesis, 2017. Available at \url{https://sites.rutgers.edu/chen-wan/research/}


\bibitem[Wan]{Wan}
C. Wan,
{\it On a multiplicity formula for spherical varieties.}  J. Eur. Math. Soc. 24 (2022), no. 10, 3629-3678

\bibitem[Weil]{Weil}
Weil, Andr\'e 
{\it Adeles and algebraic groups.} With appendices by M. Demazure and Takashi Ono. Progress in Mathematics, 23. Birkh\"auser, Boston, Mass., 1982. 

\bibitem[WZ]{WZ}
C. Wan, L. Zhang,
{\it The Multiplicity Problems for the Unitary Ginzburg-Rallis Models.} Accepted by Israel Journal of Mathematics.

\bibitem[WZ1]{WZ1}
C. Wan, L. Zhang,
{\it Multiplicities for Strongly Tempered Spherical Varieties.} Preprint, 70 pages.




\bibitem[Z]{Z}
L. Zhang,
{\it The Exterior Cubic $L$-function of $\GU(6)$ and Unitary Automorphic Induction}, arXiv:1903.04322v1 

\end{thebibliography}
\end{document}